\documentclass[11pt, twoside, leqno]{article}

\usepackage{amssymb}
\usepackage{amsmath}
\usepackage{amsthm}
\usepackage{color}
\usepackage{mathrsfs}
\usepackage{txfonts}

\usepackage{indentfirst}

\allowdisplaybreaks

\def\XXint#1#2#3{{\setbox0=\hbox{$#1{#2#3}{\int}$ }
\vcenter{\hbox{$#2#3$ }}\kern-.6\wd0}}

\def\rr{{\mathbb R}}

\def\ss{{\mathbb S}}
\def\zz{{\mathbb Z}}
\def\ee{{\mathbb E}}
\def\nn{{\mathbb N}}

\def\cf{{\mathcal F}}

\def\cM{\mathcal{M}}

\def\cS{\mathcal{S}}
\def\cT{\mathcal{T}}

\def\F{{\mathbb F}}

\def\fz{\infty }

\def\lf{\left}
\def\r{\right}
\def\hs{\hspace{0.25cm}}
\def\ls{\lesssim}

\def\noz{\nonumber}

\def\dis{\displaystyle}

\def\MP{{\mathbb{P}}}

\def\V{\varphi}
\def\D{(0,\infty)}

\def\lkl{\Lambda_{k,\ell}}

\newtheorem{theorem}{Theorem}[section]
\newtheorem{lemma}[theorem]{Lemma}
\newtheorem{corollary}[theorem]{Corollary}

\newtheorem{example}[theorem]{Example}
\theoremstyle{definition}
\newtheorem{remark}[theorem]{Remark}

\newtheorem{definition}[theorem]{Definition}
\renewcommand{\appendix}{\par
   \setcounter{section}{0}%
   \setcounter{subsection}{0}%
   \setcounter{subsubsection}{0}%
   \gdef\thesection{\@Alph\c@section}%
   \gdef\thesubsection{\@Alph\c@section.\@arabic\c@subsection}%
   \gdef\theHsection{\@Alph\c@section.}%
   \gdef\theHsubsection{\@Alph\c@section.\@arabic\c@subsection}%
   \csname appendixmore\endcsname
 }

\numberwithin{equation}{section}

\begin{document}

\arraycolsep=1pt

\title{\bf\Large New Martingale Inequalities and Applications
to Fourier Analysis
\footnotetext{\hspace{-0.35cm} 2010 {\it
Mathematics Subject Classification}. Primary  60G42;
Secondary 60G46, 42B25, 42B35, 46E30.
\endgraf {\it Key words and phrases.} Probability space, Musielak--Orlicz space,
martingale Musielak--Orlicz Hardy space, quadratic variation, atom,
Doob maximal operator, Fej\'er operator, Burkholder--Davis--Gundy inequality, weight.
\endgraf This project is supported by the National
Natural Science Foundation of China
(Grant Nos. 11571039, 11726621, 11761131002, 11471337 and 11722114)
and the Hungarian Scientific Research Funds (OTKA) No K115804.}}
\author{Guangheng Xie, Ferenc Weisz, Dachun Yang\footnote{Corresponding
author/{\color{red}September 27, 2018}/Final version.}\ \ and Yong Jiao}
\date{}
\maketitle

\vspace{-0.6cm}

\begin{center}
\begin{minipage}{13cm}
{\small {\bf Abstract}\quad
Let $(\Omega,\mathcal{F},\mathbb{P})$ be a probability space and
$\varphi:\ \Omega\times[0,\infty)\to[0,\infty)$ be a Musielak--Orlicz function.
In this article,
the authors prove that the Doob maximal operator is bounded on the Musielak--Orlicz
space $L^{\varphi}(\Omega)$. Using this and extrapolation method, the authors then
establish a Fefferman--Stein vector-valued Doob
maximal inequality on $L^{\varphi}(\Omega)$.
As applications, the authors obtain the dual version of the Doob maximal inequality
and the Stein inequality for $L^{\varphi}(\Omega)$, which are new even in weighted Orlicz spaces.
The authors then establish the atomic characterizations of martingale
Musielak--Orlicz Hardy spaces $H_{\varphi}^s(\Omega)$, $P_{\varphi}(\Omega)$,
$Q_{\varphi}(\Omega)$, $H_{\varphi}^S(\Omega)$ and $H_{\varphi}^M(\Omega)$.
From these atomic characterizations, the authors further deduce some martingale inequalities
between different martingale Musielak--Orlicz Hardy spaces, which essentially improve
the corresponding results in Orlicz space case and are also new even in weighted Orlicz spaces.
By establishing the Davis decomposition on $H_{\varphi}^S(\Omega)$ and $H_{\varphi}^M(\Omega)$,
the authors obtain the Burkholder--Davis--Gundy inequality associated with Musielak--Orlicz functions.
Finally, using the previous martingale inequalities,
the authors prove that the maximal Fej\'er operator is bounded from
$H_{\varphi}[0,1)$ to $L^{\varphi}[0,1)$, which further implies some convergence results
of the Fej\'er means; these results are new even for the weighted Hardy spaces.
}
\end{minipage}
\end{center}

\vspace{0.2cm}

\section{Introduction\label{s1}}
As is well known, the martingale theory has been well developed since Doob \cite{D53}. Using a
maximal inequality, which is now called the Doob maximal inequality, Doob \cite{D53} proved the basic almost sure
convergence properties of the martingales. This topic have been studied particularly intensively by Burkholder.
As for probability theory, specially for martingale theory,
the references greatly influence us
are \cite{D53,Garsia1973,H17,Pisier16,W94} and the articles \cite{B73,BDG72,BG70}.
Notice that the martingale theory has an extensive application in dyadic harmonic analysis,
we refer the reader to the monographes \cite{Schipp1990,W94,wk2}.

Let $\zz_{+}:=\{0\}\cup \nn:=\{0,1,\ldots\}$, $(\Omega,\mathcal{F},\mathbb{P})$ be a probability space, $\{\mathcal{F}_n\}_{n\in\mathbb{Z}_{+}}$ an
increasing sequence of sub-$\sigma$-algebras of $\mathcal{F}$ and
$\{\mathbb{E}_n\}_{n\in\mathbb{Z}_{+}}$ the associated conditional expectations.
For any $n\in\zz_{+}$ and measurable function $f$, the \emph{Doob maximal operators} $M_n(f)$ and $M(f)$
are defined, respectively, by setting
\begin{equation}\label{dm}
M_n(f):=\sup_{0\le i\le n}|\ee_i(f)|\quad \mbox{and}\quad M(f):=\sup_{n\in\mathbb{Z}_{+}}|\ee_n(f)|.
\end{equation}
Let $p\in\D$ and $w$ be a weight, the \emph{weighted Lebesgue space $L^p(\Omega,\,w\,d\MP)$}
is defined to be the set of all measurable functions $f$ on $\Omega$ such that
$$\|f\|_{L^p(\Omega,\,w\,d\MP)}
:=\lf[\int_{\Omega}|f(x)|^pw(x)\,d\MP\r]^{\frac1p}<\fz.$$
In 1977, Izumisawa and Kazamaki \cite{IK77} obtained that $w\in A_{p_0}(\Omega)$
(see Section \ref{s2} below for its definition)
for some $p_0\in(1,\fz)$ implies that the Doob maximal operator $M$ is bounded on $L^p(\Omega,\,w\,d\MP)$
for any $p\in(p_0,\fz)$. After this, Long \cite[Theorem 6.6.3]{Long93} improved their result
by proving that $w\in A_{p}(\Omega)$ with $p\in(1,\fz)$ if and only if the Doob maximal operator $M$ is bounded on $L^p(\Omega,\,w\,d\MP)$.
Recently, the sharp weighted Doob maximal inequalities have been studied by Os\c{e}kowski \cite{o18tmj,o18}.

On another hand, the classical Fefferman--Stein vector-valued inequality
was proved by Fefferman and Stein in their celebrated paper \cite{fs71}.
Later, Andersen and John \cite[Theorem 3.1]{aj80} established the weighted version
of the vector-valued inequality for the Hardy--Littlewood maximal operator.
For the Doob maximal operator, Jiao et al. \cite[Theorem 6.1]{jwp13}
proved the Fefferman--Stein theorem in rearrangement invariant spaces,
Hyt\"onen et al. \cite[Theorem 3.2.7]{H17} gave a version in the Banach-valued setting.

The well-known Burkholder--Davis--Gundy inequality was proved by Burkholder et al.
in their excellent article \cite[Theorem 1.1]{BDG72}, and is read as follows:
Let $\Phi:\ [0,\fz)\to[0,\fz)$ be an Orlicz function. If $\Phi$ is convex and 
there exists a positive constant $C$ such that, for any $\lambda\in\D$, $\Phi(2\lambda)\le C \Phi(\lambda)$,
then, for any martingale $f:=(f_n)_{n\in\zz_{+}}$,
\begin{align}\label{bdg}
\int_{\Omega}\Phi(S(f))\,d\MP\sim \int_{\Omega}\Phi(M(f))\,d\MP,
\end{align}
where $S(f):=(\sum_{n=1}^{\infty}|f_n-f_{n-1}|^2)^{1/2}$ and the equivalent positive constants are
independent of $f$.
After their outstanding work, Bonami and L\'epingle \cite[Theorem 1]{bl79} proved the
weighted version of \eqref{bdg}. Then Johnson and Schechtman \cite[Theorem 3]{JS88}
extended \eqref{bdg} to the setting of rearrangement
invariant function spaces.

Moreover, more martingale inequalities were recently studied by Os\c{e}kowski \cite{o17ap,o17},
Kikuchi \cite{k16,k15,k15rm} and Ho \cite{h18,h17}. Especially, Ba\~{n}uelos and Os\c{e}kowski studied the weighted
martingale inequalities in \cite{bo18pams,bo18}. We refer the reader to recent monographs \cite{H17,o12,Pisier16} for more discussions on martingale inequalities.
On another hand, various martingale Hardy spaces were considered in many articles,
for instance, martingale Hardy spaces, martingale Lorentz Hardy spaces and
martingale variable Hardy spaces
were studied by Weisz \cite{Weisz ptrf,W94,W16} and
Jiao et al. \cite{lzj,wyong1,jzwh17,JWYY17,jzhc16}.
Moreover, martingale Lorentz--Karamata Hardy spaces and
multi-parameter martingale Hardy
spaces were investigated by Ho \cite{h14}, Jiao et al. \cite{jxz15} and Weisz \cite{W16}.
Variable martingale Hardy spaces, martingale Morrey Hardy spaces,
martingale BLO spaces, martingale Besov spaces and Tribel--Lizorkin spaces
were studied by Nakai et al. \cite{ns13,nss13,ns17,mns12,S}.
Martingale Musielak--Orlicz Hardy spaces were dealt with by Xie at al. in \cite{xjy18}
which mainly concern Musielak--Orlicz Hardy spaces determined by Musielak--Orlicz functions of uniformly upper type $p=1$ only (see Definition \ref{d3} below for its definition).

The theory of martingales has an extensive application in dyadic Fourier analysis; see,
for example, the monographes by Schipp et al. \cite{Schipp1990} and Golubov et al. \cite{Golubov1991}. In particular, Walsh--Fourier series were investigated for example in the monographes \cite{Schipp1990,Golubov1991}
and the article \cite{wk2}. Besides the partial sums $(s_kf)_{k\in\nn}$ (see
Section \ref{s6} below for its definition) of the Walsh--Fourier series of a martingale $f$, many articles
(see, for example, \cite{wyong1,wces3}) also considered the {\it Fej{\'e}r means} defined by setting
\begin{equation}\label{e9}
\sigma_{n}f := \frac{1}{n} \sum_{k=1}^{n} s_{k} f, \qquad \forall\,n\in \mathbb N.
\end{equation}
It is known that, to obtain some convergence results for the Fej{\'e}r means,
one needs to investigate the \emph{maximal Fej{\'e}r operator} $\sigma_*$ defined by setting,
for any martingale $f$,
$$
\sigma_*f :=\sup_{n\in\mathbb N} |\sigma_nf| .
$$
In \cite{wces3,wk2}, using dyadic martingale theory, Weisz proved that the maximal operator $\sigma_*$ is bounded from $H_{p}[0,1)$ to $L^{p}[0,1)$ for any given $p\in(\frac12,\infty)$.
Very recently, Jiao et al. \cite[Theorem 7.15]{wyong1} investigated the boundedness of $\sigma_*$ from variable Hardy spaces to variable Lebesgue spaces.

Recall that a function $\varphi:\,\Omega\times[0,\infty)\to[0,\infty)$
is called a \emph{Musielak--Orlicz} \emph{function} if for any $x\in\Omega$, the function $\varphi(x,\cdot):\ [0,\infty)\to[0,\infty)$ is an Orlicz function, namely, $\varphi(x,\cdot)$ is non-decreasing, $\varphi(x,0)=0$ and $\lim_{t\to\infty}\varphi(x,t)=\infty$,
and the function $\varphi(\cdot,t)$ is a measurable function for any given $t\in[0,\infty)$.
The \emph{Musielak--Orlicz space $L^{\varphi}(\Omega)$} is defined to be the set of all measurable
functions $f$ with finite \emph{Luxemburg} (also called as the \emph{Luxemburg--Nakano})
\emph{norms} $\|f\|_{L^{\varphi}(\Omega)}$ defined by setting
$$\|f\|_{L^{\varphi}(\Omega)}:=\inf\left\{\lambda\in(0,\infty):\ \int_{\Omega}\varphi(x,|f(x)|/\lambda)\,d\mathbb{P}(x)\le1\right\}.$$
It was originated by Nakano \cite{n50} and developed by Musielak and Orlicz
\cite{m83,mo59}.

Observe that Musielak--Orlicz spaces are the natural generalization of many important
spaces. For example, let $p\in(0,\fz)$, $w$ be a weight, $\Phi$ an Orlicz function
and $p(\cdot):\ \rr^d\to[1,\fz]$ a measurable function. If $\V(x,t):=t^p$, $w(x)t^p$, $\Phi(t)$ or $t^{p(x)}$ for any $x\in\rr^d$ and $t\in\D$, then $L^{\V}(\rr^d)$
become the classical Lebesgue spaces, the weighted Lebesgue spaces,
Orlicz spaces (see, for example, \cite{rr91}) or the variable Lebesgue spaces (see, for example, \cite{cf13}), respectively.
Moreover, if $\V(x,t):=t^p+w(x)t^q$ for any $x\in\rr^d$ and $t\in\D$, $1<p\le q <\fz$,
we obtain the double phase functionals (see, for example,
\cite{bcg15na,cm15arma}). For more examples, see Example \ref{x1} below.
The Musielak--Orlicz spaces
not only have their own interest, but they are also very useful in partial differential equations
\cite{ap16,bcg15na,hhk16,GSZ}, in calculus of variations \cite{cm15arma}, in
image restoration \cite{hhlt13,ana14} and in fluid dynamics \cite{sg14,mw}.

Furthermore, Musielak--Orlicz Hardy spaces on $\rr^d$ are fruitful in dealing with
many problems of analysis; see, for example,
\cite{CCYY16tams,LY15,JY10,Yang2017}. Especially, they naturally appear in
the endpoint estimates for the div-curl lemma and the commutators of
Caldr\'{o}n--Zygmund operators (see \cite{bgk12,bijz07,Yang2017}). See also the monograph \cite{Yang2017} for a detailed and complete
survey of the recent progress related to the real-variable theory and its applications of
Musielak--Orlicz Hardy spaces.
Although the theory of Musielak--Orlicz Hardy spaces has rapidly been developed in recent years,
the corresponding martingale theory associated with Musielak--Orlicz functions has not yet been developed well.

In this article, we first obtain the boundedness of the Doob maximal operator
on $L^{\V}(\Omega)$.
Using this and extrapolation method, we prove the weighted
Fefferman--Stein vector-valued Doob maximal inequality and
a Fefferman--Stein vector-valued inequality on $L^{\V}(\Omega)$.
As applications, we obtain the dual version of the Doob maximal inequality,
the weak type maximal inequality and the Stein inequality on $L^{\varphi}(\Omega)$,
which are new even in weighed Orlicz spaces.
Then we establish the atomic characterizations of martingale Musielak--Orlicz Hardy spaces
$H_{\varphi}^s(\Omega)$, $P_{\varphi}(\Omega)$,
$Q_{\varphi}(\Omega)$, $H_{\varphi}^S(\Omega)$ and $H_{\varphi}^M(\Omega)$,
which is totally different from the classical martingale Hardy spaces (see, for example, \cite{Weisz ptrf}) and
the classical martingale Orlicz Hardy spaces (see, for example, \cite{mns12}). Using these atomic
characterizations, we explore the relationship among these martingale Musielak--Orlicz Hardy spaces.
Our theorems improve the Orlicz case \cite[Theorem 2.5 and Corollary 2.6]{mns12} of Miyamoto et al.
Moreover, our martingale inequalities partially improve Kazamaki \cite[Theorem 1]{K79}.
Since Musielak--Orlicz functions unify the weight and the Orlicz function,
and Musielak--Orlicz spaces are not
rearrangement invariant, it is natural to ask whether or not \eqref{bdg}
still holds true on $L^{\V}(\Omega)$.
We give an affirmative answer to this question.
Indeed, by establishing the Davis decomposition of $H_{\varphi}^S(\Omega)$ and $H_{\varphi}^M(\Omega)$,
we obtain the Burkholder--Davis--Gundy inequality on $L^{\V}(\Omega)$.
Finally, using the previous martingale inequalities,
some applications in Fourier analysis are presented in this article.
We prove that the maximal Fej\'er operator is bounded from
$H_{\varphi}[0,1)$ to $L^{\varphi}[0,1)$. As a consequence,
we obtain several convergence results on both the partial sums and
the Fej\'er means of the Walsh--Fourier series. In particular,
both the boundedness of the maximal Fej\'er operator and the convergence
results are new even for the weighted Hardy as well as for the (weighted) Orlicz Hardy spaces.

To be precise, this article is organized as follows.

In Section \ref{s2}, we first recall some notation and notions on martingale theory,
Musielak--Orlicz functions and weights. Then we give some properties and examples
of Musielak--Orlicz functions.

In Section \ref{s3}, with the help of \cite[Theorem 2.7]{lhy12},
we first prove that the Doob maximal operator $M$ is bounded on $L^{\V}(\Omega)$; see Theorem \ref{thm-doob} below.
Via this, we prove the weak type inequality and
the dual version of the Doob maximal inequality; see Theorems \ref{thm-weak} and \ref{thm-duald} below. Then,
using the extrapolation theorem (see Theorem \ref{thm-chjz} below,
which is a variant of \cite[Theorem 3.9]{cmp13} on probability spaces),
we obtain the weighted Fefferman--Stein inequality for the Doob maximal operator; see Theorem \ref{thm-fs} below.
By establishing a vector-valued version (see Theorem \ref{thm-v1} below)
of the extrapolation theorem (see Theorem \ref{thm-chjz}), we also obtain the
Fefferman--Stein vector-valued inequality on $L^{\V}(\Omega)$ (see Theorem \ref{thm-fsv} below).
The extrapolation theorems in $L^{\V}(\mathbb R^d)$ were proved by Cruz--Uribe and H\"ast\"o \cite{ch18}.
However, the extrapolation theorems in \cite{ch18} need to use the fact that the maximal operator is bounded
on the dual spaces. It is hard to obtain an explicit expression of the dual
space of $L^{\V}(\Omega)$ and hence the boundedness of the maximal operator on it is difficult to obtain.
However, the dual space of the weighted Lebesgue space $L^p(\Omega,\,w\,d\MP)$
is already known (see, for example, \cite[Theorem 3.9]{cmp13}), which enables us to
obtain the extrapolation theorem for the
weighted Lebesgue space $L^p(\Omega,\,w\,d\MP)$ (see Theorem \ref{thm-v1}). Using both this
and an interpolation theorem of sublinear operator on $L^{\V}(\Omega)$ (see Theorem \ref{thm-ip} below), we prove Theorem \ref{thm-fsv}. Remarkably, our method skillfully avoids the requirement
that the Doob maximal operator is bounded on the dual space of $L^{\V}(\Omega)$.
It should be mentioned that Theorem \ref{thm-fsv} is a probabilistic version
of \cite[Corollary 6.1]{ch18} and Theorem \ref{thm-fsv} covers the unknown weighted Orlicz case; see Remark \ref{rmk}. Via Theorem \ref{thm-fsv}, we totally cover the Stein inequality \cite[Theorem 3.8]{S70},
which is further generalized to the Musielak--Orlicz case; see Theorem \ref{thm-ste} below.

The target of Section \ref{s4} is to establish the atomic characterizations of
five Musielak--Orlicz martingale Hardy spaces,
$H_{\V}^M(\Omega)$, $P_{\V}(\Omega)$, $H_{\V}^s(\Omega)$,
$H_{\V}^S(\Omega)$ and $Q_{\V}(\Omega)$; see Theorems \ref{Thm-atms},
\ref{thm-atmpq} and \ref{Thm-atmSM} below.
The above five martingale Hardy spaces include weighted martingale Hardy spaces,
martingale Orlicz--Hardy spaces in \cite{mns12}, weighted martingale Orlicz--Hardy spaces,
and variable martingale Hardy spaces in \cite{jzhc16,wyong1} as special cases
(see also Remark \ref{rmk-s} below for more details).
Unlike the classical case \cite[Theorem 1.3.17]{Yang2017},
we introduce a new type of atoms which allow us to eliminate the
inaccuracy of the growth properties of $\V$, which is totally different
from the articles \cite{Weisz ptrf,h14,jxz15,W16,mns12,xjy18}. This is a
key idea to improve martingale inequalities \cite[Theorem 1.9]{xjy18}.
Moreover, our atomic characterizations of Musielak--Orlicz martingale Hardy spaces
totally cover the variable martingale Hardy spaces, the weighted martingale Hardy spaces and the weighted martingale Orlicz--Hardy spaces; see Remark \ref{rmk1} below.
In particular, the atomic characterizations of $H_{\V}^M(\Omega)$
and $H_{\V}^S(\Omega)$ are new even for martingale Hardy spaces.
It should be point out that the classical argument used in the
proof of \cite[Theorem 2.2]{W94} does not work for $H_{\V}^M(\Omega)$
and $H_{\V}^S(\Omega)$. To overcome this difficulty, we construct appropriate stopping times
under the regularity condition (see Lemma \ref{st} below),
which is another key idea of this article.

Section \ref{s5} is devoted to proving the Burkholder--Davis--Gundy inequality
and to improving \cite[Theorem 1.9]{xjy18}; see Theorems \ref{thm-mi}, \ref{t21} and \ref{t210} below.
To be precise, via atomic characterizations, we first investigate some
$\sigma$-sublinear operators defined on weighted martingale Hardy spaces;
see Theorems \ref{t20} and \ref{t200} below.
Three important examples of such operators
are $S$, $s$ and $M$, respectively, as in \eqref{ds}, \eqref{xs} and \eqref{dm}. Then we establish the
relationships among five Musielak--Orlicz martingale Hardy spaces
$H_{\V}^M(\Omega)$, $P_{\V}(\Omega)$, $H_{\V}^s(\Omega)$,
$H_{\V}^S(\Omega)$ and $Q_{\V}(\Omega)$; see Theorem \ref{thm-mi} below.
It is noteworthy that Theorem \ref{thm-mi} totally
improves Miyamoto et al. \cite[Theorem 2.5 and Corollary 2.6]{mns12}
(see Remark \ref{rmk2} below). Then we prove the Burkholder--Davis--Gundy
inequality in $L^{\V}(\Omega)$ (Theorems \ref{t21} and \ref{t210}).
Remarkably, an important tool in
the proof of the Burkholder--Davis--Gundy inequality
(see \cite[Theorem 1]{bl79} and \cite[Theorem 3]{JS88}) is the extensions of
the good-$\lambda$ inequalities, which invented
by Burkholder and Gundy \cite{BG70} (see also \cite[Lemma 7.1]{B73}).
However, good-$\lambda$ inequalities do not work anymore in the present setting.
The reason behind this is
that the space variant $x$ and the growth variant $t$
appeared in the considered Musielak--Orlicz function $\V(x,t)$ are inseparable.
Using the dual version of the Doob maximal inequality and
the Davis decomposition of martingale Musielak--Orlicz Hardy spaces,
we give the proof of Theorems \ref{t21} and \ref{t210}.
Our method is different from the classical proofs of
\cite[Theorem 15.1]{B73} and \cite[Theorem 1]{bl79}
(see also \cite[Theorem 6.6.9]{Long93}), because we did not use the
good-$\lambda$ inequality. Despite the fact that
our Burkholder--Davis--Gundy inequality covers several cases (see Example \ref{x1}),
the assumptions are not stronger than that of Bonami and L\'epingle
\cite[Theorem 1]{bl79}; see Remark \ref{rbdg} below.
Using the Doob maximal inequality and the Burkholder--Gundy inequality,
we then prove that the martingale
transform operator is bounded on $L^{\V}(\Omega)$ (see Theorem \ref{thm-tran} below).

We point out, under the condition that $\V$ is of uniformly lower type $p_{\V}^{-}\in(0,1)$ and
upper type $p_{\V}^{+}=1$ (see Definition \ref{d3} below for their definitions),
Xie et al. \cite[Theorems 1.4, 1.5 and 1.9]{xjy18} established the atomic characterizations of
martingale Musielak--Orlicz Hardy spaces $H_{\V}^s(\Omega)$, $P_{\V}(\Omega)$ and $Q_{\V}(\Omega)$,
and further explored the relationships among five martingale Musielak--Orlicz Hardy spaces
$H_{\V}^M(\Omega)$, $P_{\V}(\Omega)$, $H_{\V}^s(\Omega)$,
$H_{\V}^S(\Omega)$ and $Q_{\V}(\Omega)$. In this article, via introducing a new kind of
atoms (see Definition \ref{def atom} below), we then remove the
above restriction in \cite{xjy18} that
$\V$ is of uniformly lower type $p_{\V}^{-}\in(0,1)$ and upper type $p_{\V}^{+}=1$
(see Theorems \ref{Thm-atms} and \ref{thm-atmpq} below).
Moreover, differently from \cite{xjy18}, we also establish the atomic characterizations
of $H_{\V}^S(\Omega)$ and $H_{\V}^M(\Omega)$ (see Theorem \ref{Thm-atmSM} below).
These improved (or new) atomic characterizations further induce
the corresponding improvement on the relationships among
these martingale Musielak--Orlicz Hardy spaces.
Here, we also allow that the considered Musielak--Orlicz function
is of the uniformly lower type index $p_{\V}^{-}\in (0,\infty)$ and the uniformly
upper type index $p_{\V}^{+}\in (0,\infty)$,
which cause some extra difficulties, because now these martingale Musielak--Orlicz Hardy spaces
have wider generality than those in \cite{xjy18}, which cover,
for example, all weighted martingale Hardy spaces
and weighted martingale Orlicz--Hardy spaces, while those martingale Musielak--Orlicz
Hardy spaces in \cite{xjy18} cover only part of them. We point out that both the new atomic
characterizations of these martingale Musielak--Orlicz Hardy spaces and the technical lemma
on appropriate stopping times (Lemma \ref{st}) established in
Section \ref{s4} play an essential role in overcoming these extra difficulties.

In Section \ref{s6}, we introduce the Walsh system
and the Fej\'er means. We then prove that the partial sum of the Walsh--Fourier series
is uniformly bounded in $L^{\V}[0,1)$ (see Theorem \ref{t7} below). Moreover, we
show that the Walsh--Fourier series
is converges in the $L^{\V}[0,1)$-norm (see Corollary \ref{c1} below).

Finally, in Section \ref{s7}, we prove that the maximal Fej\'er
operator is bounded from the Musielak--Orlicz Hardy space $H_{\V}[0,1)$ to $L^{\V}[0,1)$
(see Theorem \ref{t5} below).
Theorem \ref{t5} is new even for the weighted Hardy spaces as well as for (weighted)
Orlicz Hardy spaces (see Theorems \ref{t8} and \ref{t10}).
Moreover, we also obtain the consequences
about the convergence of the Fej\'er means $(\sigma_nf)_{n\in\nn}$ of a martingale $f$
(see Corollary \ref{c22} below).

Now, we fix some conventions on notation used in this article. Throughout the article,
we always let $\mathbb{N}:=\{1,2,\ldots\}$, $\mathbb{Z}_{+}:=\mathbb{N}\cup \{0\}$
and $C$ denote a positive constant, which may vary from line to line.
For any $p\in\D$, we denote by $p'$ the \emph{conjugate exponent} to $p$,
namely, $1/p+1/p'=1$. We use the symbol $f\lesssim g$ to denote that
there exists a positive constant $C$
such that $f\le Cg.$ The symbol $f\sim g$ is used as an abbreviation
of $f\lesssim g\lesssim f$. We also use the following
convention: If $f\le Cg$ and $g=h$ or $g\le h$, we then write $f\ls g\sim h$
or $f\ls g\ls h$, \emph{rather than} $f\ls g=h$
or $f\ls g\le h$. For any subset $E$ of $\Omega$, we use
$\mathbf{1}_E$ to denote its \emph{characteristic function}.

\section{Preliminaries\label{s2}}

This section includes some basic notions and lemmas used in later sections.

Denote by $\mathcal{M}$
the set of all martingales $f:=(f_n)_{n\in\mathbb{Z}_{+}}$ related to
$\{\mathcal{F}_n\}_{n\in\mathbb{Z}_{+}}$ such that $f_0=0.$
Let $\mathcal{T}$ be the set of all stopping times related to
$\{\mathcal{F}_n\}_{n\in\mathbb{Z}_{+}}$.
For any $f\in\mathcal{M}$ and
$\nu\in\mathcal{T}$, we write $f^{\nu}:=\{f_{\nu\wedge n}\}_{n\in\mathbb{Z}_{+}}$
to denote the \emph{stopped martingale} and $B_{\nu}:=\{x\in\Omega:\ \nu(x)<\fz\}$.

Now we recall the definition of the martingale Musielak--Orlicz Hardy spaces.
For any $f\in\mathcal{M}$, denote its \emph{martingale difference} by
$$
d_nf:=f_n-f_{n-1}, \qquad \forall\,n\in\mathbb{N}.
$$
Then the \emph{quadratic variations}
$S_n(f)$ and $S(f)$, and the \emph{conditional quadratic
variations} $s_n(f)$ and $s(f)$ of a martingale $f$ are defined, respectively, by setting,
for any $n\in\nn$,
\begin{equation}\label{ds}
S_n(f):=\left(\sum_{i=1}^n\left|d_if\right|^2\right)^{\frac{1}{2}},\quad S(f):=\left(\sum_{i=1}^{\infty}\left|d_if\right|^2\right)^{\frac{1}{2}},
\end{equation}
\begin{equation}\label{xs}
s_n(f):=\left(\sum_{i=1}^n\mathbb{E}_{i-1}\left|d_if\right|^2\right)^{\frac{1}{2}}
\quad \mathrm{and}\quad
s(f):=\left(\sum_{i=1}^{\infty}\mathbb{E}_{i-1}\left|d_if\right|^2\right)^{\frac{1}{2}}.
\end{equation}
Let $\Lambda$ be the collection of all nondecreasing, nonnegative and adapted sequences $(\lambda_n)_{n\in\mathbb{Z}_{+}}$ of functions. Recall that a sequence $(\lambda_n)_{n\in\mathbb{Z}_{+}}$ of functions is said to be \emph{adapted} if, for any $n\in\mathbb{Z}_{+}$, $\lambda_n$ is $\mathcal{F}_n$ measurable. Let $\lambda_{\infty}:=\lim_{n\to\infty}\lambda_n.$
For any $f\in\mathcal{M},$ let
\begin{align*}
\Lambda[P_{\varphi}](f):=\lf\{(\lambda_n)_{n\in\mathbb{Z}_{+}}\in\Lambda:\ |f_n|\le\lambda_{n-1}\ \ \mbox{for any }n\in\mathbb{N},\ \ \lambda_{\infty}\in L^{\varphi}(\Omega)\r\}
\end{align*}
and
\begin{align*}
\Lambda[Q_{\varphi}](f):=\lf\{(\lambda_n)_{n\in\mathbb{Z}_{+}}\in\Lambda:\ S_n(f)\le\lambda_{n-1}\ \ \mbox{for any }n\in\mathbb{N},\ \ \lambda_{\infty}\in L^{\varphi}(\Omega)\r\}.
\end{align*}

\begin{definition}\label{def-spaces}
Let $\varphi$ be a Musielak--Orlicz function.
The \emph{martingale Musielak--Orlicz Hardy spaces} $H_{\varphi}^*(\Omega)$,
$H_{\varphi}^S(\Omega)$, $H_{\varphi}^s(\Omega)$, $P_{\varphi}(\Omega)$
and $Q_{\varphi}(\Omega)$ are defined, respectively, as follows:
\begin{align*}
H_{\varphi}^M(\Omega):=\left\{f\in\mathcal{M}:\ \|f\|_{H_{\varphi}^M(\Omega)}:=\lf\|M(f)\r\|_{L^{\varphi}(\Omega)}<\infty\right\},
\end{align*}
\begin{align*}
H_{\varphi}^S(\Omega):=\left\{f\in\mathcal{M}:\ \|f\|_{H_{\varphi}^S(\Omega)}
:=\lf\|S(f)\r\|_{L^{\varphi}(\Omega)}<\infty\right\},
\end{align*}
\begin{align*}
H_{\varphi}^s(\Omega):=\left\{f\in\mathcal{M}
:\ \|f\|_{H_{\varphi}^s(\Omega)}:=\lf\|s(f)\r\|_{L^{\varphi}(\Omega)}<\infty\right\},
\end{align*}
\begin{align*}
P_{\varphi}(\Omega):=\left\{f\in\mathcal{M}:\ \|f\|_{P_{\varphi}(\Omega)}:=\inf_{(\lambda_n)_{n\in\mathbb{Z}_{+}}
\in\Lambda[P_{\varphi}(\Omega)]}\|\lambda_{\infty}\|_{L^{\varphi}(\Omega)}<\infty\right\}
\end{align*}
and
\begin{align*}
Q_{\varphi}(\Omega):=\left\{f\in\mathcal{M}:\ \|f\|_{Q_{\varphi}(\Omega)}
:=\inf_{(\lambda_n)_{n\in\mathbb{Z}_{+}}\in\Lambda[Q_{\varphi}(\Omega)]}
\|\lambda_{\infty}\|_{L^{\varphi}(\Omega)}<\infty\right\}.
\end{align*}
\end{definition}

\begin{remark}\label{rmk-s}
The above five martingale Musielak--Orlicz Hardy spaces are the generalization of several
known martingale Hardy spaces. For example, let $p\in(0,\fz)$, $w$ be a weight,
$\Phi$ an Orlicz function on $\D$ and $p(\cdot):\ \Omega\to [1,\fz]$ a measurable function.
If $\V(x,t):=w(x)t^p$, $\Phi(t)$, $t^{p(x)}$ or $w(x)\Phi(t)$ for
any $x\in\Omega$ and $t\in\D$, then the corresponding martingale Musielak--Orlicz Hardy space
becomes, respectively, the weighted martingale Hardy space,
the martingale Orlicz--Hardy space (see \cite[p.\,671]{mns12}),
the variable martingale Hardy space (see \cite[Chapter 2]{wyong1}) or
the weighted martingale Orlicz--Hardy space.
\end{remark}

\begin{definition}\label{def atom}
Let $\varphi$ be a Musielak--Orlicz function. A measurable function $a$ is
called a \emph{$(\varphi,\fz)_s$-atom} if there exists a stopping time $\nu$ related to $\{\mathcal{F}_n\}_{n\in\mathbb{Z}_{+}}$
($\nu$ is called the \emph{stopping time} associated with $a$) such that
\begin{enumerate}
\item[\rm{(i)}] $a_n:=\mathbb{E}_na=0$ if $\nu\geq n,$
\item[\rm{(ii)}] $\left\|s(a)\right\|_{L^{\fz}(B_{\nu})}
          \le \|\mathbf{1}_{B_{\nu}}\|_{L^{\varphi}(\Omega)}^{-1}.$
\end{enumerate}
Similarly, \emph{$(\varphi,\fz)_S$-atoms} and \emph{$(\varphi,\fz)_M$-atoms}
are defined, respectively, via replacing (ii) in the above definition by
$$\lf\|S(a)\r\|_{L^{\fz}(B_{\nu})}
\le \|\mathbf{1}_{B_{\nu}}\|_{L^{\varphi}(\Omega)}^{-1}$$
and
$$\lf\|M(a)\r\|_{L^{\fz}(B_{\nu})}
\le \|\mathbf{1}_{B_{\nu}}\|_{L^{\varphi}(\Omega)}^{-1}.$$
\end{definition}

Let $r\in\D$. Denote by $\mathcal{A}_s(\varphi,{\fz})$ (resp., $\mathcal{A}_S(\varphi,{\fz})$
or $\mathcal{A}_{M}(\varphi,{\fz})$) the set of all sequences of triples $\{\mu^k,a^k,\nu^k\}_{k\in\mathbb{Z}},$
where $\{\mu^k\}_{k\in\zz}$ are non-negative real numbers, $\{a^k\}_{k\in\zz}$
are $(\varphi,{\fz})_s$-atoms (resp., $(\varphi,{\fz})_S$-atoms or $(\varphi,{\fz})_M$-atoms),
and $\{\nu^k\}_{k\in\zz}\subset\mathrm{\mathcal{T}}$ satisfying (i) and (ii)
of Definition \ref{def atom}, and also
\begin{align*}
\left\|\left\{\sum_{k\in\mathbb{Z}}\left[\frac{\mu^k\mathbf{1}_{B_{\nu^k}}}
{\|\mathbf{1}_{B_{\nu^k}}\|_{L^{\varphi}(\Omega)}}\right]^r\right\}^{\frac1r}
\right\|_{L^{\varphi}(\Omega)}<\infty.
\end{align*}

\begin{definition}
Let $\V$ be a Musielak--Orlicz function and $r\in\D$. The \emph{atomic martingale
Musielak--Orlicz Hardy space} $H_{{\rm at},\,r}^{\V,\fz,s}(\Omega)$ (resp., $H_{{\rm at},\,r}^{\V,\fz,S}(\Omega)$, $H_{{\rm at},\,r}^{\V,\fz,M}(\Omega)$) is defined to be the space of all $f\in\mathcal{M}$
satisfying that there exists a sequence of triples,
$\{\mu^k,a^k,\nu^k\}_{k\in\mathbb{Z}}\in\mathcal{A}_s(\varphi,{\fz})$ (resp., $\mathcal{A}_S(\varphi,{\fz})$
or $\mathcal{A}_{M}(\varphi,{\fz})$),
such that, for any $n\in\mathbb{Z}_{+}$,
\begin{align}
\sum_{k\in\mathbb{Z}}\mu^ka_n^k=f_n .
\end{align}
Moreover, let
\begin{align*}
\|f\|_{H_{{\rm at},\,r}^{\V,\fz,s}(\Omega)}\ \ \lf({\rm resp.},\ \|f\|_{H_{{\rm at},\,r}^{\V,\fz,S}(\Omega)},\ \|f\|_{H_{{\rm at},\,r}^{\V,\fz,M}(\Omega)}\r)
:=\inf\lf\{\left\|\left[\sum_{k\in\mathbb{Z}}\left\{\frac{\mu^k\mathbf{1}_{B_{\nu^k}}}
{\|\mathbf{1}_{B_{\nu^k}}\|_{L^{\varphi}(\Omega)}}\right\}^r\right]^{\frac1r}
\right\|_{L^{\varphi}(\Omega)}\r\}<\fz,
\end{align*}
where the infimum is taken over all decompositions of $f$ as above.
\end{definition}

The stochastic basis $\{\mathcal{F}_n\}_{n\in\mathbb{Z}_{+}}$ is said to be \emph{regular}
if there exists a positive constant $R$ such that, for any $n\in\nn$,
\begin{align}\label{R}
f_n\le Rf_{n-1}
\end{align}
holds true for any nonnegative martingale $(f_n)_{n\in\mathbb{Z}_{+}}$.

The weights we consider in this article are \emph{special weights}, that is, the martingales generated by a strictly positive $\varphi\in L^1(\Omega)$. To be precise, let
$\varphi(\cdot,t):=\{\varphi_n(\cdot,t)\}_{n\in\mathbb{Z}_{+}}$ be the martingale generated by $\varphi(\cdot,t)$ for any $t\in[0,\infty).$ For simplicity, we still use $\varphi(\cdot,t)$ to
denote the martingale $\varphi(\cdot,t):=\{\varphi_n(\cdot,t)\}_{n\in\mathbb{Z}_{+}}$.

The definition of $A_p$ weights for martingales was introduced by Izumisawa and Kazamaki in \cite{IK77},
which is now generalized to the Musielak--Orlicz case as follows.
\begin{definition}\label{def-wei}
Let $q\in[1,\fz)$. A positive Musielak--Orlicz function
$\V:\ \Omega\times[0,\fz)\to[0,\fz)$ is said to
satisfy the \emph{uniformly} \emph{$A_q(\Omega)$ condition}
if there exists a positive constant $K$ such that,
when $q\in(1,\fz)$,
$$\sup_{t\in\D}\ee_n(\V)(\cdot,t)
\lf[\ee_n\lf(\V^{-\frac1{q-1}}\r)(\cdot,t)\r]^{q-1}\le K
\quad \MP{\text-}{\rm almost\ \ everywhere},\quad \forall\,n\in\mathbb{Z}_{+}$$
and, when $q=1$,
$$\sup_{t\in\D}\ee_n(\V)(\cdot,t)\frac{1}{\V(\cdot,t)}\le K
\quad \MP{\text-}{\rm almost\ \ everywhere},\quad \forall\,n\in\mathbb{Z}_{+}.$$
$\V$ is said to satisfy
$A_{\fz}(\Omega)$ if $\V\in A_q(\Omega)$ for some $q\in[1,\fz)$.
\end{definition}
It is easy to see that, for any $p$, $q\in(1,\fz)$ with $p\le q$,
$A_1(\Omega)\subset A_p(\Omega)\subset A_q(\Omega)\subset A_{\fz}(\Omega)$.
Assuming that $\V$ is a Musielak--Orlicz function, let
$$q(\V):=\inf\lf\{q\in[1,\fz):\ \V\in A_{q}(\Omega)\r\}.$$

The following $\ss$ condition appears naturally
in the weighted martingale inequalities.
For more discussions, see Dol\'eans--Dade and Meyer \cite{DM78}
and also Bonami and L\'epingle \cite{bl79}.
\begin{definition}
Let $t\in[0,\infty)$. The martingale $\varphi(\cdot,t):=\{\varphi_n(\cdot,t)\}_{n\in\mathbb{Z}_{+}}$
is said to satisfy the \emph{uniformly $\mathbb{S}$ condition}, denoted by $\varphi\in \mathbb{S},$
if there exists a positive constant $K$ such that,
for any $t\in[0,\infty)$, $n\in\mathbb{N}$ and almost every $\omega\in\Omega,$
\begin{align}\label{WS}
\frac{1}{K}\varphi_{n-1}(\omega,t)\le\varphi_n(\omega,t)\le K\varphi_{n-1}(\omega,t).
\end{align}
The \emph{conditions $\mathbb{S}^{-}$} and \emph{$\mathbb{S}^{+}$} denote two parts
of $\mathbb{S}$ satisfying only the left or the right hand side of the
preceding inequalities, respectively.
\end{definition}

The following lemma comes from Dol\'eans--Dade and Meyer \cite{DM78}
(see also \cite[Corollary 6.3.3]{Long93}).
\begin{lemma}\label{lp}
Let $p\in(1,\fz)$ and $w$ be a weight. If $w\in A_p(\Omega)\cap \ss^{-}$,
then there exists a positive constant $\varepsilon$ such that $w\in A_{p-\varepsilon}(\Omega)$.
\end{lemma}
Bonami and L\'epingle \cite[Section 3]{bl79} gave an example
to illustrate that there exists a weight $w\in A_p(\Omega)$, while
$w\notin \bigcup_{\varepsilon\in\D}A_{p-\varepsilon}(\Omega)$.

\begin{lemma}\label{lem-ss}
Let $\V$ be a Musielak--Orlicz function and $\V\in A_{\fz}(\Omega)$. If the stochastic basis
$\{\mathcal{F}_n\}_{n\in\mathbb{Z}_{+}}$ is regular, then $\V\in\ss$.
\end{lemma}
\begin{proof}
Since $\V\in A_{\fz}(\Omega)$, it follows that there exists an index $p\in(1,\fz)$
such that $\V\in A_{p}(\Omega)$.
For any $x\in\Omega$ and $t\in\D$, let
$$\widehat{\V}(x,t):=\lf[\V(x,t)\r]^{-\frac1{p-1}}\quad \mbox{and}
\quad \widehat{\V}_n(x,t):=\ee_n\lf(\widehat{\V}(\cdot,t)\r)(x), \quad n\in\mathbb{Z}_{+}.$$
Then, by the H\"older inequality of the conditional expectation and $\V\in A_{p}(\Omega)$,
we find that there exists a positive constant $K$ such that, for any $n\in\mathbb{Z}_{+}$,
$$1=\lf[\ee_n\lf(\V^{\frac1p}\V^{-\frac1p}\r)\r]^{p}
\le \V_n\lf[\ee_n\lf(\V^{-\frac{1}{p-1}}\r)\r]^{p-1}
=\V_n\widehat{\V}_n^{p-1}\le K.$$
From this and the regularity, we deduce that, for any $n\in\nn$,
$$\V_{n-1}\le K\,\widehat{\V}_{n-1}^{1-p}\le K R^{p-1}\widehat{\V}_{n}^{1-p}
\le K R^{p-1} \V_n,$$
where $R$ is a positive constant as in \eqref{R}. This implies that $\V\in\ss^{-}$.
Notice that the right hand side of \eqref{WS} follows from the regularity condition \eqref{R}.
Thus, we have $\V\in\ss$, which completes the proof of Lemma \ref{lem-ss}.
\end{proof}

\begin{definition}\label{d3}
Let $\varphi$ be a Musielak--Orlicz function.
For any $p\in(0,\infty)$, $\varphi$ is said to
be of \emph{uniformly lower} (resp., \emph{upper}) \emph{type} $p$
if there exists a positive constant $C_{(p)}$, depending on \emph{p}, such that,
for any $x\in\Omega$ and $t\in[0,\infty),$ $s\in(0,1)$ (resp., $s\in[1,\infty)$),
\begin{align}\label{up}
\varphi(x,st)\le C_{(p)}s^p\varphi(x,t).
\end{align}

\end{definition}
\begin{remark}\label{rem-uni}
Obviously, if $\varphi$ is both of uniformly lower
type $p_1$ and of uniformly upper type $p_2$,
then $p_1\le p_2$. Moreover, if $\varphi$ is of uniformly
lower (resp., upper) type $p$, then,
it is also of uniformly lower (resp., upper) type
$\widetilde{p}$ for any $\widetilde{p}\in(0,p)$
(resp., $\widetilde{p}\in(p,\fz)$).
\end{remark}

\begin{definition}\label{d1}
Let $\varphi$ be a Musielak--Orlicz function. For any given $r \in (0,\infty)$, let
\[
	\varphi_r(x,t):= \varphi(x,t^{r}), \qquad \forall\,x\in\Omega,\ \ \forall\,t\in[0,\infty).
\]
\end{definition}

The following lemma can be proved with standard arguments, the details being omitted.

\begin{lemma}\label{l1}
Let $\V$ be a Musielak--Orlicz function. Then, for any measurable function $f$ and $r \in (0,\infty)$, it holds true that
\[
	\left\| |f|^{r}\right\|_{L^{\varphi}[0,1)} = \left\| f\right\|_{L^{\varphi_r}[0,1)}^{r}.
\]
\end{lemma}

\begin{lemma}\label{l2}
Assumed that $r \in (0,\infty)$. If $\varphi$ is a Musielak--Orlicz function
of uniformly lower (resp., upper) type $p \in (0,\infty)$,
then $\varphi_r$ is of uniformly lower (resp., upper) type $pr$.
\end{lemma}

\begin{proof}
The desired result follows from the fact that, for any $x\in\Omega$, $s\in(0,1)$ and $t\in\D$,
	\[
		\varphi_r(x,st)=\varphi(x,s^{r}t^{r}) \leq C_{(p)}s^{pr}\varphi(x,t^{r}) = C_{(p)}s^{pr}\varphi_r(x,t),
	\]
where $C_{(p)}$ is the positive constant as in \eqref{up}.
This finishes the proof of Lemma \ref{l2}.
\end{proof}

\begin{definition}
Let $\V$ be a Musielak--Orlicz function. Then the function
\begin{align}\label{com}
\V^{*}(x,t)=\sup_{u\in\D}\lf[ut-\V(x,u)\r], \quad \forall\,x\in\Omega,\ \ \forall\,t\in\D,
\end{align}
is said to be \emph{complimentary to $\varphi$}.
\end{definition}
Let us denote by $\varphi^{*}_{r}$ the complimentary function to $\varphi_{r}$.

The following lemma gives the relationship between the Musielak--Orlicz function
and its complimentary function.
\begin{lemma}\label{l20}
Suppose that $\varphi$ is a Musielak--Orlicz function and $\varphi^{*}$ its
complimentary function. If $\varphi$ is of uniformly lower (resp., upper) type $p$ with $p\in(1,\infty)$, then $\varphi^{*}$ is of uniformly upper (resp., lower) type $p'$.
\end{lemma}

\begin{proof}
	Suppose that $\varphi$ is of uniformly lower type $p$, namely,
there exists a positive constant $C_{(p)}$ such that
$$\varphi(x,st) \leq C_{(p)} s^{p} \varphi(x,t),
\qquad \forall\,x\in\Omega\mbox{, } \forall\,s\in(0, 1]\mbox{, } \forall\,t\in\D$$
or, equivalently,
$$
\varphi(x,t) \leq C_{(p)} s^{p} \varphi(x,t/s),
\qquad \forall\,x\in\Omega\mbox{, } \forall\,s\in(0, 1]\mbox{, } \forall\,t\in\D.
$$
	We may suppose that $C_{(p)}\in [1,\fz)$. Let $
	\tilde{\varphi}(x,t):=C_{(p)} s^{p} \varphi(x,t/s)
	$ for any $x\in\Omega$, $t\in\D$ and $s\in(0,1]$ and $(\tilde{\varphi})^{*}$ be its complementary function. By \eqref{com}, we know that, for any $x\in\Omega$, $t\in\D$ and $s\in(0,1]$,
\begin{align*}
		(\tilde\varphi)^{*}(x,t) &= \sup_{u \in (0,\infty)} \left[ut- \tilde \varphi(x,u)\right]
= \sup_{u \in (0,\infty)} \left[ut- C_{(p)} s^{p} \varphi(x,u/s)\right] \\
		&= \sup_{u \in (0,\infty)} \left[stu- C_{(p)} s^{p} \varphi(x,u)\right]
= C s^{p} \sup_{u \in (0,\infty)} \left[C_{(p)}^{-1} s^{1-p} tu- \varphi(x,u)\right] \\
		&= C_{(p)} s^{p} \varphi^{*}(x,C^{-1} s^{1-p} t).
\end{align*}
	Since $\varphi(x) \leq \tilde \varphi(x)$, from \eqref{com}, it follows that $(\tilde \varphi)^{*}(x,t) \leq \varphi^{*}(x,t)$. Thus, we obtain, for any $x\in\Omega$, $t\in\D$ and $s\in(0,1]$,
	\[
		C_{(p)} s^{p} \varphi^{*}(x,C_{(p)}^{-1} s^{1-p} t) \leq \varphi^{*}(x,t).
	\]
This further implies that, for any $x\in\Omega$ and $t\in\D$,
	\begin{equation}\label{e10}
		\varphi^{*}(x,vt) \leq C_{(p)}^{1/(p-1)} v^{p/(p-1)} \varphi^{*}(x,t)
=C_{(p)}^{1/(p-1)} v^{p'} \varphi^{*}(x,t), \qquad \forall\,v \in [1/C_{(p)},\fz).
	\end{equation}
	Since $C_{(p)} \in [1,\fz)$, this shows that $\varphi^{*}$ is of uniformly upper type $p'$.

The desired result can be proved similarly when $\varphi$ is of uniformly upper type $p$. Indeed, similarly to \eqref{e10}, we conclude that, for any $x\in\Omega$ and $t\in\D$,
	\[
		\varphi^{*}(x,vt) \leq C_{(p)}^{1/(p-1)} v^{p'} \varphi^{*}(x,t), \qquad \forall\,v \in(0, 1/C_{(p)}].
	\]
Since $\varphi^{*}$ is increasing, we have, for any $v\in(\frac1{C_{(p)}},1]$ and $t\in(0,\infty)$,
	\[
		\varphi^{*}(x,vt)\le \varphi^{*}(x,t) \leq C_{(p)}^{p'}{v^{p'}}\varphi^{*}(x,t)
	\]
	and hence $\varphi^{*}$ is of uniformly lower type $p'$.
This finishes the proof of Lemma \ref{l20}.
\end{proof}

Now we give some examples of Musielak--Orlicz functions with various properties.

\begin{example}\label{x1}
{\rm All the following functions $\V$ are Musielak--Orlicz functions,
where $w$ is a special weight on $\Omega$.
\begin{enumerate}
\item[{\rm (i)}] For any given $p\in(1,\fz)$ and any $x\in\Omega$ and $t\in\D$, let $\V(x,t):=w(x)t^p/p$ [or $\V(x,t):=w(x)t^p$].
Then $\V^{*}(x,t)=[w(x)]^{-\frac{1}{p-1}}t^{p'}/p'$
[or $\V^{*}(x,t)=[w(x)]^{-\frac{1}{p-1}}t^{p'}/p'p^{-\frac1{p-1}}$]
for any $x\in\Omega$ and $t\in\D$,
where $p'$ is the conjugate index of $p$ and hence $\V^{*}$
is of uniformly lower type $p^{-}$ for any $p^{-}\in(0,p']$ and of uniformly upper
type $p^{+}$ for any $p^{+}\in[p',\fz)$.

\item[{\rm (ii)}]  Let $\alpha\in(2,\fz)$ and $\V(x,t):=w(x)t^{\alpha}(1+|\log t|)$ for any $x\in\Omega$ and $t\in\D$.
It is not difficult to show that $\V$ is of uniformly lower type $\alpha-\varepsilon$ and of uniformly upper type $\alpha+\varepsilon$, where $\varepsilon$ is an arbitrary positive constant.
By Remark \ref{rem-uni} and Lemma \ref{l20}, we know that $\V^{*}$ is of uniformly lower type $p^{-}$ for any $p^{-}\in(0,\alpha')$ and of uniformly upper type
$p^{+}$ for any $p^{+}\in(\alpha',\fz)$.

\item[{\rm (iii)}] Let $\V(x,t):=w(x)(e^t-t-1)$ for any $x\in\Omega$ and $t\in\D$.
Then we know that $\V$ is of uniformly lower type $p^{-}$ for any $p^{-}\in(0,2]$.
However, $\V$ does not have any uniformly upper type property.

\item[{\rm (iv)}] Let $\alpha\in(1,\fz)$ and $\V(x,t):=w(x)t^{\alpha}(1+\log(1+t))$ for any $x\in\Omega$ and $t\in\D$. Then $\V$ is of uniformly lower type $\alpha$ and of uniformly upper type $\alpha+\varepsilon$ for any $\varepsilon\in(0,\fz)$. Thus, $\V^{*}$ is of uniformly lower type $p^{-}$ for any $p^{-}\in(0,\alpha')$ and of uniformly upper type $p^{+}$ for any $p^{+}\in[\alpha',\fz)$.

\item[{\rm (v)}] Let $\alpha\in[2,\fz)$ and $\V(x,t):=w(x)t^{\alpha}/\log(e+t)$ for any $x\in\Omega$ and $t\in\D$. Then $\V$ is of uniformly lower type $\alpha- \varepsilon$ for any $\varepsilon\in(0,\fz)$ and of uniformly upper type $\alpha$. Thus, $\V^{*}$ is of uniformly lower type $p^{-}$ for any $p^{-}\in(0,\alpha']$ and of uniformly upper type $p^{+}$ for any $p^{+}\in(\alpha',\fz)$.

\item[{\rm (vi)}] Let $\alpha\in(1,\fz)$, $\beta\in(0,\fz)$, $\gamma\in[0,2 \alpha(1+ \log 2)]$
and
$$
\V(x,t):= \frac{t^{\alpha}}{\left(\log(e+x)\right)^{\beta}+\left(\log(e+t)\right)^{\gamma}}
$$
for any $x\in\Omega = [0,1)$ and $t\in\D$. Then $\V \in A_1(\Omega)$, which
is of uniformly lower type $\alpha- \varepsilon$ for any $\varepsilon\in(0,\fz)$ and of uniformly upper type $\alpha$. Thus, $\V^{*}$ is of uniformly lower type $p^{-}$ for any $p^{-}\in(0,\alpha']$ and of uniformly upper type $p^{+}$ for any $p^{+}\in(\alpha',\fz)$ (see Yang et al.
\cite[p.\,11]{Yang2017}).

\item[{\rm (vii)}] Let $1<p\le q<\fz$ and $\V(x,t):=t^p+w(x)t^{q}$ for any $x\in\Omega$ and $t\in\D$. Then $\V$ is of uniformly lower type $p^{-}$ for any $p^{-}\in(0,p]$ and of uniformly upper type $p^{+}$ for any $p^{+}\in[q,\fz)$. Thus, $\V^{*}$ is of uniformly lower type $p^{-}$
    for any $p^{-}\in(0,q']$ and of uniformly upper type $p^{+}$ for any $p^{+}\in[p',\fz)$.
     If $w\in A_\infty(\Omega)$, we claim that $\varphi \in A_\infty(\Omega)$. Indeed,
    since $w\in A_\infty(\Omega)$, it follows that there exists $q_0\in[1,\infty)$ such that $w\in A_{q_0}(\Omega)$. If $q_0=1$, it is clear that $\V\in A_1(\Omega)$. Now we assume that $q_0\in(1,\fz)$.
    Notice that, for any $n\in\nn$,
    $$\frac{t^p+w_nt^q}{t^p+wt^q}
    \le \mathbf{1}_{\{x\in\Omega:\ w_n(x)<w(x)\}}+\frac{1+w_nt^{q-p}}{1+wt^{q-p}}
    \mathbf{1}_{\{x\in\Omega:\ w_n(x)\geq w(x)\}}\le 1+\frac{w_n}{w}.$$
    From this and $w\in A_{q_0}(\Omega)$, we deduce that
    \begin{align*}\sup_{t\in\D}\ee_n(\V)(\cdot,t)
\lf[\ee_n\lf(\V^{-\frac1{q_0-1}}\r)(\cdot,t)\r]^{q_0-1}
&=\sup_{t\in\D}\lf[\ee_n\lf(\lf(\frac{t^p+w_nt^q}{t^p+wt^q}\r)^{\frac1{q_0-1}}\r)\r]^{q_0-1}\\
&\lesssim 1+\lf[\ee_n\lf(\lf(\frac{w_n}{w}\r)^{\frac1{q_0-1}}\r)\r]^{q_0-1}\\
&\sim 1+\ee_nw\lf[\ee_n\lf(w^{-\frac1{q_0-1}}\r)\r]^{q_0-1}\lesssim1,
\end{align*} namely, $\V\in A_{q_0}(\Omega)$, which completes the proof of the above claim.
\end{enumerate}}
\end{example}

\begin{example}\label{exv}
{\rm Let $\V(x,t):=t^{p(x)}$ for any $x\in\Omega$ and $t\in\D$,
$$1<p_1:=\mathop\mathrm{ess\,inf}_{x\in\Omega}p(x)\le p_2:=\mathop\mathrm{ess\,sup}_{x\in\Omega}p(x)<\fz.$$
Then $\V^{*}(x,t)=t^{q(x)}$ for any $x\in\Omega$ and $t\in\D$,
where $\frac1{p(x)}+\frac1{q(x)}=1$. Thus, $\V$ is of uniformly lower
type $p^{-}$ for any $p^{-}\in(0,p_1]$ and of uniformly upper type
$p^{+}$ for any $p^{+}\in[p_2,\fz)$.
However, $\varphi \not \in A_\infty(\Omega)$ (see, for example, \cite[Remark 2.23]{yyz}).}
\end{example}

The following key lemma was originated from \cite[Theorem 13.18]{m83} and improved by
\cite[Lemma 2.7 and Remark 2.9]{ch18}.
\begin{lemma}\label{lem-du}
Let $\varphi$ be Musielak--Orlicz function and $\varphi^{*}$ its complimentary function.
If $\V$ is of uniformly lower type $1$, then there exists a positive constant $C$ such that,
for any $f\in L^{\V}(\Omega)$,
\begin{align}\label{dv}
\frac1{C}\|f\|_{L^{\V}(\Omega)}\le \sup_{g\in L^{\V^{*}}(\Omega),
\ \|g\|_{L^{\V^{*}}(\Omega)}\le1} \int_{\Omega}fg\,d\MP
\le2\|f\|_{L^{\V}(\Omega)}.\end{align}
\end{lemma}
Notice that, if the Musielak--Orlicz function $\varphi$ is of uniformly lower type
$p^{-}_{\V}$ with $p^{-}_{\V}\in(1,\fz)$, then \eqref{dv} holds true.

Actually, Lemma \ref{lem-du} was proved in \cite{ch18} for Musielak--Orlicz functions
$\V$ satisfying that there exists a positive constant
$C$ such that, for any $0<s<t$ and $x\in\Omega$,
\begin{align}\label{du1}
	\frac{\varphi(x,s)}{s} \leq C \frac{\varphi(x,t)}{t}.
\end{align}
We should point out that \eqref{du1} holds true if and only if $\varphi$ is of uniformly lower type $1$.
Indeed, if $\varphi$ is of uniformly lower type $1$, then, for any $0<s<t$ and $x\in\Omega$,
\[
	\varphi(x,s) = \varphi\left(x,\frac{s}{t} t\right) \leq C \frac{s}{t} \varphi(x,t).
\]
Conversely, if \eqref{du1} holds true with some positive constant $C$,
then, for any $s\in(0,1)$, $t\in\D$ and $x\in\Omega$,
$$\frac{\V(x,st)}{st}\le C \frac{\V(x,t)}{t}\Longleftrightarrow
\V(x,st)\le C s \V(x,t).$$

\section{The Doob maximal operator\label{s3}}

In this section, we explore the boundedness of the Doob maximal operator.

To prove the Doob maximal inequality for Musielak--Orlicz spaces,
we need the following interpolation theorem about the sublinear operator on $L^{\V}(\Omega)$,
which was proved in \cite[Theorem 2.7]{lhy12} (see also \cite[Theorem 2.1.1]{Yang2017}).
\begin{theorem}\label{thm-ip}
Let $p_1$, $p_2\in\D$, $p_1<p_2$ and $\V$ be a Musielak--Orlicz function with uniformly
lower type $p_{\V}^{-}$ and upper type $p_{\V}^{+}$. If
$0<p_1<p_{\V}^{-}\le p_{\V}^{+}<p_2<\fz$ and $T$ is a sublinear operator defined
on $L^{p_1}(\Omega,\,\V(\cdot,1)\,d\MP)+L^{p_2}(\Omega,\,\V(\cdot,1)\,d\MP)$
satisfying that, for any $i\in\{1,2\}$ and any $\alpha\in\D$, $t\in\D$,
\begin{align}\label{z}
\V\lf(\{x\in\Omega:\ |T(f)(x)|>\alpha\},t\r)
\le C_i \alpha^{-p_i}\int_{\Omega}|f(x)|^{p_i}\V(x,t)\,d\MP,
\end{align}
where $C_i$ is a positive constant independent of $f$, $t$ and $\alpha$.
Then $T$ is bounded on $L^{\V}(\Omega)$ and, moreover, there exists a positive constant
$C$ such that, for any $f\in L^{\V}(\Omega)$,
$$\int_{\Omega}\V(x,|T(f)(x)|)\,d\MP \le C\int_{\Omega}\V(x,|f(x)|)\,d\MP.$$
\end{theorem}

\begin{theorem}\label{thm-doob}
Let $\V\in A_{\fz}(\Omega)$ be a Musielak--Orlicz function with uniformly
lower type $p_{\V}^{-}$ and upper type $p_{\V}^{+}$.
If
\begin{align}\label{doob}
q(\V)<p_{\V}^{-}\le p_{\V}^{+}<\fz,
\end{align}
then the Doob maximal operator $M$
is bounded on $L^{\V}(\Omega)$ and, moreover, there exists a positive
constant $C$ such that, for any $f\in L^{\V}(\Omega)$,
$$\int_{\Omega}\V(x,M(f)(x))\,d\MP(x) \le C\int_{\Omega}\V(x,|f(x)|)\,d\MP(x).$$
\end{theorem}
\begin{proof}
From the definition of $q(\V)$, it follows that, for any $p_1$, $p_2\in(q(\V),\fz)$,
$\V\in A_{p_1}(\Omega)$ and $\V\in A_{p_2}(\Omega)$. Combining this and the weighted Doob maximal
inequality (see \cite[Theorem 6.6.3]{Long93}), we find that,
for any $f\in L^{p_1}(\Omega,\,\V(\cdot,1)\,d\MP)+L^{p_2}(\Omega,\,\V(\cdot,1)\,d\MP)$, $\alpha\in\D$ and $t\in\D$,
$$\V\lf(\{x\in\Omega:\ M(f)(x)>\alpha\},t\r)
\le \alpha^{-p_1}\int_{\Omega}\lf[M(f)(x)\r]^{p_1}\V\lf(x,t\r)\,d\MP
\lesssim \alpha^{-p_1}\int_{\Omega}|f(x)|^{p_1}\V(x,t)\,d\MP$$
and
$$\V\lf(\{x\in\Omega:\ M(f)(x)>\alpha\},t\r)
\le \alpha^{-p_2}\int_{\Omega}\lf[M(f)(x)\r]^{p_2}\V\lf(x,t\r)\,d\MP
\lesssim \alpha^{-p_2}\int_{\Omega}|f(x)|^{p_2}\V(x,t)\,d\MP.$$
From this, Theorem \ref{thm-ip} and the fact that $M$ is a sublinear operator, we deduce that $M$
is bounded on $L^{\V}(\Omega)$, which completes the proof of Theorem \ref{thm-doob}.
\end{proof}

As a consequence, we apply Theorem \ref{thm-doob} to obtain the following
weak type inequality of the Doob maximal operator on $L^{\V}(\Omega)$.

\begin{theorem}\label{thm-weak}
Let $\V\in A_{\fz}(\Omega)$ be a Musielak--Orlicz function with uniformly
lower type $p_{\V}^{-}$ and upper type $p_{\V}^{+}$ satisfying \eqref{doob}.
Then there exists a positive constant $C$ such that,
for any $f\in L^{\V}(\Omega)$,
$$\sup_{\rho\in\D}\lf[\rho\lf\|\mathbf{1}_{\{x\in\Omega:\ M(f)(x)>\rho\}}
\r\|_{L^{\V}(\Omega)}\r]\le C \|f\|_{L^{\V}(\Omega)}.$$
\end{theorem}
\begin{proof}
For any $f\in L^{\V}(\Omega)$ and $\rho\in\D$,
by Theorem \ref{thm-doob}, we know that there exists a
constant $C\in(1,\fz)$ such that
\begin{align*}
\int_{\Omega}\V\lf(x,
\frac{\mathbf{1}_{\{x\in\Omega:\ M(f)(x)>\rho\}}(x)}
{\|f\|_{L^{\V}(\Omega)}/\rho}\r)\,d\MP(x)
&=\int_{\{x\in\Omega:\ M(f)(x)>\rho\}}\V\lf(x,
\frac{\rho}{\|f\|_{L^{\V}(\Omega)}}\r)\,d\MP(x)\\
&\le \int_{\Omega}\V\lf(x,
\frac{M(f)(x)}{\|f\|_{L^{\V}(\Omega)}}\r)\,d\MP(x)
\le C\int_{\Omega}\V\lf(x,
\frac{f(x)}{\|f\|_{L^{\V}(\Omega)}}\r)\,d\MP(x)=C.
\end{align*}
Combining this and the uniformly lower type $p_{\V}^{-}$ property of $\V$,
we find that, for any $\rho\in\D$,
$$\int_{\Omega}\V\lf(x,\lf[CC_{(p_{\V}^{-})}\r]^{\frac{-1}{p_{\V}^{-}}}
\frac{\mathbf{1}_{\{x\in\Omega:\ M(f)(x)>\rho\}}(x)}
{\|f\|_{L^{\V}(\Omega)}/\rho}\r)\,d\MP(x)
\le \frac1C \int_{\Omega}\V\lf(x,
\frac{\mathbf{1}_{\{x\in\Omega:\ M(f)(x)>\rho\}}(x)}
{\|f\|_{L^{\V}(\Omega)}/\rho}\r)\,d\MP(x)\le1.$$
Therefore, for any $\rho\in\D$, we have
$$\lf\|\mathbf{1}_{\{x\in\Omega:\ M(f)(x)>\rho\}}
\r\|_{L^{\V}(\Omega)}\le \frac1{\rho}\lf[CC_{(p_{\V}^{-})}\r]^{\frac{1}{p_{\V}^{-}}}
\|f\|_{L^{\V}(\Omega)}.$$
Thus, we obtain
$$\sup_{\rho\in\D}\lf[\rho\lf\|\mathbf{1}_{\{x\in\Omega:\ M(f)(x)>\rho\}}
\r\|_{L^{\V}(\Omega)}\r]\le \lf[CC_{(p_{\V}^{-})}\r]^{\frac{1}{p_{\V}^{-}}}
\|f\|_{L^{\V}(\Omega)}.$$
This finishes the proof of Theorem \ref{thm-weak}.
\end{proof}

Using Theorem \ref{thm-doob}, we also obtain the following corollary.
\begin{corollary}\label{thm-dooc}
Let $\V$ be a Musielak--Orlicz function with uniformly lower type $p^{-}_{\V}$ and uniformly upper type $p^{+}_{\V}$. If $\varphi^{*}\in A_{\fz}(\Omega)$ satisfies
\begin{equation}\label{e38}
	q(\varphi^{*}) < (p^{+}_{\V})' \leq (p^{-}_{\V})'<\infty,
\end{equation}
then the Doob maximal operator $M$
is bounded on $L^{\varphi^{*}}(\Omega)$.
\end{corollary}

\begin{proof}
Lemma \ref{l20} and \eqref{e38} imply that \eqref{doob} holds true for the function $\V^{*}$.
Combining this and Theorem \ref{thm-doob}, we know that $M$
is bounded on $L^{\varphi^{*}}(\Omega)$. This finishes the proof of Corollary \ref{thm-dooc}.
\end{proof}

Now we turn to the dual version of Theorem \ref{thm-doob}.
A detailed treatment of the following inequality for $L^p(\Omega)$
was given by Dilworth in \cite{D1993}. Moreover, Burkholder et al.
\cite[Theorem 3.2]{BDG72} proved it in Orlicz spaces.
Our result provides a general case of the dual version of the Doob maximal inequality.

\begin{theorem}\label{thm-duald}
Let $\V$ be a Musielak--Orlicz function with uniformly lower type $p^{-}_{\V}$
for some $p^{-}_{\V}\in[1,\fz)$.
If the Doob maximal operator $M$ is bounded on $L^{\varphi^{*}}(\Omega)$,
then there exists a positive
constant $C$ such that, for any sequence $(g_k)_{k\in\nn}$ of non-negative
$\cf$ measurable functions,
$$\lf\|\sum_{k\in\nn}\ee_k(g_k)\r\|_{L^{\V}(\Omega)}
\le C\lf\|\sum_{k\in\nn}g_k\r\|_{L^{\V}(\Omega)}.$$
\end{theorem}
\begin{proof}
For any non-negative measurable function $f\in L^{\varphi^{*}}(\Omega)$ with
$\|f\|_{L^{\varphi^{*}}(\Omega)}\le1$, by Remark \ref{rem-uni}, Lemma \ref{lem-du},
the monotone convergence theorem and the assumed boundedness of $M$
on $L^{\varphi^{*}}(\Omega)$, we obtain
\begin{align*}
\int_{\Omega}\sum_{k\in\nn}\ee_k(g_k)f\,d\MP
&=\sum_{k\in\nn}\int_{\Omega}\ee_k(g_k)f\,d\MP
=\sum_{k\in\nn}\int_{\Omega}g_k\ee_k(f)\,d\MP\\
&\le \sum_{k\in\nn}\int_{\Omega}g_k\sup_{k\in\nn}\ee_k(f)\,d\MP
\le2 \lf\|\sum_{k\in\nn}g_k\r\|_{L^{\V}(\Omega)}
\lf\|M(f)\r\|_{L^{\V^{*}}(\Omega)}\\
&\lesssim\lf\|\sum_{k\in\nn}g_k\r\|_{L^{\V}(\Omega)}\|f\|_{L^{\varphi^{*}}(\Omega)}
\lesssim\lf\|\sum_{k\in\nn}g_k\r\|_{L^{\V}(\Omega)}.
\end{align*}
From this, Lemma \ref{lem-du} and the fact that $\sum_{k\in\nn}\ee_k(g_k)$ is non-negative, the conclusion follows immediately. This finishes the proof of Theorem \ref{thm-duald}.
\end{proof}

As an immediate consequence of Corollary \ref{thm-dooc} and Theorem \ref{thm-duald}, we
have the following corollary.
\begin{corollary}\label{cor-d}
Let $\V$ be a Musielak--Orlicz function with uniformly lower type $p^{-}_{\V}$ and uniformly upper type $p^{+}_{\V}$. If $\varphi^{*}\in A_{\fz}(\Omega)$ satisfies \eqref{e38},
then there exists a positive
constant $C$ such that, for any sequence $(g_k)_{k\in\nn}$ of non-negative
$\cf$ measurable functions,
$$\lf\|\sum_{k\in\nn}\ee_k(g_k)\r\|_{L^{\V}(\Omega)}
\le C\lf\|\sum_{k\in\nn}g_k\r\|_{L^{\V}(\Omega)}.$$
\end{corollary}

\begin{remark}
In Theorem \ref{thm-duald} and Corollary \ref{cor-d}, the sequence $(g_k)_{k\in\nn}$ is not assumed to be adapted.
\end{remark}

\begin{corollary}\label{cor-sS}
Let $\V$ be a Musielak--Orlicz function with uniformly lower type $p^{-}_{\V}$ for some $p^{-}_{\V}\in[2,\fz)$.
If the Doob maximal operator $M$ is bounded on $L^{\varphi_{1/2}^{*}}(\Omega)$,
then there exists a positive
constant $C$ such that,
for any $f\in\cM$,
$$\|f\|_{H_{\V}^s(\Omega)}\le C \|f\|_{H_{\V}^S(\Omega)}.$$
\end{corollary}
\begin{proof}
For any $f\in\cM$ and $k\in\zz_{+}$, let $g_k:=|d_{k+1}f|^2$. Combining this, Lemmas \ref{l1} and
\ref{l2}, and Theorem \ref{thm-duald}, we find that
\begin{align*}
\|f\|_{H_{\V}^s(\Omega)}&=\lf\|s(f)\r\|_{L^{\V}(\Omega)}=
\lf\|\left[\sum_{k\in\zz_{+}}\ee_k\lf(|d_{k+1}f|^2\r)\right]^{1/2}\r\|_{L^{\V}(\Omega)}
=\lf\|\sum_{k\in\zz_{+}}\ee_k\lf(|d_{k+1}f|^2\r)\r\|_{L^{\V_{1/2}}(\Omega)}^{1/2}\\
&\lesssim \lf\|\sum_{k\in\zz_{+}}|d_{k+1}f|^2\r\|_{L^{\V_{1/2}}(\Omega)}^{1/2}
\sim \lf\|S(f)\r\|_{L^{\V}(\Omega)}\sim \|f\|_{H_{\V}^S(\Omega)},
\end{align*}
which completes the proof of Corollary \ref{cor-sS}.
\end{proof}

Let $p\in [1,\fz)$ and $w\in A_p(\Omega)$. Denote by $\cS$ the set of pairs $(f,g)$ of nonnegative
and measurable functions. If we write
$$\int_{\Omega}\lf[f(x)\r]^pw(x)\,d\MP(x)
\le C\int_{\Omega}\lf[g(x)\r]^pw(x)\,d\MP(x),\quad \forall\,(f,g)\in\cS,$$ we mean that
the above inequality holds true for any pair $(f,g)\in\cS$ and
the positive constant $C$ depends only on $p$ and the $A_p(\Omega)$ constant
of $w$ as in Definition \ref{def-wei}.

The following extrapolation theorem, Theorem \ref{thm-chjz}, plays a crucial role in the proof of
the vector-valued extrapolation theorem, Theorem \ref{thm-v1} below.
The proof of Theorem \ref{thm-chjz} is similar to that
of \cite[Theorem 3.9]{cmp13}, the details being omitted.

\begin{theorem}\label{thm-chjz}
Suppose that, for some $p_0\in[1,\fz)$ and $w_0\in A_{p_0}(\Omega)$,
there exists a positive constant $C$ such that
\begin{align}\label{z1}
\int_{\Omega}\lf[f(x)\r]^{p_0}w_0(x)\,d\MP(x)
\le C\int_{\Omega}\lf[g(x)\r]^{p_0}w_0(x)\,d\MP(x),\quad \forall\,(f,g)\in\cS.
\end{align}
Then, for any given $p\in(1,\fz)$ and $w\in A_{p}(\Omega)$,
there exists a positive constant $C$ such that
$$\int_{\Omega}\lf[f(x)\r]^{p}w(x)\,d\MP(x)
\le C\int_{\Omega}\lf[g(x)\r]^{p}w(x)\,d\MP(x),\quad \forall\,(f,g)\in\cS.$$
\end{theorem}

\begin{theorem}\label{thm-v1}
Assume that $w_0\in A_{p_0}(\Omega)$ for some $p_0\in[1,\fz)$,
and the set $\cS$ satisfies \eqref{z1} with some positive constant $C$.
Then, for any given $p$, $r\in(1,\fz)$ and $w\in A_{p}(\Omega)$,
there exists a positive constant $C$ such that,
for any pair sequence $\{(f_j,g_j)\}_{j\in\nn}\subset \cS$,
\begin{align}\label{z2}
\lf\|\lf(\sum_{j\in\nn}f_j^r\r)^{\frac1r}\r\|_{L^p(\Omega,\,w\,d\MP)}
\le C \lf\|\lf(\sum_{j\in\nn}g_j^r\r)^{\frac1r}\r\|_{L^p(\Omega,\,w\,d\MP)}.
\end{align}
\end{theorem}
\begin{proof}
For any given $r\in(1,\fz)$, let
$$\cS_r:=\lf\{(F,G):\ F:=\lf(\sum_{j\in\nn}f_j^r\r)^{\frac1r}
,\ \ G:=\lf(\sum_{j\in\nn}g_j^r\r)^{\frac1r},\ \ \{(f_j,g_j)\}_{j\in\nn}\subset \cS \r\}.$$
Now, we claim that \eqref{z1} holds true with $p_0=r$ for the set $\cS_r$.
Indeed, for any $(F,G)\in\cS_r$, there exists a sequence
$\{(f_j,g_j)\}_{j\in\nn}\subset \cS$ such that
$$F=\lf(\sum_{j\in\nn}f_j^r\r)^{\frac1r}
\quad \mbox{and}\quad G=\lf(\sum_{j\in\nn}g_j^r\r)^{\frac1r}.$$
For any given $w\in A_r(\Omega)$, from Theorem \ref{thm-chjz} with $p=r$, it follows that
\begin{align*}
\int_{\Omega}\lf[F(x)\r]^{r}w(x)\,d\MP(x)
&=\sum_{j\in\nn}\int_{\Omega}\lf[f_j(x)\r]^rw(x)\,d\MP(x)\\
&\lesssim \sum_{j\in\nn}\int_{\Omega}\lf[g_j(x)\r]^rw(x)\,d\MP(x)
\sim \int_{\Omega}\lf[G(x)\r]^{r}w(x)\,d\MP(x),
\end{align*}
which proves the above claim. Using this claim and Theorem \ref{thm-chjz} for the set $\cS_r$,
we know that, for any given $p\in(1,\fz)$ and $w\in A_{p}(\Omega)$,
$$\int_{\Omega}\lf[F(x)\r]^{p}w(x)\,d\MP(x)
\lesssim\int_{\Omega}\lf[G(x)\r]^{p}w(x)\,d\MP(x),\quad \forall\,(F,G)\in\cS_r,$$
which implies that \eqref{z2} holds true with some positive constant $C$.
This finishes the proof of Theorem \ref{thm-v1}.
\end{proof}

Using Theorem \ref{thm-v1}, we obtain the following weighted Fefferman--Stein
vector-valued inequality for the Doob maximal operator.
\begin{theorem}\label{thm-fs}
Let $p\in(1,\fz)$ and $w\in A_{p}(\Omega)$. Then, for any given $r\in(1,\fz)$, there exists
a positive constant $C$ such that, for any sequence $\{f_j\}_{j\in\nn}$ of measurable functions,
$$\lf\|\lf\{\sum_{j\in\nn}\lf[M(f_j)\r]^r\r\}^{\frac1r}\r\|_{L^p(\Omega,\,w\,d\MP)}
\le C\lf\|\lf(\sum_{j\in\nn}\lf|f_j\r|^r\r)^{\frac1r}\r\|_{L^p(\Omega,\,w\,d\MP)}.$$
\end{theorem}
\begin{proof}
Let $p\in(1,\fz)$. Denote the family of extrapolation pairs by
$$
\cS:=\lf\{\lf(M(f),|f|\r):\ f\in L^p(\Omega,\,w\,d\MP)\r\}.$$
Then, by the weighted Doob maximal inequality (see, for example, \cite[Theorem 6.6.3]{Long93}),
we find that, for any given $p\in(1,\fz)$,
$w\in A_{p}(\Omega)$ and for the set $\cS$, \eqref{z1} holds true with some positive constant $C$.
Applying Theorem \ref{thm-v1}, we immediately obtain the Fefferman--Stein
vector-valued inequalities for the Doob maximal operator $M$ on $L^p(\Omega,\,w\,d\MP)$.
This finishes the proof of Theorem \ref{thm-fs}.
\end{proof}

Now, we are in a position to prove the following Fefferman--Stein
vector-valued Doob maximal inequality on Musielak--Orlicz spaces.
\begin{theorem}\label{thm-fsv}
Let $\V\in A_{\fz}(\Omega)$ be a Musielak--Orlicz function with uniformly
lower type $p_{\V}^{-}$ and uniformly upper type $p_{\V}^{+}$ satisfying \eqref{doob}.
Then, for any given $r\in(1,\fz)$, there exists a positive constant
$C$ such that, for any sequence $\{f_j\}_{j\in\nn}$ of measurable functions,
\begin{align}\label{fsv}
\lf\|\lf\{\sum_{j\in\nn}\lf[M(f_j)\r]^r\r\}^{\frac1r}\r\|_{L^{\V}(\Omega)}
\le C\lf\|\lf(\sum_{j\in\nn}\lf|f_j\r|^r\r)^{\frac1r}\r\|_{L^{\V}(\Omega)}.
\end{align}
\end{theorem}
\begin{proof}
Let $r\in(1,\fz)$.
For any $j\in\nn$ and $x\in\Omega$,
if $\sum_{j\in\nn}|f_j(x)|^r\neq0$, let
$$\widetilde{f_j}(x):=\frac{f_j(x)}{[\sum_{j\in\nn}|f_j(x)|^r]^{\frac1r}}$$
and, if $\sum_{j\in\nn}|f_j(x)|^r=0$, $\widetilde{f_j}(x):=0$.
Then, for any given $r\in(1,\fz)$ and any $x\in\Omega$, we obtain
\begin{equation}\label{z3}
\sum_{j\in\nn}\lf|\widetilde{f_j}(x)\r|^r= \left\{
               \begin{array}{ll}
                 1 &\quad\quad \hbox{if $\dis\sum_{j\in\nn}|f_j(x)|^r\neq0$,} \\
                 0 &\quad\quad \hbox{if $\dis\sum_{j\in\nn}|f_j(x)|^r=0$.}
               \end{array}
             \right.
\end{equation}
Since $q(\V)<p_{\V}^{-}\le p_{\V}^{+}$, we can choose $p_1\in(q(\V),p_{\V}^{-})$
and $p_2\in [p_{\V}^{+},\fz)$. For any given $r\in(1,\fz)$, consider the operator
$$T(h):=\lf\{\sum_{j\in\nn}\lf[M\lf(h\widetilde{f_j}\r)\r]^r\r\}^{\frac1r},$$
where $h\in L^{p_1}(\Omega,\,\V(\cdot,1)\,d\MP)+L^{p_2}(\Omega,\,\V(\cdot,1)\,d\MP)$.
Then, from the sublinearity of $M$ and the Minkowski inequality, we deduce that
$T$ is a sublinear operator.
By \eqref{z3} and Theorem \ref{thm-fs}, we find that, for any $i\in\{1,2\}$,
$\alpha\in\D$, $t\in\D$ and
$h\in L^{p_1}(\Omega,\,\V(\cdot,1)\,d\MP)+L^{p_2}(\Omega,\,\V(\cdot,1)\,d\MP)$,
\begin{align}\label{z4}
\V\lf(\{x\in\Omega:\ |T(h)(x)|>\alpha\},t\r)
&=\V\lf(\lf\{x\in\Omega:\ \lf\{\sum_{j\in\nn}\lf[M\lf(h\widetilde{f_j}\r)(x)\r]^r\r\}^{\frac{1}{r}}>\alpha\r\},t\r)\\
&\lesssim \alpha^{-p_i}\int_{\Omega}|h(x)|^{p_i}\lf[\sum_{j\in\nn}\lf|
\widetilde{f_j}(x)\r|^r\r]^{\frac{p_i}{r}}\V\lf(x,t\r)\,d\MP(x)\noz\\
&\lesssim \alpha^{-p_i}\int_{\Omega}|h(x)|^{p_i}\V\lf(x,t\r)\,d\MP(x).\noz
\end{align}
Thus, $T$ satisfies \eqref{z}.

For any given $r\in(1,\fz)$ and any sequence $\{f_j\}_{j\in\nn}$ of measurable functions,
let
$$h_r:=\lf(\sum_{j\in\nn}|f_j|^r\r)^{\frac1r}.$$
Then, for any $j\in\nn$, $h_r\widetilde{f_j}=f_j$.
Combining this, \eqref{z4} and Theorem \ref{thm-ip}, we conclude that, for any given $r\in(1,\fz)$,
$$\int_{\Omega}\V(x,T(h_r)(x))\,d\MP(x)\lesssim \int_{\Omega}\V(x,h_r(x))\,d\MP(x),$$
which implies that \eqref{fsv} holds true with some positive constant $C$.
This finishes the proof of Theorem \ref{thm-fsv}.
\end{proof}

\begin{remark}\label{rmk}
Theorem \ref{thm-fsv} is a probabilistic version of \cite[Corollary 6.1]{ch18}.
It should be noticed that \cite[Corollary 6.1]{ch18} does not capture weighted cases.
This is because the assumptions of \cite[Corollary 6.1]{ch18} requires the weight to
be essentially constant; see \cite[p.\,4331]{ch18} for more details.
However, Theorem \ref{thm-fsv} can cover this important case. To be precise,
in Theorem \ref{thm-fsv}, if we choose $\V(x,t):=w(x)\Phi(t)$ or $\V(x,t):=t^p+w(x)t^q$
for any $x\in\Omega$ and $t\in\D$, where $1< p\le q<\fz$,
$w$ is a special weight and $\Phi$ is an Orlicz function, then, by Example \ref{x1},
we know that Theorem \ref{thm-fsv} cover the weighted Orlicz case
and the case of the double phase functional, which are also new.
\end{remark}

Theorem \ref{thm-fsv} also implies that the following Musielak--Orlicz
version of the Stein inequality (see \cite[p.\,103, Theorem 8]{S70}) holds true.
\begin{theorem}\label{thm-ste}
Let $\V\in A_{\fz}(\Omega)$ be a Musielak--Orlicz function with uniformly
lower type $p_{\V}^{-}$ and uniformly upper type $p_{\V}^{+}$ satisfying \eqref{doob}.
Then, for any given $r\in(1,\fz)$, there exists a positive
constant $C$ such that, for any sequence $(g_k)_{k\in\nn}$
of non-negative $\cf$ measurable functions,
$$\lf\|\lf\{\sum_{k\in\nn}\lf[\ee_k(g_k)\r]^r\r\}^{\frac1r}\r\|_{L^{\V}(\Omega)}
\le C\lf\|\lf[\sum_{k\in\nn}\lf(g_k\r)^r\r]^{\frac1r}\r\|_{L^{\V}(\Omega)}.$$
\end{theorem}
\begin{proof}
Notice that, for any $k\in\nn$, $\ee_k(g_k)\le M(g_k)$. Then the desired conclusion
follows immediately from Theorem \ref{thm-fsv}. This finishes the proof of
Theorem \ref{thm-ste}.
\end{proof}
\begin{remark}
For any given $p\in(1,\fz)$, when $\V(x,t):=t^p$ for any $x\in\Omega$ and $t\in\D$,
Theorem \ref{thm-ste} with $r=2$ was studied by Stein \cite[p.\,103, Theorem 8]{S70}, and then Theorem \ref{thm-ste} with $r\in(1,\fz)$ was investigated by
Dilworth \cite[Theorem 3.2 and Corollary 2.4]{D1993}. If, in Theorem \ref{thm-ste}, we choose $\V(x,t):=t$ for any $x\in\Omega$ and $t\in\D$, then Theorem \ref{thm-ste} with $r\in(1,\fz)$ breaks down (see \cite[Remark 2.7]{D1993}). Thus, since $\V$ in Theorem \ref{thm-ste} is of
wide generality, Theorem \ref{thm-ste} generalize the Stein inequality
to more general case. Especially, let $p\in(1,\fz)$ and $w$
be a weight, if $\V(x,t):=w(x)t^p$ for any $x\in\Omega$ and $t\in\D$,
then Theorem \ref{thm-ste} becomes the weighted Stein inequality, which is also new.
\end{remark}

\section{Atomic characterizations\label{s4}}

In this section, we establish the atomic characterization of martingale
Musielak--Orlicz Hardy spaces
$H_{\V}^s(\Omega)$, $P_{\V}(\Omega)$, $Q_{\V}(\Omega)$, $H_{\V}^M(\Omega)$
and $H_{\V}^S(\Omega)$.

\begin{theorem}\label{Thm-atms}
Let $\varphi$ be a Musielak--Orlicz function.
Then, for any given $r\in(0,1]$,
$$H_{\V}^s(\Omega)=H_{{\rm at},\,r}^{\V,\fz,s}(\Omega)\qquad
\mbox{ with equivalent quasi-norms}.$$
\end{theorem}
\begin{proof}
We prove this theorem by two steps.

Step $1)$ Prove $H_{\V}^s(\Omega)\subset H_{{\rm at},\,r}^{\V,\fz,s}(\Omega)$.
To prove this, let $f\in H_{\V}^s(\Omega)$. For any $k\in\mathbb{Z},$ the stopping time
$\nu^k$ is defined by setting
$$\nu^k:=\inf\lf\{n\in\zz_{+}:\ s_{n+1}\lf(f\r)>2^k\r\}\quad \lf(\inf\emptyset=\infty\r).$$
Obviously, the sequence $(\nu^k)_{k\in\zz}$ of stopping times is non-decreasing.
Similarly to the proof of \cite[Theorem 2.2]{W94},
we have, for the stopped martingale
$f^{\nu^k}:=(f_n^{\nu^k})_{n\in\mathbb{Z}_{+}}:=(f_{n\wedge\nu^k})_{n\in\mathbb{Z}_{+}}$,
\begin{align*}
\sum\limits_{k\in\mathbb{Z}}\lf(f_n^{\nu^{k+1}}-f_n^{\nu^k}\r)
=f_n \quad \quad {\rm almost\ everywhere}.
\end{align*}
For any $k\in\mathbb{Z},$ let
$$\mu^k:=2^{k+1}\lf\|\mathbf{1}_{\{\nu^k<\infty\}}\r\|_{L^{\varphi}(\Omega)}.$$
Moreover, for any $k\in\mathbb{Z}$ and $n\in\mathbb{Z}_{+},$ if $\mu^k\neq0,$ let
$$a_n^k:=\frac{f_n^{\nu^{k+1}}-f_n^{\nu^k}}{\mu^k};$$
otherwise, let $a_n^k:=0$. We first show that, for any fixed
$k\in\mathbb{Z}$, $a^k:=(a_n^k)_{n\in\mathbb{Z}_{+}}$ is a $(\varphi,\fz)_s$-atom.
By the definition of $a^k,$ it is easy to know that $a^k$ is a martingale.
Then, for any $n\in\mathbb{Z}_{+}$, when $\nu^k\geq n$, by the definition of $f^{\nu^k}$,
we have
\begin{align}\label{atom pf}
a_n^k=\frac{f_n^{\nu^{k+1}}-f_n^{\nu^k}}{\mu^k}=\frac{f_n-f_n}{\mu^k}=0.
\end{align}
Thus, $a^k$ satisfies Definition \ref{def atom}(i). From \eqref{atom pf}, we further
deduce that
$$\mathbf{1}_{\{\nu^k=\infty\}}\lf[s(a^k)\r]^2
\le\sum_{n=1}^{\infty}\mathbf{1}_{\{\nu^k\geq n\}}\mathbb{E}_{n-1}\lf(\lf|d_na^k\r|^2\r)
=\sum_{n=1}^{\infty}\mathbb{E}_{n-1}\lf(\lf|\mathbf{1}_{\{\nu^k\geq n\}}d_na^k\r|^2\r)=0.$$
Thus, we have
\begin{align*}
\lf\{x\in\Omega:\ s(a^k)(x)\not=0\r\}\subset B_{v^k}.
\end{align*}
Moreover, by the definition of $\nu^k$, we obtain
$$\lf[s(a^k)\r]^2=\sum_{n=1}^{\infty}\mathbb{E}_{n-1}\lf(\lf|d_na^k\r|^2\r)
\le\lf[\frac{s_{\nu^{k+1}}(f)}{\mu^k}\r]^2
\le\lf(\frac{2^{k+1}}{\mu^k}\r)^2,$$
which implies that $(a_n^k)_{n\in\mathbb{Z}_{+}}$ is an $L^2$-bounded martingale
and hence $(a_n^k)_{n\in\mathbb{Z}_{+}}$ converges in $L^2(\Omega)$.
Denoting the limit still by $a^k$, then, for any $n\in\zz_{+}$, $\mathbb{E}_n(a^k)=a_n^k$.
For any $n\in\zz_{+}$ and $n\le\nu^k$, we know that $a_n^k=0$ and
$\|s(a^k)\|_{L^{\infty}(\Omega)}\le\frac{2^{k+1}}{\mu^k}.$
Therefore,
$$\lf\|s(a^k)\r\|_{L^{\fz}(B_{\nu^k})}\le\frac{2^{k+1}}{\mu^k}
=\lf\|\mathbf{1}_{B_{\nu^k}}\r\|_{L^{\varphi}(\Omega)}^{-1}.$$
Thus, $a^k$ satisfies Definition \ref{def atom}(ii) and we further conclude
that $a^k$ is a $(\varphi,\fz)_s$-atom.

We now show that $\{\mu^k,a^k,\nu^k\}_{k\in\mathbb{Z}}\in\mathcal{A}_s(\varphi,\fz)$.
To this end, we first notice that
\begin{align}\label{p0}
\sum_{k\in\mathbb{Z}}\left[\frac{\mu^k\mathbf{1}_{B_{\nu^k}}}
{\|\mathbf{1}_{B_{\nu^k}}\|_{L^{\varphi}(\Omega)}}\right]^r
=\sum_{k\in\mathbb{Z}}\left(2^{k+1}\mathbf{1}_{B_{\nu^k}}
\right)^r.
\end{align}
For any $k\in\zz$, by the definition of $\nu^k$, we find that
$B_{\nu^k}=\{x\in\Omega:\ s(f)(x)>2^k\}$ and
$B_{\nu^k}\supset B_{\nu^{k+1}}.$ Let $G_k:=B_{\nu^k}\backslash B_{\nu^{k+1}},$
then $G_k:=\{x\in\Omega:\ 2^k<s(f)(x)\le2^{k+1}\}$.
Now we claim that, for any given $r\in\D$,
\begin{align}\label{p1}
\sum_{k\in\mathbb{Z}}\lf(2^{k+1}\mathbf{1}_{B_{\nu^k}}\r)^r
=\frac{2^r}{2^r-1} \lf(\sum_{k\in\mathbb{Z}}2^{k+1}\mathbf{1}_{G_k}\r)^r.
\end{align}
Indeed, for any $k\in\zz$, $B_{\nu^k}=\bigcup_{j=k}^{\fz}G_j$. From this and
the fact that $\{G_k\}_{k\in\zz}$ are disjoint, it follows that,
for any $k\in\zz$,
$$\mathbf{1}_{B_{\nu^k}}=\sum_{j=k}^{\fz}\mathbf{1}_{G_j},$$
which, implies that, for any given $r\in\D$,
\begin{align*}
\sum_{k\in\mathbb{Z}}\lf(2^{k+1}\mathbf{1}_{B_{\nu^k}}\r)^r
=\sum_{k\in\mathbb{Z}}2^{(k+1)r}\lf(\sum_{j=k}^{\fz}\mathbf{1}_{G_j}\r)
=\sum_{j\in\mathbb{Z}}\sum_{k\le j}2^{(k+1)r}\mathbf{1}_{G_j}
=\frac{2^r}{2^r-1}\sum_{k\in\mathbb{Z}}2^{(k+1)r}\mathbf{1}_{G_k}.
\end{align*}
Combining this and the fact that $\{G_k\}_{k\in\zz}$ are disjoint,
we obtain \eqref{p1}. This proves the above claim.
From this claim, the definition of $G_k$ and \eqref{p0}, we deduce that,
for any given $r\in\D$ and any $\lambda\in\D$,
\begin{align}\label{e11}
\int_{\Omega}\V\lf(x,\frac1{\lambda}\left\{\sum_{k\in\mathbb{Z}}
\left[\frac{\mu^k\mathbf{1}_{B_{\nu^k}}(x)}
{\|\mathbf{1}_{B_{\nu^k}}\|_{L^{\varphi}(\Omega)}}\right]^r\right\}^{\frac1r}\r)\,d\MP(x)
&=\int_{\Omega}\V\lf(x,\frac1{\lambda}\lf(\frac{2^r}{2^r-1}\r)^{\frac1r}
\sum_{k\in\mathbb{Z}}2^{k+1}\mathbf{1}_{G_k}(x)\r)\,d\MP(x)\\
&\le \sum_{k\in\mathbb{Z}}\int_{G_k}\V\lf(x,2
\lf(\frac{2^r}{2^r-1}\r)^{\frac1r}\frac{s(f)(x)}{\lambda}\r)\,d\MP(x)\noz\\
&\le\int_{\Omega}\V\lf(x,2
\lf(\frac{2^r}{2^r-1}\r)^{\frac1r}\frac{s(f)(x)}{\lambda}\r)\,d\MP(x)\noz.
\end{align}
For any given $r\in\D$, letting $\lambda:=2(\frac{2^r}{2^r-1})^{\frac1r}\|f\|_{H_{\V}^s(\Omega)}$,
then we obtain
$$\lf\|\left\{\sum_{k\in\mathbb{Z}}\left[\frac{\mu^k\mathbf{1}_{B_{\nu^k}}(x)}
{\|\mathbf{1}_{B_{\nu^k}}\|_{L^{\varphi}(\Omega)}}\right]^r\right\}^{\frac1r}\r\|_{L^{\V}(\Omega)}
\le 2\lf(\frac{2^r}{2^r-1}\r)^{\frac1r}\|f\|_{H_{\V}^s(\Omega)}.$$
This implies that, for any given $r\in\D$, $f\in H_{{\rm at},\,r}^{\V,\fz,s}(\Omega)$ and
$\|f\|_{H_{{\rm at},\,r}^{\V,\fz,s}(\Omega)}\lesssim \|f\|_{H_{\V}^s(\Omega)}.$
The proof of Step $1)$ is now complete.

Step $2)$ Prove $H_{{\rm at},\,r}^{\V,\fz,s}(\Omega)\subset H_{\V}^s(\Omega)$. To show this,
for any given $r\in(0,1]$, let $f\in H_{{\rm at},\,r}^{\V,\fz,s}(\Omega).$
Then we can write
$f=\sum_{k\in\zz}\mu^ka^k$ for some sequence of triples,
$$\lf\{\mu^k,a^k,\nu^k\r\}_{k\in\mathbb{Z}}\in\mathcal{A}_s(\varphi,\fz).$$
By Definition \ref{def atom}(i), we know that, for any $k\in\mathbb{Z},$
$$\lf\{x\in\Omega:\ s(a^k)(x)\not=0\r\}\subset B_{v^k}.$$
Thus, for almost every $x\in\Omega$,
we have
$$s(a^k)(x)\le \lf\|s(a^k)\r\|_{L^{\fz}(\Omega)}\mathbf{1}_{B_{v^k}}(x)
\le \mathbf{1}_{B_{v^k}}(x)\lf\|\mathbf{1}_{B_{v^k}}\r\|_{L^{\varphi}(\Omega)}^{-1}.$$
Combining this and the subadditive of the operator $s$, we find that,
for any given $r\in(0,1]$ and $\lambda\in\D$,
\begin{align*}
\int_{\Omega}\V\lf(x,\frac{s(f)(x)}{\lambda}\r)\,d\MP(x)
&\le \int_{\Omega}\V\lf(x,\sum_{k\in\zz}\frac{\mu^ks(a^k)(x)}{\lambda}\r)\,d\MP(x)
\le \int_{\Omega}\V\lf(x,\sum_{k\in\zz}
\frac{\mu^k\mathbf{1}_{B_{v^k}}(x)}{\lambda\|
\mathbf{1}_{B_{v^k}}\|_{L^{\varphi}(\Omega)}}\r)\,d\MP(x)\\
&\le\int_{\Omega}\V\lf(x,\frac1{\lambda}\left\{\sum_{k\in\mathbb{Z}}
\left[\frac{\mu^k\mathbf{1}_{B_{\nu^k}}(x)}
{\|\mathbf{1}_{B_{\nu^k}}\|_{L^{\varphi}(\Omega)}}\right]^r\right\}^{\frac1r}\r)\,d\MP(x).
\end{align*}
Therefore, for any given $r\in(0,1]$, we conclude that $f\in H_{\V}^s(\Omega)$ and
$\|f\|_{H_{\V}^s(\Omega)}\le \|f\|_{H_{{\rm at},\,r}^{\V,\fz,s}(\Omega)}.$
This finishes the proof of Theorem \ref{Thm-atms}.
\end{proof}

\begin{remark}
Clearly, Theorem \ref{Thm-atms} does not need the assumptions that $\V$ is
of uniformly lower type $p$ for some $p\in(0,1]$ and of uniformly upper type $1$.
This is different from \cite[Theorem 1.4]{xjy18} which need the both assumptions.
\end{remark}

\begin{remark}\label{r1}
If $\V$ is of uniformly upper type $p$ for some $p\in\D$
and $\V(x,\cdot)$ is right-continuous at $0$ for almost
every $x\in\Omega$ in Theorem \ref{Thm-atms}, then the sum
$\sum_{k=l}^m\mu^ka^k$ converges to $f$ in $H_{\varphi}^s(\Omega)$ as $m\to\fz$, $\ell\to-\fz$.
Indeed, for any $m,\ \ell\in\mathbb{Z}$ and $\ell<m$, we obtain
$$
f-\sum_{k=l}^m\mu^ka^k=\lf(f-f^{\nu^{m+1}}\r)+f^{\nu^l} \quad {\rm and} \quad \lf[s\lf(f-f^{\nu^{m+1}}\r)\r]^2=\lf[s\lf(f\r)\r]^2-\lf[s\lf(f^{\nu^{m+1}}\r)\r]^2.
$$
From this, we deduce that, for any $m$, $\ell\in\mathbb{Z}$ with $\ell<m$ and $x\in\Omega$,
\begin{align}\label{C}
\V\lf(x,s\lf(f-\sum_{k=l}^m\mu^ka^k\r)(x)\r)
&\le \V\lf(x,s\lf(f-f^{\nu^{m+1}}\r)(x)+s\lf(f^{\nu^l}\r)(x)\r)\\
&\le \V\lf(x,2s\lf(f-f^{\nu^{m+1}}\r)(x)\r)+\V\lf(x,2s\lf(f^{\nu^l}\r)(x)\r)\noz.
\end{align}
In addition, for any $m$, $\ell\in\mathbb{Z}$ and almost every $x\in\Omega$,
$$\lim_{m\to\infty}s\lf(f-f^{\nu^{m+1}}\r)(x)=0\quad \mbox{and} \quad
\lim_{l\to-\infty}s\lf(f^{\nu^l}\r)(x)=0.$$
Combining this and $\V(x,\cdot)$ is right-continuous at $0$ for almost
every $x\in\Omega$, we conclude that, for any $m$, $\ell\in\mathbb{Z}$ and
almost every $x\in\Omega$,
\begin{align}\label{C1}\lim_{m\to\infty}\V\lf(x,2s\lf(f-f^{\nu^{m+1}}\r)(x)\r)=0
\quad \mbox{and} \quad \ \lim_{l\to-\infty}\V\lf(x,2s\lf(f^{\nu^l}\r)(x)\r)=0.
\end{align}
Also, we find that, for any $m$, $\ell\in\mathbb{Z}$,
$$s\lf(f-f^{\nu^{m+1}}\r)\le s(f)\quad \mbox{and} \quad s\lf(f^{\nu^l}\r)\le s(f),$$
which, together with \eqref{C}, \eqref{C1}, the fact that $\V$ is of
uniformly upper type $p$ for some $p\in\D$ and
the Lebesgue dominated convergence theorem, implies that the
series $\sum_{k=l}^m\mu^ka^k$ converges to
$f$ in $H_{\varphi}^s(\Omega)$ norm as $m\to\infty,$ $l\to-\infty.$
\end{remark}

Similarly, we also obtain the following atomic characterizations of
$P_{\V}(\Omega)$ and $Q_{\V}(\Omega)$.
\begin{theorem}\label{thm-atmpq}
Let $\varphi$ be a Musielak--Orlicz function.
Then, for any given $r\in(0,1]$,
$$P_{\V}(\Omega)=H_{{\rm at},\,r}^{\V,\fz,M}(\Omega)
\quad  and \quad Q_{\V}(\Omega)=H_{{\rm at},\,r}^{\V,\fz,S}(\Omega)
\qquad \mbox{ with equivalent quasi-norms}.$$
\end{theorem}
\begin{proof}
The proof of Theorem \ref{thm-atmpq} is similar to that of
Theorem \ref{Thm-atms}. So we just give the outline and omit the details.
We only give the proof for $P_{\V}(\Omega)$ because it is just slightly different
from the one for $Q_{\V}(\Omega)$.

Let $f\in P_{\V}(\Omega)$ and $(\lambda_n)_{n\in\mathbb{N}}$
be an adapted non-decreasing sequence satisfying that, for any $n\in\mathbb{N}$,
$|f_n|\le\lambda_{n-1}$ and $\lambda_{\infty}\in L^{\varphi}(\Omega).$
The stopping times, in these cases, are defined by setting, for any $k\in\mathbb{Z},$
$$\nu^k:=\inf\lf\{n\in\mathbb{N}:\ \lambda_n>2^k\r\}\ \ \  (\inf\emptyset=\infty).$$
For any $k\in\mathbb{Z},$ let $a^k$ and $\mu^k$ be defined as
in the proof of Theorem \ref{Thm-atms}.
Then, similarly to the proof of Theorem \ref{Thm-atms}, we easily know that $a^k$
is a $(\varphi,\infty)_M$-atom. By the argument same as in the proof of Theorem \ref{Thm-atms},
we can prove that, for any given $r\in(0,1]$,
\begin{align}\label{P2}
\lf\|\left\{\sum_{k\in\mathbb{Z}}\left[\frac{\mu^k\mathbf{1}_{B_{\nu^k}}(x)}
{\|\mathbf{1}_{B_{\nu^k}}\|_{L^{\varphi}(\Omega)}}\right]^r\right\}^{\frac1r}\r\|_{L^{\V}(\Omega)}
\le 2\lf(\frac{2^r}{2^r-1}\r)^{\frac1r}\|\lambda_{\fz}\|_{L^{\V}(\Omega)}.
\end{align}
This implies that, for any given $r\in(0,1]$,
$\|f\|_{H_{{\rm at},\,r}^{\V,\fz,M}(\Omega)}\lesssim \|f\|_{P_{\V}(\Omega)}$
and $P_{\V}(\Omega)\subset H_{{\rm at},\,r}^{\V,\fz,M}(\Omega).$

Conversely, for any given $r\in(0,1]$, assume that $f\in H_{{\rm at},\,r}^{\V,\fz,M}(\Omega)$.
Then we can write
$f=\sum_{k\in\zz}\mu^ka^k$ for some sequence of triples,
$$\lf\{\mu^k,a^k,\nu^k\r\}_{k\in\mathbb{Z}}\in\mathcal{A}_M(\varphi,\fz).$$
For any $n\in\zz_{+}$, let
$$\lambda_n:=\sum_{k\in\zz}\mu^k\mathbf{1}_{\{x\in\Omega:\ \nu^{k}(x)\le n\}}
\lf\|M(a^k)\r\|_{L^{\fz}(\Omega)}.$$
Then $\{\lambda_n\}_{n\in\zz_{+}}$ is a nonnegative adapted sequence and,
for any $n\in\zz_{+}$, $|f_n|\le \lambda_{n-1}$.
Thus, for any given $r\in(0,1]$, we have
\begin{align*}
\|f\|_{P_{\V}(\Omega)}\le\|\lambda_{\fz}\|_{L^{\V}(\Omega)}
\le \lf\|\sum_{k\in\zz}\frac{\mu^k\mathbf{1}_{B_{\nu^k}}}
{\|\mathbf{1}_{B_{\nu^k}}\|_{L^{\varphi}(\Omega)}}\r\|_{L^{\V}(\Omega)}
\le\lf\|\left\{\sum_{k\in\mathbb{Z}}\left[\frac{\mu^k\mathbf{1}_{B_{\nu^k}}(x)}
{\|\mathbf{1}_{B_{\nu^k}}\|_{L^{\varphi}(\Omega)}}\right]^r\right\}^{\frac1r}\r\|_{L^{\V}(\Omega)}.
\end{align*}
From this, it follows that, for any given $r\in(0,1]$, $f\in P_{\V}(\Omega)$ and
$\|f\|_{P_{\V}(\Omega)}\le \|f\|_{H_{{\rm at},\,r}^{\V,\fz,M}(\Omega)}$.
This finishes the proof of Theorem \ref{thm-atmpq}.
\end{proof}
\begin{remark}
If $\V$ is of uniformly upper type $p$ for some $p\in\D$
and $\V(x,\cdot)$ is right-continuous at $0$ for almost
every $x\in\Omega$, then
we claim that the sum $\sum_{k=l}^m\mu^ka^k$ converges to $f$
in $P_{\varphi}(\Omega)$ as $l\to-\infty$ and $m\to\infty.$
Indeed, for any $n\in\nn$ and $\ell$, $m\in\mathbb{Z}$, let
$$\lambda_{n,1}:=\sum_{k=m+1}^{\infty}\mu^k
\mathbf{1}_{\{x\in\Omega:\ \nu^{k}(x)\le n\}}\lf\|M(a^k)\r\|_{L^{\fz}(\Omega)}\quad \mathrm{and}
\quad \lambda_{n,2}:=\sum_{k=-\fz}^{\ell}\mu^k
\mathbf{1}_{\{x\in\Omega:\ \nu^{k}(x)\le n\}}\lf\|M(a^k)\r\|_{L^{\fz}(\Omega)}.$$
Then $\{\lambda_{n,1}\}_{n\in\nn}$ and $\{\lambda_{n,2}\}_{n\in\nn}$ are
nonnegative adapted sequences and, for any $n\in\nn$ and $\ell$, $m\in\mathbb{Z}$,
\begin{align*}\lf|f_n-\sum_{k=l}^m\mu^ka^k_n\r|&=\lf|f_n-f^{\nu^{m+1}}_n+f^{\nu^{\ell}}_n\r|
\le\sum_{k=m+1}^{\infty}\lf|f^{\nu^{k+1}}_n-f^{\nu^{k}}_n\r|
+\sum_{k=-\fz}^{\ell}\lf|f^{\nu^{k+1}}_n-f^{\nu^{k}}_n\r|\\
&=\sum_{k=m+1}^{\infty}\mu^k\mathbf{1}_{\{x\in\Omega:\ \nu^{k}(x)\le n-1\}}\lf|a_n^k\r|
+\sum_{k=-\fz}^{\ell}\mu^k\mathbf{1}_{\{x\in\Omega:\ \nu^{k}(x)\le n-1\}}\lf|a_n^k\r|
\le \lambda_{n-1,1}+\lambda_{n-1,2}.
\end{align*}
Using this, the analogue of \eqref{C} and \eqref{C1},
we obtain
$$
\left\|f-\sum_{k=l}^m\mu^ka^k \right\|_{P_{\varphi}(\Omega)}\to0 \quad \mbox{as $l\to-\infty$ and $m\to\infty,$}
$$ which completes the proof of the above claim.
\end{remark}

The following atomic characterizations of Musielak--Orlicz Hardy spaces
$H_{\V}^{M}(\Omega)$ and $H_{\V}^S(\Omega)$ are new even for classical
martingale Hardy spaces. For the dyadic stochastic basis, it was proved by
Weisz \cite{wk2}.
\begin{theorem}\label{Thm-atmSM}
Let $\varphi$ be a Musielak--Orlicz function and of uniformly lower type
$p\in(0,\fz)$. If $\V\in \mathbb{S}^{-}$ and the stochastic basis is regular,
then, for any given $r\in(0,1]$,
$$H_{\V}^{M}(\Omega)=H_{{\rm at},\,r}^{\V,\fz,M}(\Omega)
\quad and \quad H_{\V}^S(\Omega)=H_{{\rm at},\,r}^{\V,\fz,S}(\Omega)
\qquad \mbox{ with equivalent quasi-norms}.$$
\end{theorem}

To show Theorem \ref{Thm-atmSM}, we need the next technical lemma, which
might have some independent interest.

\begin{lemma}\label{st}
Let $w$ be a special weight and $w:=(w_n)_{n\in\zz_{+}}\in \mathbb{S}^{-}$.
If the stochastic basis $\{\mathcal{F}_n\}_{n\in\mathbb{Z}_{+}}$ is regular,
then, for any given nonnegative adapted process $\gamma=(\gamma_n)_{n\in\mathbb{Z}_{+}}$
and $\lambda\in(\|\gamma_0\|_{L^{\fz}(\Omega)},\fz)$, there exists a stopping
time $\tau_{\lambda}$ such that
\begin{align}\label{st2}
\lf\{x\in\Omega:\ M(\gamma)(x)>\lambda\r\}\subset \lf\{x\in\Omega:\ \tau_{\lambda}(x)<\fz\r\},
\end{align}
\begin{align}\label{st3}
\sup_{n\le\tau_{\lambda}(x)}\gamma_n(x)=:M_{\tau_{\lambda}}(\gamma)(x)\le\lambda,
\quad\quad\forall\,x\in\Omega\end{align}
and \begin{align}\label{st4}
w\lf(\lf\{x\in\Omega:\ \tau_{\lambda}(x)<\fz\r\}\r)\le RK
w\lf(\lf\{x\in\Omega:\ M(\gamma)(x)>\lambda\r\}\r),
\end{align}
where $R$ and $K$ are the same as in \eqref{WS} and \eqref{R}, respectively.
Moreover, for any $\lambda_1$, $\lambda_2\in\D$ with $\lambda_1<\lambda_2$,
$\tau_{\lambda_1}\le \tau_{\lambda_2}$.
\end{lemma}
\begin{proof}
Let $\lambda\in(\|\gamma_0\|_{L^{\fz}(\Omega)},\fz)$ and
$\gamma=(\gamma_n)_{n\in\mathbb{Z}_{+}}$ be any nonnegative
adapted process. For any $n\in\nn$, let
$$G_n:=\lf\{x\in\Omega:\ \frac1{w_{n-1}(x)}\ee\lf(\mathbf{1}_{\{x\in\Omega:\ \gamma_{n}(x)>\lambda\}}w|\cf_{n-1}\r)(x)>\frac1{RK}\r\}$$
and $$\tau_{\lambda}(x):=\inf\lf\{n\in\zz_{+}:\ x\in G_{n+1}\r\},\quad\forall\, x\in\Omega.$$
Clearly, for any $n\in\nn$, $G_n\in\cf_{n-1}$ and hence $\tau_{\lambda}$ is a stopping time.
We claim that \eqref{st2} holds true. Indeed, by $w\in \mathbb{S}^{-}$, the regularity and
the fact that $(\gamma_n)_{n\in\mathbb{Z}_{+}}$
is adapted, we find that, for any $n\in\nn$,
$$\mathbf{1}_{\{x\in\Omega:\ \gamma_n(x)>\lambda\}}=
\frac1{w_{n}}\ee\lf(\mathbf{1}_{\{x\in\Omega:\ \gamma_{n}(x)>\lambda\}}w|\cf_{n}\r)
\le \frac{K}{w_{n-1}}R\ee\lf(\mathbf{1}_{\{x\in\Omega:\ \gamma_{n}(x)>\lambda\}}w|\cf_{n-1}\r).$$
From this, it follows that, for any $n\in\nn$,
\begin{align}\label{gm}
\lf\{x\in\Omega:\ \gamma_n(x)>\lambda\r\}\subset G_n.
\end{align}
Since $\lambda>\|\gamma_0\|_{L^{\fz}(\Omega)}$, it follows that, for any $x\in\lf\{x\in\Omega:\ M(\gamma)(x)>\lambda\r\}$, there exists $n\in\zz_{+}$ such that
$$x\in\lf\{x\in\Omega:\ \gamma_{n+1}(x)>\lambda\r\}.$$
Combining this and \eqref{gm}, we obtain $x\in G_{n+1}.$ This implies that
$\tau_{\lambda}(x) \leq n$, and hence \eqref{st2} holds true. This proves the above claim.

Now we show \eqref{st3}. When $x\in\{x\in\Omega:\ \tau_{\lambda}(x)=0\}$, \eqref{st3} follows from the assumption $\lambda>\|\gamma_0\|_{L^{\fz}(\Omega)}$. When
$x\in\{x\in\Omega:\ \tau_{\lambda}(x)=\fz\}$, by \eqref{st2}, we know that
$$\lf\{x\in\Omega:\ \tau_{\lambda}(x)=\fz\r\}\subset \lf\{x\in\Omega:\ M(\gamma)(x)\le\lambda\r\}.$$
In the remaining case $x\in\{x\in\Omega:\ 0<\tau_{\lambda}(x)<\fz\}$,
there exists a positive integer $n$ such that $\tau_{\lambda}(x)=n$. From
the definition of $\tau_{\lambda}$, it follows that
$x\notin\bigcup_{m=1}^n G_m.$
Combining this and \eqref{gm}, we obtain
$$x\in\bigcup_{m=1}^n \lf\{x\in\Omega:\ \gamma_n(x)\le\lambda\r\}.$$
This implies that $M_{\tau_{\lambda}}(\gamma)(x)\le\lambda.$
Therefore, \eqref{st3} holds true.

Next we prove \eqref{st4}. From the definition of $\tau_{\lambda}$ and $\{x\in\Omega:\ \tau_{\lambda}(x)=n\}\in\cf_n$, it follows that,
for any $n\in\zz_{+}$,
\begin{align*}
\lf\{x\in\Omega:\ \tau_{\lambda}(x)=n\r\}
&=\lf\{x\in\Omega:\ \tau_{\lambda}(x)=n\r\}\cap G_{n+1}\\
&=\lf\{x\in\Omega:\ \frac1{w_{n}(x)}\ee
\lf(\mathbf{1}_{\{x\in\Omega:\ \gamma_{n+1}(x)>\lambda\}}w|\cf_{n}\r)(x)
\mathbf{1}_{\{x\in\Omega:\ \tau_{\lambda}(x)=n\}}(x)>\frac1{RK}\r\}\\
&=\lf\{x\in\Omega:\ \frac1{w_{n}(x)}\ee
\lf(\mathbf{1}_{\{x\in\Omega:\ \gamma_{n+1}(x)>\lambda\}
\cap\{x\in\Omega:\ \tau_{\lambda}(x)=n\}}w|\cf_{n}\r)(x)>\frac1{RK}\r\},
\end{align*}
which, together with \eqref{st2} and the fact that, for any $n,$ $m\in\zz_{+}$ and $n\neq m$,
$$
\{x\in\Omega:\ \tau_{\lambda}(x)=n\}\cap\{x\in\Omega:\ \tau_{\lambda}(x)=m\}=\varnothing,$$
implies that
\begin{align*}
w\lf(\lf\{x\in\Omega:\ \tau_{\lambda}(x)<\fz\r\}\r)
&=\sum_{n\in\zz_{+}}w\lf(\lf\{x\in\Omega:\ \tau_{\lambda}(x)=n\r\}\r)\\
&\le\sum_{n\in\zz_{+}}RK\int_{\Omega}\frac{1}{w_{n}(x)}
\ee\lf(\mathbf{1}_{\{x\in\Omega:\ \gamma_{n+1}(x)>\lambda\}
\cap\{x\in\Omega:\ \tau_{\lambda}(x)=n\}}w|\cf_{n}\r)(x)w(x)\,d\MP(x)\\
&=\sum_{n\in\zz_{+}}RK\int_{\Omega}\frac{1}{w_{n}(x)}
\ee\lf(\mathbf{1}_{\{x\in\Omega:\ \gamma_{n+1}(x)>\lambda\}
\cap\{x\in\Omega:\ \tau_{\lambda}(x)=n\}}w|\cf_{n}\r)(x)w_n(x)\,d\MP(x)\\
&=\sum_{n\in\zz_{+}}RK w\lf(\{x\in\Omega:\ \gamma_{n+1}(x)>\lambda\}
\cap\{x\in\Omega:\ \tau_{\lambda}(x)=n\}\r)\\
&\le\sum_{n\in\zz_{+}}RK w\lf(\{x\in\Omega:\ M(\gamma)(x)>\lambda\}
\cap\{x\in\Omega:\ \tau_{\lambda}(x)=n\}\r)\\
&=RK w\lf(\{x\in\Omega:\ M(\gamma)(x)>\lambda\}\r).
\end{align*}
Thus, \eqref{st4} holds true.

Finally, for any $\lambda_1$, $\lambda_2\in\D$ with $\lambda_1<\lambda_2$,
we know that, for any $n\in\nn,$
$$\ee\lf(\mathbf{1}_{\{x\in\Omega:\ \gamma_{n}(x)>\lambda_1\}}w|\cf_{n-1}\r)
\geq\ee\lf(\mathbf{1}_{\{x\in\Omega:\ \gamma_{n}(x)>\lambda_2\}}w|\cf_{n-1}\r).$$
Combining this and the definitions of $G_n$ and $\tau_{\lambda}$, we
obtain $\tau_{\lambda_1}\le\tau_{\lambda_2}.$ This finishes the proof of
Lemma \ref{st}.
\end{proof}

Now we are ready to prove Theorem \ref{Thm-atmSM}.
\begin{proof}[Proof of Theorem \ref{Thm-atmSM}]

We prove this theorem only for $H_{\V}^{M}(\Omega)$
because the proof for $H_{\V}^S(\Omega)$ is similar. We do this by two steps.

Step $1)$ Prove $H_{\V}^{M}(\Omega)\subset H_{{\rm at},\,r}^{\V,\fz,M}(\Omega)$.
To this end,
let $f:=(f_n)_{n\in\mathbb{Z}_{+}}\in H_{\V}^{M}(\Omega)$.
For any $k\in\mathbb{Z}$ and nonnegative adapted sequence $(|f_n|)_{n\in\mathbb{Z}_{+}}$,
by Lemma \ref{st}, we know that there exists a stopping time $\nu^k$ such that
\begin{align*}
\lf\{x\in\Omega:\ M(f)(x)>2^k\r\}\subset \lf\{x\in\Omega:\ \nu^k(x)<\fz\r\} =:B_{v^k},
\end{align*}
\begin{align}\label{m1}
M_{\nu^k}(f)(x)\le2^k,
\quad\quad\forall\,x\in\Omega\end{align}
and \begin{align}\label{m2}
\V\lf(\lf\{x\in\Omega:\ \nu^k(x)<\fz\r\},t\r)\le RK
\V\lf(\lf\{x\in\Omega:\ M(f)(x)>2^k\r\},t\r), \quad\forall\,t\in\D.
\end{align}
Moreover, the sequence $\{\nu^k\}_{k\in\zz}$ of stopping times is non-decreasing
and $\nu^k\to\infty$ as $k\to\infty$.
Similarly to the proof of Theorem \ref{Thm-atms},
for any $k\in\mathbb{Z},$ let
$$\mu^k:=3\times2^{k}\lf\|\mathbf{1}_{B_{\nu^k}}\r\|_{L^{\varphi}(\Omega)}.$$
Moreover, for any $k\in\mathbb{Z}$ and $n\in\mathbb{Z}_{+},$ if $\mu^k\neq0,$ let
$$a_n^k:=\frac{f_n^{\nu^{k+1}}-f_n^{\nu^k}}{\mu^k};$$
otherwise, let $a_n^k:=0$. We first show that, for any fixed
$k\in\mathbb{Z}$, $a^k:=(a_n^k)_{n\in\mathbb{Z}_{+}}$ is a $(\varphi,\fz)_M$-atom.
Clearly, $a^k$ is a martingale and
\begin{align}\label{supp}
\lf\{x\in\Omega:\ M(a^k)(x)\not=0\r\}\subset B_{v^k}.
\end{align}
By \eqref{m1}, we have, for any $n\in\mathbb{Z}_{+}$,
$$\lf|a_n^k\r|\le \frac{|f_n^{\nu^{k+1}}|+|f_n^{\nu^k}|}{\mu^k}
\le \frac1{\mu^k}\lf[M_{\nu^{k+1}}(f)+M_{\nu^{k}}(f)\r]
\le \lf\|\mathbf{1}_{B_{\nu^k}}\r\|_{L^{\varphi}(\Omega)}^{-1},$$
which implies that
\begin{align}\label{atii}
\lf\|M(a^k)\r\|_{L^{\fz}(B_{\nu^k})}
\le \|\mathbf{1}_{B_{\nu^k}}\|_{L^{\varphi}(\Omega)}^{-1}.
\end{align}
Thus, for any $p\in(1,\fz)$, $(a_n^k)_{n\in\mathbb{Z}_{+}}$ is an $L^p$-bounded martingale
and hence $(a_n^k)_{n\in\mathbb{Z}_{+}}$ converges in $L^p(\Omega)$.
Denoting this limit still by $a^k$, then $\mathbb{E}_n(a^k)=a_n^k$.
For any $n\in\mathbb{Z}_{+}$, if $\nu^k\geq n$, by the definition of $f^{\nu^k}$,
we have
$$
a_n^k=\frac{f_n^{\nu^{k+1}}-f_n^{\nu^k}}{\mu^k}=\frac{f_n-f_n}{\mu^k}=0.
$$ Combining this and \eqref{atii}, we conclude
that $a^k$ is a $(\varphi,\fz)_M$-atom.

Now we show that $\{\mu^k,a^k,\nu^k\}_{k\in\mathbb{Z}}\in\mathcal{A}_M(\varphi,\fz)$.
To this end, for any $k\in\zz$, let $G_k:=B_{\nu^k}\backslash B_{\nu^{k+1}}.$
Similarly to the proof of \eqref{p1}, we find that, for any given $r\in\D$,
\begin{align*}
\sum_{k\in\mathbb{Z}}\left[\frac{\mu^k\mathbf{1}_{B_{\nu^k}}}
{\|\mathbf{1}_{B_{\nu^k}}\|_{L^{\varphi}(\Omega)}}\right]^r
=\frac{6^r}{2^r-1}\sum_{k\in\mathbb{Z}}\left(2^{k}\mathbf{1}_{G_k}
\right)^r.
\end{align*}
Combining this, \eqref{m2} and the fact that $\{G_k\}_{k\in\zz}$ are disjoint, we know that,
for any given $r\in\D$ and any $\lambda\in\D$,
\begin{align*}
\int_{\Omega}\V\lf(x,\frac1{\lambda}\left\{\sum_{k\in\mathbb{Z}}
\left[\frac{\mu^k\mathbf{1}_{B_{\nu^k}}(x)}
{\|\mathbf{1}_{B_{\nu^k}}\|_{L^{\varphi}(\Omega)}}\right]^r\right\}^{\frac1r}\r)\,d\MP(x)
&=\int_{\Omega}\V\lf(x,\frac1{\lambda}\lf(\frac{6^r}{2^r-1}\r)^{\frac1r}
\sum_{k\in\mathbb{Z}}2^{k}\mathbf{1}_{G_k}(x)\r)\,d\MP(x)\\
&= \sum_{k\in\mathbb{Z}}\int_{G_k}\V\lf(x,\lf(\frac{6^r}{2^r-1}\r)^{\frac1r}
\frac{2^k}{\lambda}\r)\,d\MP(x)\\
&\le \sum_{k\in\mathbb{Z}}\int_{B_{\nu^k}}\V\lf(x,\lf(\frac{6^r}{2^r-1}\r)^{\frac1r}
\frac{2^k}{\lambda}\r)\,d\MP(x)\\
&\le RK\sum_{k\in\mathbb{Z}}\int_{\{x\in\Omega:\ M(f)(x)>2^k\}}
\V\lf(x,\lf(\frac{6^r}{2^r-1}\r)^{\frac1r}
\frac{2^k}{\lambda}\r)\,d\MP(x)=:{\rm I}
\end{align*}
Since $\V$ is of uniformly lower type $p\in\D$, we know that,
for any given $r\in\D$ and any $\lambda\in\D$,
\begin{align*}
{\rm I}&\lesssim \sum_{k\in\mathbb{Z}}\sum_{j=k}^{\fz}
\int_{\{x\in\Omega:\ 2^j<M(f)(x)\le2^{j+1}\}}
\V\lf(x,\lf(\frac{6^r}{2^r-1}\r)^{\frac1r}\frac{2^k}{\lambda}\r)\,d\MP(x) \\
&\lesssim \sum_{k\in\mathbb{Z}}\sum_{j=k}^{\fz}2^{(k-j)p}
\int_{\{x\in\Omega:\ 2^j<M(f)(x)\le2^{j+1}\}}
\V\lf(x,\lf(\frac{6^r}{2^r-1}\r)^{\frac1r}\frac{2^j}{\lambda}\r)\,d\MP(x)\\
&\lesssim \sum_{j\in\mathbb{Z}}\sum_{k\le j}2^{(k-j)p}
\int_{\{x\in\Omega:\ 2^j<M(f)(x)\le2^{j+1}\}}
\V\lf(x,\lf(\frac{6^r}{2^r-1}\r)^{\frac1r}\frac{M(f)(x)}{\lambda}\r)\,d\MP(x)\\
&\lesssim \int_{\Omega}\V\lf(x,\lf(\frac{6^r}{2^r-1}\r)^{\frac1r}
\frac{M(f)(x)}{\lambda}\r)\,d\MP(x).
\end{align*}
From this, we deduce that, for any $\lambda\in\D$,
$$\int_{\Omega}\V\lf(x,\frac1{\lambda}\left\{\sum_{k\in\mathbb{Z}}
\left[\frac{\mu^k\mathbf{1}_{B_{\nu^k}}(x)}
{\|\mathbf{1}_{B_{\nu^k}}\|_{L^{\varphi}(\Omega)}}\right]^r\right\}^{\frac1r}\r)\,d\MP(x)
\lesssim \int_{\Omega}\V\lf(x,\lf(\frac{6^r}{2^r-1}\r)^{\frac1r}
\frac{M(f)(x)}{\lambda}\r)\,d\MP(x),$$
which, together with the fact that $\V$ is of uniformly lower type $p$, implies that, for any given $r\in\D$,
$$\lf\|\left\{\sum_{k\in\mathbb{Z}}\left[\frac{\mu^k\mathbf{1}_{B_{\nu^k}}(x)}
{\|\mathbf{1}_{B_{\nu^k}}\|_{L^{\varphi}(\Omega)}}\right]^r\right\}^{\frac1r}\r\|_{L^{\V}(\Omega)}
\lesssim \|f\|_{H_{\V}^M(\Omega)}.$$
By this, we further know that, for any given $r\in\D$,
$f\in H_{{\rm at},\,r}^{\V,\fz,M}(\Omega)$ and
$\|f\|_{H_{{\rm at},\,r}^{\V,\fz,M}(\Omega)}\lesssim \|f\|_{H_{\V}^M(\Omega)}.$
The proof of Step $1)$ is now complete.

Step $2)$ Prove $H_{{\rm at},\,r}^{\V,\fz,M}(\Omega)\subset H_{\V}^M(\Omega)$.
To show this,
for any given $r\in(0,1]$, Let $f\in H_{{\rm at},\,r}^{\V,\fz,M}(\Omega).$ Then we can write
$f=\sum_{k\in\zz}\mu^ka^k$ for some sequence of triples,
$$\lf\{\mu^k,a^k,\nu^k\r\}_{k\in\mathbb{Z}}\in\mathcal{A}_M(\varphi,\fz).$$
By Definition \ref{def atom}(i), we know that, for any $k\in\mathbb{Z},$
$$\lf\{x\in\Omega:\ M(a^k)(x)\not=0\r\}\subset B_{v^k}.$$
Thus, for almost every $x\in\Omega$,
we have
$$M(a^k)(x)\le \|M(a^k)\|_{L^{\fz}(\Omega)}\mathbf{1}_{B_{v^k}}(x)
\le \mathbf{1}_{B_{v^k}}(x)\|\mathbf{1}_{B_{v^k}}\|_{L^{\varphi}(\Omega)}^{-1}.$$
Combining this and the subadditivity of the operator $M$, we find that,
for any given $r\in(0,1]$ and any $\lambda\in\D$,
\begin{align*}
\int_{\Omega}\V\lf(x,\frac{M(f)(x)}{\lambda}\r)\,d\MP(x)
&\le \int_{\Omega}\V\lf(x,\sum_{k\in\zz}\frac{\mu^kM(a^k)(x)}{\lambda}\r)\,d\MP(x)
\le \int_{\Omega}\V\lf(x,\sum_{k\in\zz}
\frac{\mu^k\mathbf{1}_{B_{v^k}}(x)}{\lambda\|
\mathbf{1}_{B_{v^k}}\|_{L^{\varphi}(\Omega)}}\r)\,d\MP(x)\\
&\le\int_{\Omega}\V\lf(x,\frac1{\lambda}\left\{\sum_{k\in\mathbb{Z}}
\left[\frac{\mu^k\mathbf{1}_{B_{\nu^k}}(x)}
{\|\mathbf{1}_{B_{\nu^k}}\|_{L^{\varphi}(\Omega)}}\right]^r\right\}^{\frac1r}\r)\,d\MP(x).
\end{align*}
Therefore, we conclude that, for any given $r\in(0,1]$, $f\in H_{\V}^M(\Omega)$ and
$\|f\|_{H_{\V}^M(\Omega)}\le \|f\|_{H_{{\rm at},\,r}^{\V,\fz,M}(\Omega)}$,
which completes the proof of Theorem \ref{Thm-atmSM}.
\end{proof}
\begin{remark}\label{rmk1}
From Example \ref{exv}, it follows that Theorems \ref{Thm-atms}, \ref{thm-atmpq} and \ref{Thm-atmSM}
cover the atomic characterizations of the variable martingale Hardy spaces in
\cite[Theorem 3.11, Theorem 3.12 and Proposition 4.19]{wyong1}.
Let $w$ be a special weight and $\Phi$ be an Orlicz function.
For any $x\in\Omega$ and $t\in\D$, if $\V(x,t):=w(x)\Phi(t)$,
then Theorems \ref{Thm-atms}, \ref{thm-atmpq} and \ref{Thm-atmSM} give
the atomic characterizations of the weighted martingale Orlicz Hardy
spaces, which is also new.
\end{remark}

\section{Martingale inequalities\label{s5}}

Let $X$ be a martingale space and $Y$ a measurable function space.
An operator $T:\ \ X\rightarrow Y$ is said to be a $\sigma$-sublinear operator
if, for any complex number $\alpha$,
$$
\left|T\left(\sum_{k=1}^{\infty}f_k\right)\right|\leq \sum_{k=1}^{\infty} |T(f_k)|\quad \text{and}\quad |T(\alpha f)|= |\alpha||T(f)|.$$

\begin{lemma}\label{lst}
Let $a$ be a measurable function. If there exists a stopping time
$\nu$ such that, for any $n\in\nn$ with $n\le\nu$,
$\ee_n(a)=0$, then, for any $n\in\nn$,
$$\ee_n(a\mathbf{1}_A)=\ee_n(a)\mathbf{1}_{A}, \quad \forall\,A\in\cf_{\nu}.$$
Moreover, let $T$ be any one of operators $S$, $s$ and $M$.
Then, for any $(\V,\fz)_T$-atom $a$,
$$T(a\mathbf{1}_A)=T(a)\mathbf{1}_{A}, \quad \forall\,A\in\cf_{\nu},$$
where $\nu$ is the stopping time associated with $a$.
\end{lemma}
\begin{proof}
Since $A\in\cf_{\nu}$, it follows that, for any $n\in\nn$,
$A\cap\{\nu<n\}\in\cf_{n-1}$. By this,
we find that, for any $n\in\nn$,
$$\ee_n(a\mathbf{1}_A)=\ee_n(a\mathbf{1}_{A\cap\{\nu<n\}}+a\mathbf{1}_{A\cap\{\nu\geq n\}})
=\ee_n(a)\mathbf{1}_{A\cap\{\nu<n\}}+\ee_n(a\mathbf{1}_{A\cap\{\nu\geq n\}}).$$
Now we claim that, for any $n\in\nn$, $\ee_n(a\mathbf{1}_{A\cap\{\nu\geq n\}})=0$. Indeed,
for any $m\in\nn$ and $m\geq n$, we have $\ee_m(a)\mathbf{1}_{\{\nu=m\}}=0$.
Combining this and $A\cap\{\nu=m\}\in\cf_m$,
we know that, for any $B\in\cf_n$,
$$\int_{B}a\mathbf{1}_{A\cap\{\nu\geq n\}}\,d\MP
=\sum_{m=n}^{\fz}\int_{B\cap A\cap\{\nu=m\}}a\,d\MP
=\sum_{m=n}^{\fz}\int_{B\cap A\cap\{\nu=m\}}\ee_ma\,d\MP=0.$$
This proves the above claim. Notice that, for any $n\in\nn$, $\ee_n(a)=0$
on the set $\{x\in\Omega:\ \nu(x)\geq n\}$. Therefore,
we conclude that, for any $n\in\nn$,
$$\ee_n(a\mathbf{1}_A)=\ee_n(a)\mathbf{1}_{A\cap\{\nu<n\}}=\ee_n(a)\mathbf{1}_{A}.$$
From this, it follows that, for any $(\V,\fz)_T$-atom $a$,
$$T(a\mathbf{1}_A)=T(a)\mathbf{1}_{A}, \quad \forall\,A\in\cf_{\nu},$$
where $\nu$ is the stopping time associated with $a$ and
$T$ any one of operators $S$, $s$ and $M$. This finishes the proof of Lemma \ref{lst}.
\end{proof}

Let $w$ be a special weight and $p\in\D$. For any $x\in\Omega$ and $t\in\D$, let $\V(x,t):=w(x)t^p$.
Let us denote the corresponding Musielak--Orlicz Hardy spaces by $H_{p}^{M}(\Omega,wd\MP)$, $H_{p}^{S}(\Omega,wd\MP)$, $H_{p}^{s}(\Omega,wd\MP)$, $Q_{p}(\Omega,wd\MP)$ and $P_{p}(\Omega,wd\MP)$, respectively.

\begin{theorem} \label{t20}
Let $w \in A_\infty(\Omega)\cap \mathbb{S}^{-}$ be a special weight and $p\in(0,\infty)$. If, for any $q \in (2,\infty)$, $T:H_q^S(\Omega) \rightarrow L^q(\Omega)$ (resp., $T:H_q^M(\Omega) \rightarrow L^q(\Omega)$) is a bounded $\sigma$-sublinear operator and, for any $(\varphi,\infty)_S$-atoms (resp., $(\varphi,\infty)_M$-atoms) $a$,
\begin{equation}\label{e30}
T(a) \mathbf{1}_A = T(a\mathbf{1}_A), \qquad \forall\,A \in \mathcal{F}_\nu,
\end{equation}
where $\nu$ is the associated stopping time with $a$, then there exists a positive constant $C$ such that,
for any $f\in Q_{p}(\Omega,wd\MP)$ (resp., $f\in P_{p}(\Omega,wd\MP)$),
$$
\|T(f)\|_{L^{p}(\Omega,wd\MP)}\le C \|f\|_{Q_{p}(\Omega,wd\MP)}
\qquad
\lf(resp.,\ \|T(f)\|_{L^{p}(\Omega,wd\MP)}\le C \|f\|_{P_{p}(\Omega,wd\MP)} \r).
$$
\end{theorem}

\begin{proof}
For any $p\in(0,\infty)$, $x\in\Omega$ and $t\in\D$, let $\V(x,t):=w(x)t^p$.
For any given $f\in Q_{\V}(\Omega)$, by Theorem \ref{thm-atmpq}, we know that there exists a sequence of triples,
$\{\mu^k,a^k,\nu^k\}_{k\in\mathbb{Z}}\in\mathcal{A}_S(\varphi,{\fz})$,
such that $f=\sum_{k\in\mathbb{Z}}\mu^ka^k$
and
\begin{align}\label{k1}
\left\|\left[\sum_{k\in\mathbb{Z}}2^{(k+1)r}\mathbf{1}_{B_{\nu^k}}\right]^{\frac1r}
\right\|_{L^{\varphi}(\Omega)} = \left\|\left\{\sum_{k\in\mathbb{Z}}\left[\frac{\mu^k\mathbf{1}_{B_{\nu^k}}}
{\|\mathbf{1}_{B_{\nu^k}}\|_{L^{\varphi}(\Omega)}}\right]^r\right\}^{\frac1r}
\right\|_{L^{\varphi}(\Omega)} \lesssim \|f\|_{Q_{\varphi}(\Omega)},
\end{align}
where $r \in (0,1]$, $a^{k}$ is a $(\varphi,\infty)_S$-atom with the associated
stopping time $\nu^{k}$ and $\mu^k=2^{k+1}\|\mathbf{1}_{B_{\nu^k}}\|_{L^{\varphi}(\Omega)}$
for any $k\in\mathbb{Z}$. By the $\sigma$-sublinearity  of the operator $T$,
we have
$$
\left| T(f) \right|\leq  \sum_{k\in \mathbb Z} \mu^{k} \lf|T(a^{k})\r|.
$$
From this and Lemma \ref{l1}, it follows that, for any $r\in(0,1]$,
$$
\|T(f)\|_{L^{\varphi}(\Omega)} \leq \left\|\left\{\sum_{k\in \mathbb Z} \left[\mu^{k} \lf|T(a^{k})\r| \right]^{r}\right\}^{\frac {1}{r}}\right\|_{L^{\varphi}(\Omega)}= \left\|\sum_{k\in \mathbb Z} \left[\mu^{k} \lf|T(a^{k})\r| \right]^{r}\right\|_{L^{\varphi_{1/r}}(\Omega)}^{\frac {1}{r}},
$$ where $\V_{1/r}$ is as in Definition \ref{d1}.
By this, Lemma \ref{lem-du}, \eqref{e30} and the fact that $a^k=a^k \mathbf{1}_{B_{\nu^k}}$ for any $k\in\zz$,
we may choose a function $g\in L^{\varphi^{*}_{1/r}}(\Omega)$ with norm less than or
equal to $1$ such that, for any $r\in(0,\min\{1,p\}]$,
\begin{align*}
\|T(f)\|_{L^{\varphi}(\Omega)}^{r}& \lesssim \int_\Omega \sum_{k\in \mathbb Z}
\left[\mu^{k} \lf|T(a^{k})\r| \right]^{r} g\, d\mathbb P \\
&\lesssim \sum_{k\in \mathbb Z} \int_\Omega 2^{(k+1)r} \|\mathbf{1}_{B_{\nu^k}}\|_{L^{\varphi}(\Omega)}^{r} \mathbf{1}_{B_{\nu^k}} \ee_{\nu^{k}} \left(\lf|T(a^{k})\r|^{r} |g|\right)\, d\mathbb P.
\end{align*}
From the H\"{o}lder inequality, we deduce that, for any $r\in(0,\min\{1,p\}]$
and $q\in(\max\{2, p\},\fz)$,
\begin{align}\label{e301}
\|T(f)\|_{L^{\varphi}(\Omega)}^{r}
& \lesssim \sum_{k\in \mathbb Z} \int_\Omega 2^{(k+1)r} \mathbf{1}_{B_{\nu^k}} \|\mathbf{1}_{B_{\nu^k}}\|_{L^{\varphi}(\Omega)}^{r}
\left[\ee_{\nu^{k}} \left(|T(a^{k})|^{q}\right)\right]^{r/q} \left[\ee_{\nu^{k}} \left(|g|^{(q/r)'}\right)\right]^{1/(q/r)'} \, d\mathbb P.
\end{align}
For any $k\in\zz$ and $A\in\cf_{\nu^k}$, by \eqref{e30}, Lemma \ref{lst}
and the boundedness of $T$ from $H_q^S(\Omega)$ to $L^q(\Omega)$, we find that
$$\int_{A}\lf|T(a^{k})\r|^{q}\,d\MP=\int_{\Omega}\lf|T(a^{k}\mathbf{1}_{A})\r|^{q}\,d\MP
\lesssim \int_{\Omega}\lf[S(a^{k}\mathbf{1}_{A})\r]^{q}\,d\MP
\lesssim \int_{A}\lf[S(a^{k})\r]^{q}\,d\MP.$$
By this and the fact that $a^k$ is a $(\varphi,{\fz})_S$-atom, we know that,
for any $A\in\cf_{\nu^k}$,
$$\int_{A}\lf|T(a^{k})\r|^{q}\,d\MP\lesssim \|\mathbf{1}_{B_{\nu^k}}\|_{L^{\varphi}(\Omega)}^{-q}\MP(A),$$
which implies that, for almost every $x\in\Omega$,
$$
\ee_{\nu^{k}} \left(\lf|T(a^{k})\r|^{q} \right)(x)
\lesssim \|\mathbf{1}_{B_{\nu^k}}\|_{L^{\varphi}(\Omega)}^{-q}.
$$
From this, \eqref{e301} and Lemma \ref{lem-du}, it follows that, for any $r\in(0,\min\{1,p\}]$
and $q\in(\max\{2, p\},\fz)$,
\begin{align}\label{k}
\|T(f)\|_{L^{\varphi}(\Omega)}^{r}
& \lesssim \sum_{k\in \mathbb Z} \int_\Omega 2^{(k+1)r} \mathbf{1}_{B_{\nu^k}}
 \left[\ee_{\nu^{k}} \left(|g|^{(q/r)'}\right)\right]^{1/(q/r)'} \, d\mathbb P \\
& \lesssim \int_\Omega \sum_{k\in \mathbb Z} 2^{(k+1)r} \mathbf{1}_{B_{\nu^k}}
 \left[M \left(|g|^{(q/r)'}\right)\right]^{1/(q/r)'} \, d\mathbb P \noz\\
&\lesssim  \left \|\sum_{k\in \mathbb Z} 2^{(k+1)r} \mathbf{1}_{B_{\nu^k}}
\right \|_{L_{\varphi_{1/r}}(\Omega)}  \left\|\left[M \left(|g|^{(q/r)'}\right)\right]^{1/(q/r)'} \right\|_{L_{\varphi^{*}_{1/r}}(\Omega)}.\noz
\end{align}
If $p/r>1$, Example \ref{x1}(i) tells us that
$\V_{1/r}^{*}(x,t)=w(x)^{1-(p/r)'} t^{(p/r)'}{p^{-\frac1{p-1}}}$ for any $x\in\Omega$ and $t\in\D$,
and hence we can write the second norm in the last expression of \eqref{k} as
\begin{align}\label{e31}
	\left\|\left[M \left(|g|^{(q/r)'}\right)\right]^{1/(q/r)'} \right\|_{L_{\varphi^{*}_{1/r}}(\Omega)}^{(p/r)'} = p^{-\frac1{p-1}}\int_\Omega \left[M \left(|g|^{(q/r)'}\right)\right]^{(p/r)'/(q/r)'} w^{1-(p/r)'} \, d\MP.
\end{align}
Since $w \in A_\infty(\Omega)$, it follows that there exists $r \in (0,\min\{1, p\}]$ and $r\neq p$
such that $w \in A_{p/r}(\Omega)$. From this, it follows that $w^{1-(p/r)'} \in A_{(p/r)'}(\Omega)$.
By Lemma \ref{lp} and $w \in  \mathbb{S}^{-}$, we know that there exists $\varepsilon\in\D$
such that $w^{1-(p/r)'} \in A_{(p/r)'-\varepsilon}(\Omega)$. We can choose $q$ large
enough such that
$$(p/r)'/(q/r)'>(p/r)'-\varepsilon.$$
By this, \eqref{e31} and the
Doob maximal inequality (Theorem \ref{thm-doob}), we have
\begin{align*}
	\left\|\left[M \left(|g|^{(q/r)'}\right)\right]^{1/(q/r)'} \right\|_{L_{\varphi^{*}_{1/r}}(\Omega)}^{(p/r)'}
\lesssim \int_\Omega |g|^{(p/r)'} w^{1-(p/r)'} \, d\MP \lesssim 1,
\end{align*}
because $\|g\|_{L^{\varphi^{*}_{1/r}}(\Omega)} \le 1$. From this, \eqref{k1} and \eqref{k}, we deduce that
\begin{align*}
\|T(f)\|_{L^{\varphi}(\Omega)}
&\lesssim  \left \|\sum_{k\in \mathbb Z} 2^{(k+1)r} \mathbf{1}_{B_{\nu^k}} \right \|_{L_{\varphi_{1/r}}(\Omega)}^{1/r}\lesssim \left\|\left[\sum_{k\in\mathbb{Z}}2^{(k+1)r}\mathbf{1}_{B_{\nu^k}}\right]^{\frac1r}
\right\|_{L^{\varphi}(\Omega)} \lesssim \|f\|_{Q_{\varphi}(\Omega)},
\end{align*}
which completes the proof for $Q_{\varphi}(\Omega)$. The proof for $P_{\varphi}(\Omega)$ is similar.
This finishes the proof of Theorem \ref{t20}.
\end{proof}

Similarly to the proof of Theorem \ref{t20}, we can prove the following theorem.
\begin{theorem} \label{t200}
Let $w \in A_\infty(\Omega)\cap \mathbb{S}^{-}$ be a special weight and $p\in(0,2)$.
If $T:H_2^s(\Omega) \rightarrow L^2(\Omega)$ is a bounded $\sigma$-sublinear operator
and, for any $(\varphi,\infty)_s$-atom $a$,
\begin{equation}\label{e300}
T(a) \mathbf{1}_A = T(a\mathbf{1}_A), \qquad \forall\,A \in \mathcal{F}_\nu,
\end{equation}
where $\nu$ is the associated stopping
time with $a$, then there exists a positive constant $C$ such that,
for any $f\in H_{p}^s(\Omega,wd\MP)$,
$$
\|T(f)\|_{L^{p}(\Omega,wd\MP)}\le C \|f\|_{H_{p}^s(\Omega,wd\MP)}.
$$
\end{theorem}
\begin{proof}
For any $\varepsilon\in(0,1)$, we choose $r$ small enough such
that $(p/r)'/(2/r)'>(p/r)'-\varepsilon$, this is due to
$$\lim_{r\to0}\frac{(p/r)'}{(p/r)'-\varepsilon}=\frac{1}{1-\varepsilon}
\quad \mbox{and} \quad \lim_{r\to0}(2/r)'=1.$$
Combing this and the proof of Theorem \ref{t20} with $q$ replaced by $2$, we
obtain the desired conclusion.
This finishes the proof of Theorem \ref{t200}.
\end{proof}

\begin{corollary} \label{c20}
Let $w \in A_\infty(\Omega)\cap \mathbb{S}^{-}$ be a special weight.
\begin{enumerate}
\item[{\rm(i)}] If $p\in\D$, then there exists a positive constant $C$ such that,
for any $f\in\cM$,
\begin{align}\label{e32}
	\|f\|_{H_{p}^{M}(\Omega,wd\MP)}\le C \|f\|_{P_{p}(\Omega,wd\MP)},\qquad \|f\|_{H_{p}^{S}(\Omega,wd\MP)}\le C \|f\|_{Q_{p}(\Omega,wd\MP)},
\end{align}
\begin{align}\label{e33}
	\|f\|_{H_{p}^{S}(\Omega,wd\MP)}\le C \|f\|_{P_{p}(\Omega,wd\MP)},\qquad \|f\|_{H_{p}^{M}(\Omega,wd\MP)}\le C \|f\|_{Q_{p}(\Omega,wd\MP)}
\end{align}
and
\begin{align}\label{e43}
	\|f\|_{H_{p}^{s}(\Omega,wd\MP)}\le C \|f\|_{P_{p}(\Omega,wd\MP)},\qquad \|f\|_{H_{p}^{s}(\Omega,wd\MP)}\le C \|f\|_{Q_{p}(\Omega,wd\MP)}.
\end{align}
\item[{\rm(ii)}] If $p\in(0,2)$, then there exists a positive constant $C$ such that,
for any $f\in\cM$,
\begin{align}\label{e320}
	\|f\|_{H_{p}^{M}(\Omega,wd\MP)}\le C \|f\|_{H_{p}^{s}(\Omega,wd\MP)},\qquad \|f\|_{H_{p}^{S}(\Omega,wd\MP)}\le C \|f\|_{H_{p}^{s}(\Omega,wd\MP)}.
\end{align}
\end{enumerate}
\end{corollary}

\begin{proof} The two inequalities in \eqref{e32} follow easily from Definition \ref{def-spaces}.
For the two inequalities in \eqref{e33}, consider the  operator $T=M$ or $S$ in Theorem \ref{t20}. Then
the both inequalities in \eqref{e33} follow from the Burkholder--Gundy inequality
	$$
	\|S(f)\|_{L^q(\Omega)}\sim \|M(f)\|_{L^q(\Omega)},  \qquad \forall\,q\in(1,\fz)
	$$
	(see, for example, \cite[Theorem 2.12]{W94}), Lemma \ref{lst} and Theorem \ref{t20}.

The both inequalities in \eqref{e43} can be deduced from choosing $T=s$ and applying the inequality
	$$
	\|s(f)\|_{L^q(\Omega)}\lesssim \|M(f)\|_{L^q(\Omega)} \sim \|S(f)\|_{L^q(\Omega)}, \qquad \forall\,q\in(2,\infty)
	$$
	 (see, for example, \cite[Theorem 2.11(ii)]{W94}), Lemma \ref{lst} and Theorem \ref{t20}.

In order to prove the both inequalities in \eqref{e320}, consider the operator $T=M$ or $S$ in Theorem \ref{t200}.
Then the desired inequalities in \eqref{e320} follow immediately from
$$\|S(f)\|_{L^2(\Omega)}\lesssim \|s(f)\|_{L^2(\Omega)}
\quad \mbox{and}\quad \|M(f)\|_{L^2(\Omega)}\lesssim \|s(f)\|_{L^2(\Omega)}$$
(see \cite[Theorem 2.11(i)]{W94} or \cite[Theorem 5.3(ii)]{BG70}), Lemma \ref{lst} and Theorem \ref{t200}.
This finishes the proof of Corollary \ref{c20}.
\end{proof}

\begin{remark}
Inequalities \eqref{e320} with a special class of weights were first studied by Kazamaki \cite[Theorem 1]{K79}.
Since the weight in \eqref{e320} is of wide generality, \eqref{e320}
generalizes \cite[Theorem 1]{K79} in the case when $p\in(0,2)$.
\end{remark}

\begin{theorem} \label{t22}
Let $\varphi \in A_\infty(\Omega)\cap \mathbb{S}^{-}$ be a Musielak--Orlicz function with uniformly lower type $p^{-}_{\V}$ and uniformly upper type $p^{+}_{\V}$ satisfying $0<p^{-}_{\V}\le p^{+}_{\V}<\infty.$
Suppose that the $\sigma$-sublinear operator $T$ satisfies \eqref{e30} and, for some $q \in (\max\{1,p^{+}_{\V}\},\infty)$ and any $t\in\D$,
\begin{align}\label{e34}
	\|T(f)\|_{L^{q}(\Omega,\varphi(\cdot,t)\, d\MP)}\leq C \|f\|_{H_{q}^{s}(\Omega,\varphi(\cdot,t)\,d\MP)}, \qquad \forall\, f\in H_{q}^{s}(\Omega,\varphi(\cdot,t)\,d\MP),
\end{align}
where $C$ is a positive constant independent of $f$ and $t \in  (0,\infty)$. Then
there exists a positive constant $C$ such that, for any $f\in H_{\varphi}^{s}(\Omega)$,
\begin{equation}\label{e35}
	\|T(f)\|_{L^{\varphi}(\Omega)}\le C \|f\|_{H_{\varphi}^{s}(\Omega)} .
\end{equation}
The same holds true if one replaces the spaces $H_{q}^{s}(\Omega,\varphi(\cdot,t)\,d\MP)$ and $H_{\varphi}^{s}(\Omega)$ by $Q_{q}(\Omega,\varphi(\cdot,t)\,d\MP)$ and $Q_{\varphi}(\Omega)$ or by $P_{q}(\Omega,\varphi(\cdot,t)\,d\MP)$ and $P_{\varphi}(\Omega)$, respectively.
\end{theorem}
\begin{proof}
Let $f\in H_{\V}^s(\Omega)$. Then, by Theorem \ref{Thm-atms}, we know that
there exists a sequence of triples,
$\{\mu^k,a^k,\nu^k\}_{k\in\mathbb{Z}}\in\mathcal{A}_s(\varphi,{\fz})$,
such that $f=\sum_{k\in\mathbb{Z}}\mu^ka^k$
and, for any $\lambda\in\D,$
\begin{align}\label{mi}
\int_{\Omega}\V\lf(x,\frac2{\lambda}
\sum_{k\in\mathbb{Z}}2^{k+1}\mathbf{1}_{G_k}(x)\r)\,d\MP(x)
\le\int_{\Omega}\V\lf(x,4\frac{s(f)(x)}{\lambda}\r)\,d\MP(x),
\end{align}
where $\nu^k:=\inf\{n\in\mathbb{N}:\ s_{n+1}(f)>2^k\}$, $\mu^k:=2^{k+1}\|\mathbf{1}_{B_{\nu^k}}\|_{L^{\varphi}(\Omega)}$
and $G_k:=B_{\nu^k}\backslash B_{\nu^{k+1}}$ for any $k\in\mathbb{Z}$ (see \eqref{e11}).
Thus, for any $k\in\mathbb{Z}$, we have $G_k=\{x\in\Omega:\ 2^k<s(f)(x)\le2^{k+1}\}$.
Since $\{G_k\}_{k\in\zz}$ are disjoint, we obtain, for any $k\in\mathbb{Z}$,
$\mathbf{1}_{B_{\nu^k}}=\sum_{j=k}^{\fz}\mathbf{1}_{G_j}$. From this, it follows that, for any $\lambda\in\D$,
\begin{align}\label{mi1}
\int_{\Omega}\V\lf(x,\frac{|T(f)(x)|}{\lambda}\r)\,d\MP(x)
&\le\int_{\Omega}\V\lf(x,\frac{1}{\lambda}
\sum_{k\in\mathbb{Z}}\mu^k\lf|T(a^k)(x)\r|\mathbf{1}_{B_{\nu^k}}(x)\r)\,d\MP(x)\\
&=\int_{\Omega}\V\lf(x,\frac{1}{\lambda}
\sum_{k\in\mathbb{Z}}\sum_{j=k}^{\fz}\mu^k\lf|T(a^k)(x)\r|\mathbf{1}_{G_j}(x)\r)\,d\MP(x)\noz\\
&=\sum_{j\in\mathbb{Z}}\int_{G_j}\V\lf(x,\frac{1}{\lambda}
\sum_{k=-\fz}^{j}2^{k+1}\lf\|\mathbf{1}_{B_{\nu^k}}\r\|_{L^{\varphi}(\Omega)}
\lf|T(a^k)(x)\r|\r)\,d\MP(x)\noz\\
&=:{\rm I}.\noz
\end{align}
Let $q\in(\max\{p^{+}_{\V},1\},\fz)$ and $\ell\in(0,1)$ be such that $q(1-\ell)=p^{+}_{\V}$.
Since $\V$ is of uniformly upper type $p^{+}_{\V}$, we know that
$\V$ is of uniformly upper type $q$. From this, the H\"older inequality and
the fact that $\{G_k\}_{k\in\zz}$ are disjoint, we deduce that,
for any $\lambda\in\D$,
\begin{align}
{\rm I}&\lesssim\sum_{j\in\mathbb{Z}}\int_{G_j}\frac{1}{2^{(j+1)q}}
\lf[\sum_{k=-\fz}^{j}2^{k+1}\lf\|\mathbf{1}_{B_{\nu^k}}\r\|_{L^{\varphi}(\Omega)}
\lf|T(a^k)(x)\r|\r]^q\V\lf(x,\frac{2^{j+1}}{\lambda}\r)\,d\MP(x) \nonumber\\
&\hs+\sum_{j\in\mathbb{Z}}\int_{G_j}\V\lf(x,\frac{2^{j+1}}{\lambda}\r)\,d\MP(x) \nonumber\\
&\lesssim\sum_{j\in\mathbb{Z}}\frac{1}{2^{(j+1)q}}\int_{G_j}
\lf(\sum_{k=-\fz}^{j}2^{k\ell q^{'}}\r)^{\frac{q}{q^{'}}}
\sum_{k=-\fz}^{j}2^{-k\ell q}2^{(k+1)q}\lf\|\mathbf{1}_{B_{\nu^k}}\r\|^q_{L^{\varphi}(\Omega)}
\lf|T(a^k)(x)\r|^q\V\lf(x,\frac{2^{j+1}}{\lambda}\r)\,d\MP(x) \nonumber\\
&\hs+\int_{\Omega}\V\lf(x,\frac1{\lambda}
\sum_{j\in\mathbb{Z}}2^{j+1}\mathbf{1}_{G_j}(x)\r)\,d\MP(x) \nonumber\\
&\lesssim\sum_{j\in\mathbb{Z}}\frac{1}{2^{jp^{+}_{\V}}}
\sum_{k=-\fz}^{j}2^{kp^{+}_{\V}}\lf\|\mathbf{1}_{B_{\nu^k}}\r\|^q_{L^{\varphi}(\Omega)}
\int_{G_j}\lf|T(a^k)(x)\r|^q\V\lf(x,\frac{2^{j+1}}{\lambda}\r)\,d\MP(x) \nonumber\\
&\hs+\int_{\Omega}\V\lf(x,\frac1{\lambda}
\sum_{j\in\mathbb{Z}}2^{j+1}\mathbf{1}_{G_j}(x)\r)\,d\MP(x)=:{\rm I_1+I_2}.\nonumber
\end{align}
Since $\V$ is of uniformly upper type $p^{+}_{\V}$,
from \eqref{e34}, the fact that $\{G_j\}_{j\in\zz}$ are disjoint, $a^k$ is a $(\varphi,\fz)_s$-atom and $\bigcup_{j=k}^{\fz}G_j=B_{\nu^k}$ for any $k\in\zz$,
we deduce that, for any $\lambda\in\D,$
\begin{align} \label{e36}
{\rm I_1}&= \sum_{k\in\zz}2^{kp^{+}_{\V}}\lf\|\mathbf{1}_{B_{\nu^k}}\r\|^q_{L^{\varphi}(\Omega)}
\sum_{j=k}^{\fz}\frac{1}{2^{jp^{+}_{\V}}}
\int_{G_j}\lf|T(a^k)(x)\r|^q\V\lf(x,\frac{2^{j+1}}{\lambda}\r)\,d\MP(x) \\
&\lesssim \sum_{k\in\zz}2^{kp^{+}_{\V}}\lf\|\mathbf{1}_{B_{\nu^k}}\r\|^q_{L^{\varphi}(\Omega)}
\sum_{j=k}^{\fz}\frac{1}{2^{jp^{+}_{\V}}}
\int_{G_j}\lf|T(a^k)(x)\r|^q2^{(j+1-k)p^{+}_{\V}}\V\lf(x,\frac{2^{k}}{\lambda}\r)\,d\MP(x)\nonumber\\
&\sim \sum_{k\in\zz}\lf\|\mathbf{1}_{B_{\nu^k}}\r\|^q_{L^{\varphi}(\Omega)}
\int_{B_{\nu^k}}\lf|T(a^k)(x)\r|^q\V\lf(x,\frac{2^{k}}{\lambda}\r)\,d\MP(x)\nonumber\\
&\lesssim \sum_{k\in\zz}\lf\|\mathbf{1}_{B_{\nu^k}}\r\|^q_{L^{\varphi}(\Omega)}
\int_{B_{\nu^k}}\lf[s(a^k)(x)\r]^q\V\lf(x,\frac{2^{k}}{\lambda}\r)\,d\MP(x) \nonumber\\
&\lesssim \sum_{k\in\zz}\lf\|\mathbf{1}_{B_{\nu^k}}\r\|^q_{L^{\varphi}(\Omega)}
\lf\|s(a^k)\r\|_{L^{\fz}(\Omega)}^q\int_{B_{\nu^k}}\V\lf(x,\frac{2^{k}}{\lambda}\r)\,d\MP(x)
\lesssim \sum_{k\in\zz}\int_{B_{\nu^k}}\V\lf(x,\frac{2^{k}}{\lambda}\r)\,d\MP(x). \nonumber
\end{align}
Combining this and the fact that $\V$ is of uniformly lower type $p^{-}_{\V}$
and, for any $j\in\zz$, $s(f)>2^j$ on $G_j$, we know that, for any $\lambda\in\D,$
\begin{align*}
{\rm I_1}&\lesssim \sum_{k\in\zz}\sum_{j=k}^{\fz}
\int_{G_j}\V\lf(x,\frac{2^{k}}{\lambda}\r)\,d\MP(x)
\lesssim \sum_{k\in\zz}\sum_{j=k}^{\fz}2^{(k-j)p^{-}_{\V}}
\int_{G_j}\V\lf(x,\frac{2^{j}}{\lambda}\r)\,d\MP(x)\\
&\lesssim \sum_{j\in\zz}\sum_{k=-\fz}^{j}2^{(k-j)p^{-}_{\V}}
\int_{G_j}\V\lf(x,\frac{s(f)(x)}{\lambda}\r)\,d\MP(x)
\lesssim \int_{\Omega}\V\lf(x,\frac{s(f)(x)}{\lambda}\r)\,d\MP(x).
\end{align*}
From this, it follows that, for any $\lambda\in\D$,
$${\rm I}\lesssim \int_{\Omega}\V\lf(x,\frac{s(f)(x)}{\lambda}\r)\,d\MP(x)
+{\rm I_2}.$$
Combining this and \eqref{mi1}, we deduce that, for any $\lambda\in\D,$
$$\int_{\Omega}\V\lf(x,\frac{|T(f)(x)|}{\lambda}\r)\,d\MP(x)
\lesssim \int_{\Omega}\V\lf(x,\frac{s(f)(x)}{\lambda}\r)\,d\MP(x)
+{\rm I_2},$$
which, together with \eqref{mi}, implies that, for any $\lambda\in\D$,
$$\int_{\Omega}\V\lf(x,\frac{|T(f)(x)|}{\lambda}\r)\,d\MP(x)
\lesssim\int_{\Omega}\V\lf(x,\frac{s(f)(x)}{\lambda}\r)\,d\MP(x).$$
Thus, we complete the proof of \eqref{e35}.

Assume now that $f\in Q_{\V}(\Omega)$. Then there exists an
optimal control sequence $(\lambda_n^{(1)}(f))_{n\in\zz_{+}}$ such that $S_n(f)\leq
\lambda_{n-1}^{(1)}(f)$ with $\lambda_\infty^{(1)}(f)\in L^{\varphi}(\Omega)$. If $a$ is a $(\varphi,\fz)_S$-atom, then $\lambda_\infty^{(1)}(a) \leq  \|S(a)\|_{L^{\infty}(\Omega)}$.
In the proof of \eqref{e35}, instead of \eqref{e36}, using \eqref{e34}, we conclude that,
for any $\lambda\in\D$,
\begin{align*}
{\rm I_1}&\lesssim\sum_{k\in\zz}\lf\|\mathbf{1}_{B_{\nu^k}}\r\|^q_{L^{\varphi}(\Omega)}
\int_{B_{\nu^k}}\lf|T(a^k)(x)\r|^q\V\lf(x,\frac{2^{k}}{\lambda}\r)\,d\MP(x)\\
&\lesssim \sum_{k\in\zz}\lf\|\mathbf{1}_{B_{\nu^k}}\r\|^q_{L^{\varphi}(\Omega)}
\int_{B_{\nu^k}}\lf[\lambda_\infty^{(1)}(a^k)(x)\r]^q\V\lf(x,\frac{2^{k}}{\lambda}\r)\,d\MP(x) \nonumber\\
&\lesssim \sum_{k\in\zz}\lf\|\mathbf{1}_{B_{\nu^k}}\r\|^q_{L^{\varphi}(\Omega)}
\lf\|S(a^k)\r\|_{L^{\fz}(\Omega)}^q\int_{B_{\nu^k}}\V\lf(x,\frac{2^{k}}{\lambda}\r)\,d\MP(x)
\lesssim \sum_{k\in\zz}\int_{B_{\nu^k}}\V\lf(x,\frac{2^{k}}{\lambda}\r)\,d\MP(x). \nonumber
\end{align*}
Thus, the proof of \eqref{e35} with $H_{\V}^s(\Omega)$ replaced by $Q_\varphi(\Omega)$ can be finished as above. The proof of \eqref{e35} with $H_{\V}^s(\Omega)$ replaced by $P_\varphi(\Omega)$ is similar. This finishes the proof of Theorem \ref{t22}.
\end{proof}

\begin{theorem}\label{thm-mi}
Let $\varphi \in A_\infty(\Omega)$ be a Musielak--Orlicz function with
uniformly lower type $p^{-}_{\V}$ and uniformly upper type $p^{+}_{\V}$
satisfying $0<p^{-}_{\V}\le p^{+}_{\V}<\infty$.
\begin{enumerate}
\item[{\rm(i)}] If $\varphi \in \mathbb{S}^{-}$ and $p_\varphi^{+}\in(0,2)$, then
there exists a positive constant $C$ such that, for any $f\in H_{\V}^s(\Omega)$,
 \begin{equation}\label{e51}
 \|f\|_{H_{\V}^M(\Omega)}\le C\|f\|_{H_{\V}^s(\Omega)}.
 \end{equation}
\item[{\rm(ii)}]
If $\varphi \in \mathbb{S}^{-}$ and $p_\varphi^{+}\in(0,2)$, then
there exists a positive constant $C$ such that, for any $f\in H_{\V}^s(\Omega)$,
 \begin{equation}\label{e52}
 \|f\|_{H_{\V}^S(\Omega)}\le C \|f\|_{H_{\V}^s(\Omega)}.
 \end{equation}
\item[{\rm(iii)}]
If $\varphi \in \mathbb{S}^{-}$, then
there exists a positive constant $C$ such that, for any $f\in Q_{\V}(\Omega)$,
 \begin{equation}\label{e53}
\|f\|_{H_{\V}^M(\Omega)}\leq \|f\|_{P_{\V}(\Omega)},\qquad  \|f\|_{H_{\V}^S(\Omega)}\leq \|f\|_{Q_{\V}(\Omega)},
\end{equation}
 \begin{equation}\label{e54}
\|f\|_{H_{\V}^S(\Omega)}\le C\|f\|_{P_{\V}(\Omega)},\qquad  \|f\|_{H_{\V}^M(\Omega)}\le C\|f\|_{Q_{\V}(\Omega)},
\end{equation}
 \begin{equation}\label{e55}
\|f\|_{H_{\V}^s(\Omega)}\le C\|f\|_{P_{\V}(\Omega)},\qquad  \|f\|_{H_{\V}^s(\Omega)}\le C\|f\|_{Q_{\V}(\Omega)},
\end{equation}
and
\begin{equation}\label{e56}
\frac1C\|f\|_{P_{\V}(\Omega)} \le \|f\|_{Q_{\V}(\Omega)} \le C\|f\|_{P_{\V}(\Omega)}.
\end{equation}
\end{enumerate}
Moreover, if $\{\cf_n\}_{n\in\zz_{+}}$ is regular, then
$$
H_{\V}^M(\Omega)=P_{\V}(\Omega)=H_{\V}^s(\Omega)=H_{\V}^S(\Omega)=Q_{\V}(\Omega)
$$
with equivalent quasi-norms.
\end{theorem}

\begin{proof}
By Lemma \ref{lst}, we know that the operators $M,S$ and $s$ are all satisfy \eqref{e30}. Then \eqref{e51} and \eqref{e52} follow from \eqref{e320} and Theorem \ref{t22} with $T=M$ or $T=S$.
Inequalities \eqref{e53} come easily from the definition of $P_{\V}(\Omega)$ and $Q_{\V}(\Omega)$.
Inequalities \eqref{e54} and \eqref{e55} follow from Corollary \ref{c20} and
Theorem \ref{t22} by choosing $T=M$, $S$ or $s$.

To prove \eqref{e56}, we use \eqref{e54}. If $f=(f_n)_{n\in\zz_{+}}\in Q_{\V}(\Omega)$,
then there exists an optimal control $(\lambda_n^{(1)}(f))_{n\in\zz_{+}}$ such
that $S_n(f)\leq \lambda_{n-1}^{(1)}(f)$
with $\lambda_\infty^{(1)}(f)\in L^{\varphi}(\Omega)$. Since
$$|f_n|\leq M_{n-1}(f) +\lambda_{n-1}^{(1)}(f),
$$
by the second inequality of \eqref{e54}, we have
$$
\|f\|_{P_{\V}(\Omega)}\lesssim\|f\|_{H_{\V}^M(\Omega)}+\lf\|\lambda_\infty^{(1)}(f)\r\|_{L^{\varphi}(\Omega)} \lesssim\|f\|_{Q_{\V}(\Omega)}.
$$
Thus, we have $f=(f_n)_{n\in\zz_{+}}\in P_{\V}(\Omega)$. Then, by the definition of
$P_{\V}(\Omega)$ (see Definition \ref{def-spaces}), we know that
there exists an optimal control $(\lambda_n^{(2)}(f))_{n\in\zz_{+}}$
such that, for any $n\in\nn$, $|f_n|\leq \lambda_{n-1}^{(2)}(f)$ and
$\lambda_\infty^{(2)}(f)\in L^{\varphi}(\Omega)$. Notice that, for any $n\in\nn$,
$$
S_n(f)\leq S_{n-1}(f)+2\lambda_{n-1}^{(2)}(f).$$
Using the first inequality of \eqref{e54},  we obtain the second inequality of \eqref{e56}.

Further, assume that $\{\mathcal {F}_n\}_{n\in\zz_{+}}$ is regular.
Since $\V\in A_\infty(\Omega)$, from Lemma \ref{lem-ss}, it follows that $\V\in\ss$.
By this and \eqref{e56}, we know that $P_\varphi(\Omega)=Q_\varphi(\Omega)$.
Let $f\in P_{\V}(\Omega)$. From Theorems \ref{thm-atmpq} and \ref{Thm-atmSM},
it follows that, for any given $r\in(0,1]$,
$$\|f\|_{P_{\V}(\Omega)}\sim\|f\|_{H_{{\rm at},\,r}^{\V,\fz,M}(\Omega)}\sim \|f\|_{H_{\V}^M(\Omega)}.$$
Similarly, for any $f\in Q_{\V}(\Omega)$, we know that, for any given $r\in(0,1]$,
$$
\|f\|_{Q_{\V}(\Omega)}\sim\|f\|_{H_{{\rm at},\,r}^{\V,\fz,S}(\Omega)}\sim \|f\|_{H_{\V}^S(\Omega)}.
$$
By the regularity condition of $\{\cf_n\}_{n\in\zz_{+}}$, we find that, for any $n\in\nn,$ $|d_nf|^2\le R\ee_{n-1}(|d_nf|^2)$ (see, for example, \cite[Proposition 2.18]{W94}) and
hence
$$S_n(f)\le R^{\frac12}s_n(f).
$$
Since $s_n(f)\in \mathcal {F}_{n-1}$ for any $n\in\nn$, from the definition of
$Q_{\V}(\Omega)$, we deduce that
$$
\|f\|_{Q_{\V}(\Omega)}\lesssim \|s(f)\|_{L^\varphi(\Omega)}\sim\|f\|_{H_{\V}^s(\Omega)}.
$$
Now \eqref{e55} yields that
$$
\|f\|_{Q_{\V}(\Omega)}\sim \|f\|_{H_{\V}^s(\Omega)},
$$
which completes the proof of Theorem \ref{thm-mi}.
\end{proof}

\begin{remark}\label{rmk2}
\begin{enumerate}
\item[{\rm (i)}] Let $p\in\D$. If $\V(x,t):=t^p$ for any $x\in\Omega$ and $t\in\D$, then Theorem \ref{thm-mi} in this case becomes \cite[Theorem 2.22]{W94}.
\item[{\rm (ii)}] Let $\Phi$ be an Orlicz function. Theorem \ref{thm-mi}
when $\V(x,t):=\Phi(t)$ for any $x\in\Omega$ and $t\in\D$ was proved by Miyamoto et al.
\cite[Theorem 2.5 and Corollary 2.6]{mns12}. But, the assumptions of
\cite[Theorems 2.11 and 2.12]{W94} require that
$\V$ is of lower type $p_{\V}^{-}$ and of upper type
$p_{\V}^{+}$ satisfying $0<p_{\V}^{-}\le p_{\V}^{+}\le1$. However,
Theorem \ref{thm-mi} only needs $p_{\V}^{+}\in(0,2)$ in \eqref{e51} and \eqref{e52}.
Therefore, in this sense, Theorem \ref{thm-mi} generalizes and
improves \cite[Theorem 2.5 and Corollary 2.6]{mns12}.

\item[{\rm (iii)}]
Let $w\in A_\infty(\Omega)$ be a special weight and $\Phi$ an Orlicz function with
lower type $p_{\V}^{-}$ and upper type $p_{\V}^{+}$ satisfying
$0<p_{\V}^{-}\le p_{\V}^{+}<\fz$. Letting $\V(x,t):=w(x)\Phi(t)$ for any $x\in\Omega$ and $t\in\D$,
then Theorem \ref{thm-mi} with such a $\V$ is completely new.
\end{enumerate}
\end{remark}

Now we are ready to generalize the well-known Burkholder--Davis--Gundy inequality. To this end, we shall need the Davis decomposition of the martingales from $H_{\V}^S(\Omega)$ and $H_{\V}^M(\Omega)$ and some additional notions.

\begin{definition}\label{d20}
Let $\varphi$ be a Musielak--Orlicz function.
The \emph{martingale Musielak--Orlicz Hardy space} $G_{\varphi}(\Omega)$ is defined by setting
\begin{align*}
G_{\varphi}(\Omega):=\left\{f\in\mathcal{M}:\ \|f\|_{G_{\varphi}(\Omega)}:= \left\|\sum_{n\in\zz_{+}} |d_nf| \right\|_{L^{\varphi}(\Omega)}<\infty\right\}.
\end{align*}
\end{definition}

\begin{lemma}\label{l22}
Let $\V$ be a Musielak--Orlicz function with
uniformly lower type $p^{-}_{\V}$
for some $p^{-}_{\V}\in[1,\fz)$. If the Doob maximal operator $M$ is bounded
on $L^{\varphi^{*}}(\Omega)$ and $f \in H_{\V}^S(\Omega)$, then there exist
$h \in G_\varphi(\Omega)$ and $g \in Q_\varphi(\Omega)$ such that $f_n=h_n+g_n$ for any $n\in\mathbb{Z}_{+}$, and there exists a positive constant $C$, independent of $f$,
such that
	$$
	\|h\|_{G_\varphi(\Omega)} \le C \|f\|_{H_{\V}^S(\Omega)} \quad \mbox{and} \quad
	\|g\|_{Q_\varphi(\Omega)} \le C \|f\|_{H_{\V}^S(\Omega)}.
	$$
\end{lemma}

\begin{proof}
Let $f\in H_{\V}^S(\Omega)$.
Suppose that $0=\lambda_0 \leq \lambda_1 \leq \cdots$ is an adapted sequence of functions such that,
for any $n\in\mathbb{Z}_{+}$,
	$$
	S_n(f) \leq \lambda_n\quad \mbox{and} \quad
\lambda_\infty := \sup_{n\in\mathbb{Z}_{+}}\lambda_n \in L^\varphi(\Omega).
	$$
	Clearly, for any $n\in\nn$, we have
	$$
	d_nf=d_nf\mathbf{1}_{\{x\in\Omega:\ \lambda_n(x) > 2\lambda_{n-1}(x)\}}
+ d_nf\mathbf{1}_{\{x\in\Omega:\ \lambda_n(x) \le 2\lambda_{n-1}(x)\}}.
	$$
For any $n\in\nn$, let
	$$
	h_n:=\sum_{k=1}^n \left[d_kf\mathbf{1}_{\{x\in\Omega:\ \lambda_k(x) > 2\lambda_{k-1}(x)\}}-\ee_{k-1}\lf(d_kf\mathbf{1}_{\{x\in\Omega:\ \lambda_k(x) > 2\lambda_{k-1}(x)\}}\r)\right]
	$$
	and
$$
g_n:=\sum_{k=1}^n \left[d_kf\mathbf{1}_{\{x\in\Omega:\ \lambda_k(x) \le 2\lambda_{k-1}(x)\}}-
\ee_{k-1}\lf(d_kf\mathbf{1}_{\{x\in\Omega:\ \lambda_k(x) \le2\lambda_{k-1}(x)\}}\r)\right].
$$
Then, for any $n\in\nn$, $f_n=h_n+g_n$. For any $k\in\zz$, on the set
$\{x\in\Omega:\ \lambda_k(x) > 2\lambda_{k-1}(x)\}$,
we have $\lambda_k < 2(\lambda_k - \lambda_{k-1})$, henceforth
$$
|d_kf|\mathbf{1}_{\{x\in\Omega:\ \lambda_k(x) > 2\lambda_{k-1}(x)\}}
\leq\lambda_k \mathbf{1}_{\{x\in\Omega:\ \lambda_k(x) > 2\lambda_{k-1}(x)\}}
\leq2(\lambda_k - \lambda_{k-1}).
$$
Thus, we conclude that, for any $n\in\nn$,
\begin{equation}\label{e37}
	\sum_{k=1}^n |d_kh| \leq 2\lambda_n + 2 \sum_{k=1}^n \ee_{k-1} (\lambda_k - \lambda_{k-1}).
\end{equation}
From this and Theorem \ref{thm-duald}, it follows that
$$
\|h\|_{G_\varphi(\Omega)} \lesssim \|\lambda_\infty\|_{L^\varphi(\Omega)}.
$$
On another hand, for any $k\in\nn$, we have
$$
|d_kf|\mathbf{1}_{\{x\in\Omega:\ \lambda_k(x) \le 2\lambda_{k-1}(x)\}}
\leq
\lambda_k \mathbf{1}_{\{x\in\Omega:\ \lambda_k(x) \le 2\lambda_{k-1}(x)\}} \leq
2\lambda_{k-1},
$$
which implies that
$$
|d_kg| \leq 4\lambda_{k-1}.
$$
Combining this and \eqref{e37}, we conclude that, for any $n\in\nn$,
\begin{align*}
	S_n(g) &\leq S_{n-1}(g)+|d_ng|
\leq S_{n-1}(f)+S_{n-1}(h)+4\lambda_{n-1} \cr
&\leq \lambda_{n-1} + 2\lambda_{n-1} +
2 \sum_{k=1}^{n-1} \ee_{k-1} (\lambda_k - \lambda_{k-1}) + 4\lambda_{n-1}.
\end{align*}
From this and Theorem \ref{thm-duald}, it follows that
$$
\|g\|_{Q_\varphi(\Omega)} \lesssim \|\lambda_\infty\|_{L^\varphi(\Omega)}.
$$
For any $n\in\zz_{+}$, letting $\lambda_n:=S_n(f)$, we then obtain the desired conclusion. This finishes the proof of Lemma \ref{l22}.
\end{proof}

Using Theorem \ref{thm-duald}, we also obtain the Davis decomposition of $H_{\V}^M(\Omega)$,
whose proof is similar to that of Lemma \ref{l22}, the details being omitted.
\begin{lemma}\label{l21}
Let $\V$ be a Musielak--Orlicz function with uniformly lower type $p^{-}_{\V}$
for some $p^{-}_{\V}\in[1,\fz)$. If the Doob maximal operator $M$ is bounded
on $L^{\varphi^{*}}(\Omega)$ and $f \in H_{\V}^M(\Omega)$, then there exist $h \in G_\varphi(\Omega)$ and $g \in P_\varphi(\Omega)$ such that $f_n=h_n+g_n$ for any $n\in\mathbb{Z}_{+}$, and there exists a positive constant $C$, independent of $f$,
such that
	$$
	\|h\|_{G_\varphi(\Omega)} \le C \|f\|_{H_{\V}^M(\Omega)} \qquad \mbox{and}\qquad
	\|g\|_{P_\varphi(\Omega)} \le C \|f\|_{H_{\V}^M(\Omega)}.
	$$
\end{lemma}

The generalization of Burkholder--Davis--Gundy inequalities reads as follows.

\begin{theorem}\label{t21}
	Let $\varphi \in A_\infty(\Omega)\cap \mathbb{S}^{-}$ be a Musielak--Orlicz function with uniformly lower type $p^{-}_{\V}$ and uniformly upper type $p^{+}_{\V}$ satisfying $1<p^{-}_{\V}\le p^{+}_{\V}<\infty.$ If the Doob maximal operator $M$ is bounded
on $L^{\varphi^{*}}(\Omega)$,
then there exists a positive constant $C$ such that, for any martingale $f\in\mathcal{M}$,
	\begin{equation}\label{bd}
	\frac{1}{C}\|f\|_{H_\varphi^{S}(\Omega)} \le \|f\|_{H_{\V}^M(\Omega)}
\le C \|f\|_{H_\varphi^{S}(\Omega)}.
	\end{equation}
\end{theorem}
\begin{proof}
First we prove the second inequality of \eqref{bd}.
Let $f \in H_{\V}^S(\Omega)$. By Lemma \ref{l22}, we know that there exist
$h \in G_\varphi(\Omega)$ and $g \in Q_\varphi(\Omega)$ such that,
for any $n\in\mathbb{Z}_{+}$, $f_n=h_n+g_n$ and
\begin{align}\label{b1}
\|h\|_{G_\varphi(\Omega)} \lesssim \|f\|_{H_{\V}^S(\Omega)}, \qquad
	\|g\|_{Q_\varphi(\Omega)} \lesssim \|f\|_{H_{\V}^S(\Omega)}.
\end{align}
It is easy to show that
$M(h)\le \sum_{n\in\zz_{+}}\lf|d_nh\r|$.
From this, it follows that
$$
\|h\|_{H_{\V}^M(\Omega)} \leq \|h\|_{G_\varphi(\Omega)}.
$$
Using this, \eqref{e54} and \eqref{b1}, we conclude that
$$
\|f\|_{H_{\V}^M(\Omega)} \leq \|h\|_{H_{\V}^M(\Omega)} + \|g\|_{H_{\V}^M(\Omega)} \lesssim \|h\|_{G_\varphi(\Omega)} + \|g\|_{Q_\varphi(\Omega)} \lesssim \|f\|_{H_{\V}^S(\Omega)}.
$$
The first inequality of \eqref{bd} can be proved in the same way.
This finishes the proof of Theorem \ref{t21}.
\end{proof}
\begin{remark}\label{rmk-bdg}
If, in Lemmas \ref{l22} and \ref{l21} and Theorem \ref{t21}, the boundedness of the Doob maximal operator on $L^{\varphi^{*}}(\Omega)$ is replaced by the condition that $\V^{*}\in A_{\fz}(\Omega)$ satisfies \eqref{e38}, then, from Corollary \ref{thm-dooc},
we deduce that Lemmas \ref{l22} and \ref{l21} and Theorem \ref{t21} still hold true.
\end{remark}

Now, we turn to Burkholder--Davis--Gundy inequalities for the endpoint case: the uniformly lower type index $p^{-}_{\V}=1$.
To this end, we first prove the following
theorem, which is a generalization of \cite[Theorem 3.2]{BDG72}.
\begin{theorem}\label{t32}
Let $\Phi:[0,\fz)\to[0,\fz)$ be an Orlicz function and $w$ a weight.
Let $\V(x,t)=w(x)\Phi(t)$ for any $x\in\Omega$ and $t\in\D$.
If $w\in A_1(\Omega)$ and
$\V$ is of uniformly lower type $1$ and of uniformly
upper type $p_{\V}^{+}$ for some $p_{\V}^{+}\in[1,\infty)$,
then there exists a positive
constant $C$ such that, for any sequence $(g_k)_{k\in\zz_{+}}$ of non-negative
$\cf$ measurable functions,
$$\lf\|\sum_{k\in\zz_{+}}\ee_k(g_k)\r\|_{L^{\V}(\Omega)}
\le C\lf\|\sum_{k\in\zz_{+}}g_k\r\|_{L^{\V}(\Omega)}.$$
\end{theorem}
\begin{proof}
Let $d\widehat{\MP}:=w d\MP$. Denote the expectation and the
conditional expectation operators related to $\widehat{\MP}$,
respectively, by $\widehat{\ee}$ and $\widehat{\ee}_n$. Since $w\in L^1(\Omega)$,
we may assume that $\widehat{\MP}$ is a probability measure. Then
$(\Omega,\cf,\{\cf_n\}_{n\in\zz_{+}},\widehat{\MP}) $ is a probability space.
Combining this and \cite[Theorem 3.2]{BDG72}, we obtain
\begin{align}\label{32}
\widehat{\ee}\Phi\lf(\sum_{n=1}^{\fz}\widehat{\ee}_n(g_n)\r)
\lesssim \widehat{\ee}\Phi\lf(\sum_{n=1}^{\fz}g_n\r) .
\end{align}
For any $\cf$ measurable function $f$,
by \cite[Proposition 6.1.7]{Long93} (or \cite[Section 4]{DM78}), we know that,
for any $n\in\zz_{+}$,
$$\widehat{\ee}_n(f)=\frac1{w_n}\ee_n(fw).$$
By this and $w\in A_1(\Omega)$, we find that, for any $n\in\zz_{+}$,
$$\widehat{\ee}_n(f)=\ee_n\lf(\frac1{w_n}fw\r) \gtrsim \ee_n(f).$$
From this and \eqref{32}, it follows that
$$\widehat{\ee}\Phi\lf(\sum_{n=1}^{\fz}\ee_n(g_n)\r)
\lesssim \widehat{\ee}\Phi\lf(\sum_{n=1}^{\fz}g_n\r) ,$$
which implies that
$$\int_{\Omega}\V\lf(x,\sum_{n=1}^{\fz}\ee_n(g_n)(x)\r)\,d\MP
\lesssim \int_{\Omega}\V\lf(x,\sum_{n=1}^{\fz}g_n(x)\r)\,d\MP.$$
This finishes the proof of Theorem \ref{t32}.
\end{proof}

\begin{remark}\label{r20}
In Theorem \ref{t32}, the uniformly lower type $1$
and the uniformly upper type $p^{+}\in [1,\fz)$ properties of $\V$ can be
replaced by the condition that $\Phi$ is convex and $\Phi\in\Delta_2$
(that is, $\Phi(2t)\le K \Phi(t)$ for any $t\in\D$). Indeed, if $\Phi$ is of lower type $1$,
then, by \cite[Proposition 2.3]{hhk16} (see also \cite[Lemma 2.2]{ch18}), we know that
$\Phi$ is equivalent to a convex function. On another hand,
it is clear that $\Phi\in\Delta_2$ if and only if $\V$ is of uniformly upper type $p^{+}\in [1,\fz)$.
\end{remark}

Using Theorem \ref{t32} and Remark \ref{r20},
we can prove the following theorem in the way same as the proof of
Theorem \ref{t21}, the details being omitted.
\begin{theorem}\label{t210}
Let $\Phi:[0,\fz)\to[0,\fz)$ be an Orlicz function and $w$ a special weight.
Let $\V(x,t):=w(x)\Phi(t)$ for any $x\in\Omega$ and $t\in\D$.
If $\V$ is of uniformly lower type $1$ and of uniformly upper type
$p_{\V}^{+}$ for some $p_{\V}^{+}\in[1,\fz)$ and $w\in A_1(\Omega)$,
then there exists a positive constant $C$ such that, for any martingale $f\in\mathcal{M}$,
	$$
	\frac1{C}\|f\|_{H_\varphi^{S}(\Omega)} \le \|f\|_{H_{\V}^M(\Omega)}
\le C \|f\|_{H_\varphi^{S}(\Omega)}.
	$$
\end{theorem}
Now, we compare Theorems \ref{t21} and \ref{t210}
with Bonami and L\'epingle \cite[Theorem 1]{bl79}.
To this end, we need the following notion,
which was first introduced by Dol\'eans--Dade et al. \cite{DM78}.
\begin{definition}
Let $p\in(1,\fz)$ and $w$ be a special weight. Then $w$ is said to satisfy an
$\widehat{A}_{p}(\Omega)$ condition, denoted by $w\in \widehat{A}_{p}(\Omega)$,
if there exists a positive constant $C$ such that, for any $n\in\zz_{+}$,
$$ \frac{1}{w_n}\lf[\widehat{\ee}_n\lf(w^{\frac1{p-1}}\r)\r]^{p-1}\le C . $$
Then $w$ is said to satisfy $\widehat{A}_{\fz}(\Omega)$
if $w\in \widehat{A}_{p}(\Omega)$ for some $p\in(1,\fz)$.
\end{definition}

\begin{remark}\label{rbdg} Recall that Bonami and L\'epingle \cite[Theorem 1]{bl79}
proved that, if $w\in \widehat{A}_{\fz}(\Omega) \cap \ss^{-}$, then the
Burkholder--Davis--Gundy inequality holds true in the weighted Orlicz case.
Now we claim that $\widehat{A}_{\fz}(\Omega) \cap \ss^{-}$ is slightly stronger than the
condition $A_{\fz}(\Omega) \cap \ss^{-}$. Indeed,
let $w\in \widehat{A}_{\fz}(\Omega)\cap \ss^{-}$, then there exists an index $p\in(1,\fz)$ such
that $w\in \widehat{A}_{p}(\Omega)$. From \cite[p.\,298]{bl79}, we deduce that
$w\in \widehat{A}_{p}(\Omega)$ is equivalent to
$$\ee_n\lf(w^{\frac{p}{p-1}}\r)\le C_p w_n^{\frac{p}{p-1}},\quad \quad \forall\,n\in\zz_{+}. $$
Combining this, $w\in\ss^{-}$ and \cite[Proposition 5]{DM78}, we know that
there exists an index $q\in(1,\fz)$ such that $w\in A_q(\Omega)$. From this, it follows that
$w\in A_{\fz}(\Omega) \cap \ss^{-}$.
This proves the above claim (see also \cite[Proposition 6.1.8 and Remark 6.6.9]{Long93}).

Thus, when $\V(x,t):=w(x)\Phi(t)$ for any $x\in\Omega$ and $t\in\D$, the assumption on
the weight of Theorem \ref{t21} is slightly weaker than the assumption
on the weight of \cite[Theorem 1]{bl79}. However, Theorems \ref{t21} and \ref{t210}
can not cover \cite[Theorem 1]{bl79},
because Theorem \ref{t21} needs that the Doob maximal operator $M$ is
bounded on $L^{\V^{*}}(\Omega)$ and Theorem \ref{t210} needs $w\in A_{1}(\Omega)$.
If $\V(x,t):=t^p$ for any $x\in\Omega$ and $t\in\D$, with
$p\in[1,\infty)$, then Theorems \ref{t21} and \ref{t210} become the classical
Burkholder--Davis--Gundy inequality (see \cite{BDG72}).
\end{remark}

The following corollary follows immediately from Lemma \ref{lem-ss} and
Theorems \ref{Thm-atms}, \ref{thm-atmpq}, \ref{Thm-atmSM} and \ref{thm-mi}.
\begin{corollary}\label{cor-mi}
Let $\varphi$ be a Musielak--Orlicz function with uniformly lower type
$p^{-}_{\V}$ and uniformly upper type $p^{+}_{\V}$ satisfying
$0<p^{-}_{\V}\le p^{+}_{\V}<\infty.$ If $\V\in A_{\fz}(\Omega)$ and $\{\cf_n\}_{n\in\zz_{+}}$
is regular, then, for any given $r\in(0,1]$,
$H_{\V}(\Omega)=H_{{\rm at},\,r}^{\V,\fz,T}(\Omega)
\mbox{ with equivalent quasi-norms}$,
here $T$ stands for any one of operators $M$, $s$ and $S$, and $H_{\V}(\Omega)$
denotes any one of five Musielak--Orlicz martingale Hardy spaces
$H_{\V}^M(\Omega)$, $P_{\V}(\Omega)$, $H_{\V}^s(\Omega)$,
$H_{\V}^S(\Omega)$ and $Q_{\V}(\Omega)$.
\end{corollary}

\begin{definition}
For any $f:=(f_n)_{n\in\zz_{+}}\in\cM$, the \emph{martingale transform $\mathcal{T}$} is
defined by setting, for any $n\in\nn$,
$$(\mathcal{T}f)_n:=\sum_{k=1}^n v_{k-1}d_kf \mbox{ and }(\mathcal{T}f)_0:=0,$$
where, for any $k\in\zz_{+}$, $v_k$ is $\cf_k$ measurable and $\|v_k\|_{L^{\fz}(\Omega)}\le1$.
\end{definition}
Now we turn to the boundedness of the martingale transform on $L^{\V}(\Omega)$.
\begin{theorem}\label{thm-tran}
Let $\V\in A_{\fz}(\Omega)$ be a Musielak--Orlicz function satisfying \eqref{doob}. If the stochastic basis $\{\cf_n\}_{n\in\zz_{+}}$ is regular,
then the martingale transform $\mathcal{T}$ is bounded on $L^{\V}(\Omega)$.
\end{theorem}
\begin{proof}
Since $\|v_k\|_{L^{\fz}(\Omega)}\le1$, we have $S(\mathcal{T}f)\le S(f)$ $\MP$-almost
everywhere. From this, Theorems \ref{thm-mi} and \ref{thm-doob},
it follows that
$$\|\mathcal{T}f\|_{L^{\V}(\Omega)}\le \|M(\mathcal{T}f)\|_{L^{\V}(\Omega)}
\lesssim \|S(f)\|_{L^{\V}(\Omega)}\lesssim \|M(f)\|_{L^{\V}(\Omega)}
\lesssim \|f\|_{L^{\V}(\Omega)},$$
which completes the proof of Theorem \ref{thm-tran}.
\end{proof}

\section{Walsh system and Fej{\'e}r means\label{s6}}

Let us investigate the dyadic martingales. Namely,
let $\Omega:=[0,1)$, $\mathbb P$ be the Lebesgue measure
and $\mathcal F$ the set of all Lebesgue measurable sets.
By a {\it dyadic interval}, we mean one of the form
$[k2^{-n},(k+1)2^{-n})$ for some $k,n \in \zz_{+}$ and
$0 \leq k <2^n$. For any given $n\in \zz_{+}$ and $x \in [0,1)$,
denote by $I_n(x)$ the dyadic interval of length $2^{-n}$ which
contains $x$. For any $n\in \zz_{+}$, the $\sigma$-algebras generated by the dyadic
intervals $\{I_{n}(x): x\in [0,1)\}$ is denoted by $\mathcal F_{n}$.
It is easy to see that $(\mathcal F_n)_{n\in\zz_{+}}$ is regular and increasing.

For any $n\in\zz_{+}$, the \emph{Rademacher function $r_n$} is defined by setting
$$
r(x):= \left\{
         \begin{array}{ll}
           1 \quad& \hbox{if $x\in [0,\frac{1}{2})$}, \\
           -1 \quad& \hbox{if $x\in [\frac{1}{2},1)$}
         \end{array}
       \right.
$$
and, for any $x\in [0,1)$,
$$
r_n(x):=r(2^n x) .
$$
It is clear that, for any $n\in\zz_{+}$, $r_n$ is $\cf_{n+1}$ measurable.
The product system generated by the Rademacher functions is the {\it Walsh system}
$$
w_{n}:=\prod_{k\in\zz_{+}}r_k^{n_k},\qquad \forall\,n\in\zz_{+},
$$
where
\begin{equation}\label{e25}
n=\sum_{k\in\zz_{+}} n_k 2^k,\qquad n_k \in\{0,1\}.
\end{equation}

Recall (see Fine \cite{Fine1949}) that the {\it Walsh-Dirichlet kernels},
defined by setting, for any $n\in\nn$,
$$
D_n := \sum_{k=0}^{n-1} w_k,
$$
satisfy
\begin{equation*}
D_{2^k}(x) = \left\{
               \begin{array}{ll}
                 2^k \quad & \hbox{if $x \in [0,2^{-k})$,} \\
                 0 \quad & \hbox{if $x \in [2^{-k},1)$,}
               \end{array} \qquad \forall\,k\in \zz_{+}.
             \right.
\end{equation*}
If $f \in L^1[0,1)$, for any $n\in \zz_{+}$, the number
$
\widehat {f}(n) := \mathbb E(f w_{n})
$
is called the $n$th {\it Walsh--Fourier coefficient} of $f$.
We can extend this definition to martingales as follows.
If $f:=(f_{k})_{k\in\zz_{+}}$ is a martingale,
then, for any $n\in \zz_{+}$, let
$$
\widehat {f}(n) := \lim_{k \to \infty} \mathbb  E(f_{k} w_{n}).
$$
Since, for any $k$, $n\in \zz_{+}$ and $n<2^k$, $w_{n}$ is $\mathcal F_{k}$ measurable,
it can immediately be seen that the above limit does exist.
Recall that, if $f \in L^1[0,1)$, then $\mathbb E_{k}f \to f$ in the $L^1[0,1)$-norm, as
$k \to \infty$, hence, for any $n\in \zz_{+}$,
$$
\widehat {f}(n) = \lim_{k \to \infty} \mathbb E(( \mathbb E_{k}f) w_{n}).
$$
Thus, the Walsh--Fourier coefficients of $f \in L^1[0,1)$ are the same as the ones of
the martingale $(\mathbb E_{k}f)_{k\in\zz_{+}}$ obtained from $f$.

For any $n\in\nn$, denote by $s_{n}f$ the \emph{$n$th partial sum of the Walsh--Fourier series} of a martingale $f$, namely,
$$
s_{n}f := \sum_{k=0}^{n-1} \widehat {f}(k)w_{k}.
$$
If $f\in L^1[0,1)$, then, for any $n\in\nn$,
$$
s_{n} f(x) = \int_0^1 f(t) D_{n}(x \dot + t) \, dt ,
$$
where $\dot +$ denotes the dyadic addition (see, for example, Schipp et al. \cite{Schipp1990} or Golubov et al. \cite{Golubov1991}).
It is easy to see that, for any $n\in\zz_{+}$,
\begin{equation}\label{e6}
	s_{2^n}f=f_{n}
\end{equation}
and hence, by the martingale convergence theorem \cite[Theorem 1.34]{Pisier16},
we know that, for any $p\in[1,\fz)$ and $f\in L^{p}[0,1)$,
$$
\lim_{n\to\infty} s_{2^n}f=f \qquad \mbox{in the $L^{p}[0,1)$-norm}.
$$
This result was generalized by
Schipp et al. \cite[Theorem 4.1]{Schipp1990}. More
precisely, they proved that, for any $p\in(1,\fz)$ and $f\in L^{p}[0,1)$,
$$\lim_{n\to\infty} s_{n}f=f \qquad \mbox{in the $L^{p}[0,1)$-norm}.$$
Using the method of martingale transforms,
we generalize \cite[Theorem 4.1]{Schipp1990} to $L^{\varphi}[0,1)$.
\begin{theorem}\label{t7}
Let $\varphi \in A_\infty[0,1)$ be a Musielak--Orlicz function
satisfying \eqref{doob}. Then
there exists a positive constant $C$ such that, for any $f\in L^{\varphi}[0,1)$,
$$
\sup_{n\in \mathbb N} \left\|s_nf\right\|_{L^{\varphi}[0,1)} \le C \left\|f\right\|_{L^{\varphi}[0,1)}.
$$
\end{theorem}
\begin{proof}
It was proved by Schipp et al. \cite[p. 95]{Schipp1990} that,
for any $n\in\nn$,
$$
s_nf = w_n \cT_0(fw_n),
$$
where
$$
\cT_0 f := \sum_{k=1}^{\infty} n_{k-1} d_kf
$$
and the binary coefficients $n_k$ are defined as in (\ref{e25}).
Combining this, Theorem \ref{thm-tran} and $|w_n|=1$ for any $n\in\nn$,
we know that, for any $n\in\nn$,
$$
\left\|s_nf\right\|_{L^{\varphi}[0,1)}
= \left\|\mathcal T_0(fw_n)\right\|_{L^{\varphi}[0,1)}
\lesssim \left\|fw_n\right\|_{L^{\varphi}[0,1)}
\lesssim \left\|f\right\|_{L^{\varphi}[0,1)},
$$
which completes the proof of Theorem \ref{t7}.
\end{proof}

\begin{corollary}\label{c1}
Let $\varphi \in A_\infty[0,1)$ be a Musielak--Orlicz function
satisfying \eqref{doob}.
Then, for any $f\in L^{\varphi}[0,1)$,
\begin{equation}\label{e100}
	\lim_{n\to\infty} s_{n} f=f \qquad \mbox{in the $L^{\varphi}[0,1)$-norm.}
\end{equation}
\end{corollary}

\begin{proof}
	It is enough to show that the Walsh polynomials are dense in $L^{\varphi}[0,1)$. Indeed, by \eqref{e6}, $f_n$ is a Walsh polynomial. By Theorem \ref{thm-mi}, we know that, for any $n\in\zz_{+}$,
\begin{align}\label{i6}
\int_{0}^{1}\V\lf(x,|f(x)-f_n(x)|\r) \, dx
\le \int_{0}^{1}\V\lf(x,M(f-f_n)(x)\r) \, dx \sim \int_{0}^{1}\V\lf(x,S(f-f_n)(x)\r) \, dx.
\end{align}
For almost every $x\in [0,1)$, we have
$$
\lim_{n\to\infty}S(f-f_n)(x)=\lim_{n\to\infty} \left(\sum_{k=n+1}^{\infty} |d_{k+1 }f|^{2}\right)^{1/2}=0.
$$
By this and the facts that $S(f-f_n)(x) \leq S(f)(x)$ and $\V$ is of uniformly lower type $p_{\varphi}^{-}$, we know that, for almost every $x\in [0,1)$,
\begin{align*}
\lim_{n\to\infty} \V\lf(x,S(f-f_n)(x)\r) \lesssim \lim_{n\to\infty} \lf[\frac{S(f-f_n)(x)}{S(f)(x)}\r]^{p_{\varphi}^{-}} \V\lf(x,S(f)(x)\r)
=0.
\end{align*}
From this and the Lebesgue dominated convergence theorem, it follows that
\begin{align*}
\lim_{n\to\infty} \int_{0}^{1}\V\lf(x,S(f-f_n)(x)\r) \, dx =0.
\end{align*}
Combining this and \eqref{i6}, we have
\begin{align*}
\lim_{n\to\infty} \left\|f-f_n\right\|_{L^{\varphi}[0,1)} =0.
\end{align*}
This proves that the Walsh polynomials are dense in $L^{\varphi}[0,1)$.
The corollary follows from Theorem \ref{t7} with the usual density argument.
This finishes the proof of Corollary \ref{c1}.
\end{proof}

If we do not suppose that $\varphi \in A_\infty [0,1)$ satisfies \eqref{doob},
then \eqref{e100} is not true. However, to generalize the convergence result,
in this case, we can consider a summability method. Recall that, for any $n\in\nn$,
the Fej{\'e}r means $\sigma_{n}f$ of the Walsh--Fourier series of a martingale
$f$ is defined in \eqref{e9}.
Of course, $\{\sigma_{n}f\}_{n\in\nn}$ has better convergence properties than $\{s_{n} f\}_{n\in\nn}$. It is easy to show that, for any $f\in L^1[0,1)$ and $n \in \nn$,
$$
\sigma_{n}f(x) = \int_0^1 f(t) K_n(x \dot + t) \, dt ,
$$
where, for any $n \in \nn$, the {\it Walsh-Fej\'er kernel $K_n$} is defined by setting
$$
K_n := \frac{1}{n} \sum_{k=1}^n D_k .
$$

It is known (see Fine \cite{{Fine1949}} or
Schipp et al. \cite[Theorem 1.16]{Schipp1990}) that,
for any $n \in\nn$ and $x \in [0,1)$,
\begin{equation}\label{e16}
|K_n(x)| \leq \sum_{j=0}^{N-1} 2^{j-N} \sum_{i=j}^{N-1}
\left[D_{2^i}(x) + D_{2^i}(x \dot + 2^{-j-1})\right]
\end{equation}
and
\begin{equation}\label{e17}
K_{2^n}(x) = \frac{1}{2} \left[2^{-n} D_{2^n}(x) + \sum_{j=0}^{n} 2^{j-n} D_{2^n}(x \dot + 2^{-j-1})\right],
\end{equation}
where $N$ is a positive integer such that $2^{N-1} \leq n < 2^N$.

\section{The maximal Fej\'er operator on $H_{\varphi}[0,1)$\label{s7}}

We have proved in Theorem \ref{thm-mi} that, if $\V\in A_{\fz}(\Omega)$ is a Musielak--Orlicz function with uniformly lower type $p^{-}_{\V}$ and uniformly upper type $p^{+}_{\V}$ and $\{\cf_n\}_{n\in\zz_{+}}$ regular, then all the five martingale Musielak--Orlicz Hardy spaces are equivalent. In this section, we consider the dyadic $\sigma$-algebras, so $\{\cf_n\}_{n\in\zz_{+}}$ is regular. If we deal with Musielak--Orlicz Hardy spaces, we always suppose the other conditions, namely, $\V\in A_{\fz}[0,1)$ is a Musielak--Orlicz function with uniformly lower type $p^{-}_{\V}$ and uniformly upper type $p^{+}_{\V}$. Denote by $H_{\varphi}[0,1)$ one of the five martingale Musielak--Orlicz Hardy spaces.

In this section, we prove the boundedness of $\sigma_*$ from $H_{\varphi}[0,1)$ to $L^{\varphi}[0,1)$. It is known that, for any $f\in L^{\varphi}[0,1)$, the Doob maximal
operator $M(f)$ can be written as
\[
	M(f)(x)= \sup_{x\in I} \frac{1}{\MP(I)} \left| \int_{I} f \, d \MP\right|,
\]
where $I$ is a dyadic interval and $\MP$ a Lebesgue measure. Motivating by this and the kernel functions (\ref{e16}) and (\ref{e17}), we give two other dyadic maximal functions,
which are originally introduced by Jiao et al. \cite{wyong1}.
Let $f\in L^{\varphi}[0,1)$, $r\in\D$ and $\nu$ be a bounded measure. For any $n\in\nn$ and $x\in[0,1)$, let
$$
U_{\nu,r,n} (f)(x):= \sup_{I\ni x} \sum_{j=0}^{n-1} 2^{(j-n)r}
\frac{1}{\nu(I\dot + 2^{-j-1})} \left| \int_{I\dot + 2^{-j-1}} f \, d \nu\right|
$$
and
$$
V_{\nu,r,n} (f)(x):=\sup_{I\ni x} \sum_{j=0}^{n-1}\sum_{i=j}^{n-1} 2^{(j-n)r} 2^{(i-n)r}
2^{n-i}\frac{1}{\nu(I\dot + [2^{-j-1},2^{-j-1} \dot + 2^{-i}))}
\left| \int_{I\dot + [2^{-j-1},2^{-j-1} \dot + 2^{-i})} f \, d \nu \right|,
$$ where the suprema are taken over all dyadic intervals $I$ with
length $2^{-n}$, which contain $x$.
For $\nu=\MP$, we write simply $U_{r,n}$ and $V_{r,n}$. The proof of
the following theorem is similar to
that of \cite[Theorems 7.7 and 7.10]{wyong1} and the details are omitted.
\begin{theorem}\label{t12}
Let $p\in(1,\infty]$ and $\nu$ be a bounded measure. If $r\in(0,\infty)$, then there exists a positive constant $C$ such that, for any $n\in\nn$ and $f\in L^{p}([0,1),d\nu)$,
\begin{equation*}
\|U_{\nu,r,n}( f)\|_{L^{p}([0,1),d\nu)} \leq C \|f\|_{L^{p}([0,1),d\nu)}.
\end{equation*}
If $r\in(\frac12,1]$, then there exists a positive constant $C$ such that, for any $n\in\nn$
and $f\in L^{p}([0,1),d\nu)$,
\begin{equation*}
\|V_{\nu,r,n} (f)\|_{L^{p}([0,1),d\nu)} \leq C \|f\|_{L^{p}([0,1),d\nu)}.
\end{equation*}
\end{theorem}

Let $p\in(1,\fz)$. Suppose that $w \in A_p [0,1)$ and $\nu$ is the measure generated by $w$, namely, $d \nu=  w d \MP$. For any dyadic interval $I$, by $w \in A_p [0,1)$, we have
\begin{align*}
	\frac{1}{\MP(I)} \int_{I} |f|\, d\MP
	& \leq \frac{1}{\MP(I)} \left(\int_{I} |f|^{p}w \, d \MP\right)^{1/p} \left[\int_{I} w^{-1/(p-1)} \, d \MP\right]^{(p-1)/p}\\
	& \leq \left[\frac{1}{\nu(I)} \int_{I} |f|^{p} \, d \nu\right]^{1/p} \left\{\frac{1}{\MP(I)} \int_{I} w \, d \MP \left[\frac{1}{\MP(I)} \int_{I} w^{-1/(p-1)} \, d \MP\right]^{p-1}\right\}^{1/p}\\
	& \lesssim \left[\frac{1}{\nu(I)} \int_{I} |f|^{p} \, d \nu\right]^{1/p}.
\end{align*}
Using this, we can prove the next theorem in the usual way
(see, for example, Str\"omberg and Torchinsky \cite{S89}), the proof being omitted.

\begin{theorem}\label{t3}
Let $p\in(1,\fz)$ and $w \in A_p [0,1)$. If $r\in\D$, then
there exists a positive constant $C$ such that, for any $n\in\nn$
and $f\in L^{p}([0,1),w d \MP)$,
\begin{equation}\label{e5}
\|U_{r,n}(f)\|_{L^{p}([0,1),w d \lambda)} \leq C\|f\|_{L^{p}([0,1),w d \MP)}.
\end{equation}
If $r\in (\frac12,1]$, then there exists a positive constant $C$ such that, for any $n\in\nn$
and $f\in L^{p}([0,1),w d \MP)$,
\begin{equation}\label{e22}
\|V_{r,n}(f)\|_{L^{p}([0,1),w d \lambda)} \leq C \|f\|_{L^{p}([0,1),w d \MP)}.
\end{equation}
\end{theorem}

\begin{theorem}\label{t4}
	Let $\varphi \in A_\infty[0,1)$ be a Musielak--Orlicz function with uniformly lower type $p_\varphi^{-}$ and uniformly upper type $p_\varphi^{+}$ satisfying \eqref{doob}.
	If $r\in\D$, then there exists a positive constant $C$ such that, for any $n\in\nn$ and $f \in L^{\varphi}[0,1)$,
	\[
		\left\|U_{r,n}(f)\right\|_{L^{\varphi}[0,1)} \le C \left\|f\right\|_{L^{\varphi}[0,1)}.
	\]
	Moreover, if $r\in (\frac12,1]$, then there exists a positive constant $C$ such that, for any $n\in\nn$ and $f \in L^{\varphi}[0,1)$,
	\[
		\left\|V_{r,n}(f)\right\|_{L^{\varphi}[0,1)} \le C \left\|f\right\|_{L^{\varphi}[0,1)}.
	\]
\end{theorem}

\begin{proof}
	Using Theorems \ref{thm-ip} and \ref{t3}, we can show Theorem \ref{t4}
by the way same as the proof of Theorem \ref{thm-doob}, the details being omitted.
This finishes the proof of Theorem \ref{t4}.
\end{proof}

Next we give a sufficient condition for a $\sigma$-sublinear operator to be bounded from $H_{\varphi}[0,1)$ to $L^{\varphi}[0,1)$.

\begin{theorem}\label{t1}
Let $\V\in A_{\fz} [0,1)$ be a Musielak--Orlicz function with uniformly lower
type $p^{-}_{\V}$ and uniformly upper type $p^{+}_{\V}$. Suppose that the
Doob maximal operator $M$ is bounded on $L^{\varphi^{*}_{1/r}}[0,1)$
for some $r\in(0,\min\{1, \wedge p^{-}_{\V}\}]$.
Suppose further that the $\sigma$-sublinear operator $T:L^\infty [0,1)\rightarrow L^\infty [0,1)$ is bounded and
\begin{equation}\label{e1}
\left\|\sum_{k \in \zz} \lf(\mu^k\r)^{r} \lf|T(a^k)\r|^{r}\mathbf{1}_{\{\tau_k=\infty\}}\right\|_{L^{\varphi_{1/r}}[0,1)}\le C\left\|\sum_{k \in \zz} 2^{kr} \mathbf{1}_{\{\tau_k<\infty\}}\right\|_{L^{\varphi_{1/r}}[0,1)},
\end{equation}
where $C$ is a positive constant and, for any $k\in\zz$, $\tau_k$ is the stopping time associated with the $(\varphi,\fz)_M$-atom $a^k$ and $\mu^k=3\cdot 2^k \|\mathbf{1}_{\{\tau_k<\infty\}}\|_{L^{\varphi}[0,1)}$.
Then there exists a positive constant $C$ such that, for any $f\in H_{\varphi}[0,1)$,
$$
\|T(f)\|_{L^{\varphi}[0,1)}\le C \|f\|_{H_{\varphi}[0,1)}.
$$
\end{theorem}

\begin{proof}
Let $f\in H_{\varphi}[0,1)$. By Corollary \ref{cor-mi} and Theorem \ref{Thm-atmSM},
we know that there exist a sequence $(a^k)_{k \in \zz}$ of $(\varphi,\fz)_M$-atoms
such that
\begin{align}\label{f2}
f=\sum_{k \in \zz} \mu^k a^k\qquad \mbox{and}\qquad \left\|\left[\sum_{k \in \zz} \lf(3\cdot 2^k\r)^{r} \mathbf{1}_{\{\tau_k<\infty\}}\right]^{1/r} \right \|_{L^{\varphi}[0,1)}
\lesssim \|f\|_{H_{\varphi}[0,1)},
\end{align}
where the sequence $(\tau_k)_{k\in\mathbb Z}$ are stopping times,
respectively, associated with $(a^k)_{k \in \zz}$ and, for any $k\in\zz$,
$\mu^k:=3\cdot 2^k \|\mathbf{1}_{\{\tau_k<\infty\}}\|_{L^{\varphi}[0,1)}$. Then we have
\begin{align*}
\|T(f)\|_{L^{\varphi}[0,1)} \le \left\| \sum_{k \in \zz} \mu^k T(a^k)\mathbf{1}_{\{\tau_k<\infty\}} \right\|_{L^{\varphi}[0,1)} +\left\| \sum_{k \in  \zz} \mu^k T(a^k)\mathbf{1}_{\{\tau_k=\infty\}} \right\|_{L^{\varphi}[0,1)}
 =:{\rm B_1}+{\rm B_2}.
\end{align*}

We first estimate  ${\rm B_1}$. Notice that, for any $k\in\zz$, the sets $\left\{\tau_k=\ell\right\}_{\ell\in \mathbb Z_{+}}$ are disjoint and,
for any $\ell\in \mathbb Z_{+}$, there exist a finite set $\Lambda_{k,\ell}$
and disjoint atoms
$(I_{k,\ell,m})_{m\in\Lambda_{k,\ell}}\subset \mathcal{F}_{\ell}$ such that
$\left\{\tau_k=\ell\right\}=\bigcup_{m\in\Lambda_{k,\ell}} I_{k,\ell,m}.$
Thus, for any $k\in\zz$, we have
\begin{equation}\label{e3}
\{\tau_k<\infty\}=\bigcup_{\ell\in \mathbb Z_{+}} \bigcup_{m\in\lkl} I_{k,\ell,m},
\end{equation}
where $\{I_{k,\ell,m}\}_{\ell\in \mathbb Z_{+},\,m\in\Lambda_{k,\ell}}$ are disjoint for fixed $k$. From this and $r\in(0, 1]$, we deduce that
\begin{align}\label{f1}
{\rm B_1}\leq  \left\|\sum_{k \in \zz}\lf(\mu^k\r)^{r} \sum_{\ell \in \zz_{+}} \sum_{m\in\lkl} \lf|T(a^k)\r|^{r}\mathbf{1}_{I_{k,\ell,m}}\right\|_{L^{\varphi_{1/r}}[0,1)}^{1/r}.
\end{align}
Since $r\le p_{\V}^{-}$, by Lemma \ref{l2} and Remark \ref{rem-uni},
we know that $\V_{1/r}$ is of uniformly lower type $1$.
By this and Lemma \ref{lem-du}, we can choose a function $g\in L^{\varphi^{*}_{1/r}}[0,1)$
with norm less than or equal to $1$ such that
$$
\left\|\sum_{k \in \zz}\lf(\mu^k\r)^{r} \sum_{\ell \in \zz_{+}} \sum_{m\in\lkl} \lf|T(a^k)\r|^{r}\mathbf{1}_{I_{k,\ell,m}}\right\|_{L^{\varphi_{1/r}}[0,1)}
\lesssim\int_0^{1} \sum_{k \in \zz}\lf(\mu^k\r)^{r} \sum_{\ell \in \zz_{+}} \sum_{m\in\lkl} \lf|T(a^k)\r|^{r}\mathbf{1}_{I_{k,\ell,m}} g \,d\mathbb P.
$$
Combining this, the boundedness of $T$ on $L^\infty [0,1)$ and \eqref{f1}, we obtain
\begin{align*}
\lf({\rm B_1}\r)^r& \leq \int_0^{1} \sum_{k \in \zz}\lf(\mu^k\r)^{r} \sum_{\ell \in \zz_{+}} \sum_{m\in\lkl} \lf|T(a^k)\r|^{r}\mathbf{1}_{I_{k,\ell,m}} g \, d\mathbb P\\
 &\lesssim \sum_{k \in \zz}\lf(\mu^k\r)^{r} \sum_{\ell \in \zz_{+}} \sum_{m\in\lkl} \lf\|T(a^k)\r\|^{r}_{L^{\infty}[0,1)} \lf\|\mathbf{1}_{I_{k,\ell,m}} g \r\|_{L^1(\Omega)}\\
&\lesssim \sum_{k \in \zz} \sum_{\ell \in \zz_{+}} \sum_{m\in\lkl} \lf(3\cdot 2^k\r)^{r} \lf\|\mathbf{1}_{\{\tau_k<\infty\}}\r\|_{L^{\varphi}[0,1)}^{r} \lf\|a^k\r\|_{L^{\infty}[0,1)}^{r} \lf\|\mathbf{1}_{I_{k,\ell,m}} g \r\|_{L^1(\Omega)}\\
&\lesssim \sum_{k \in \zz} \sum_{\ell \in \zz_{+}} \sum_{m\in\lkl} \lf(3\cdot 2^k\r)^{r}
\mathbb P\lf(I_{k,\ell,m}\r)
\left[\frac{1}{\mathbb P(I_{k,\ell,m})}\int_{I_{k,\ell,m}}|g| \, d\MP\right],
\end{align*}
which, together with the definition of the maximal operator and Lemma \ref{lem-du},
further implies that
\begin{align*}
\lf({\rm B_1}\r)^r & \lesssim \sum_{k \in \zz} \sum_{\ell \in \zz_{+}} \sum_{m\in\lkl} \lf(3\cdot 2^k\r)^{r} \int_0^{1}\mathbf{1}_{I_{k,\ell,m}} M(g) \, d\mathbb P \\
& \lesssim \left\|\sum_{k \in \zz} \sum_{\ell \in \zz_{+}} \sum_{m\in\lkl}
\lf(3\cdot 2^k\r)^{r}\mathbf{1}_{I_{k,\ell,m}} \right \|_{L^{\varphi_{1/r}}[0,1)} \lf\|M(g)\r\|_{L^{\varphi^{*}_{1/r}}[0,1)}.
\end{align*}
From this, the boundedness of $M$ on $L^{\varphi^{*}_{1/r}}[0,1)$ and \eqref{f2}, it follows that
\begin{align*}
{\rm B_1} & \lesssim \left\|\sum_{k \in \zz} \sum_{\ell \in \zz_{+}} \sum_{m\in\lkl} \lf(3\cdot 2^k\r)^{r}\mathbf{1}_{I_{k,\ell,m}} \right \|_{L^{\varphi_{1/r}}[0,1)}^{1/r} \|g\|_{L^{\varphi^{*}_{1/r}}[0,1)}^{1/r} \\
& \lesssim \left\|\sum_{k \in \zz} \lf(3\cdot 2^k\r)^{r} \mathbf{1}_{\{\tau_k<\infty\}} \right \|_{L^{\varphi_{1/r}}[0,1)}^{1/r}
\sim \left\|\left[\sum_{k \in \zz} \lf(3\cdot 2^k\r)^{r} \mathbf{1}_{\{\tau_k<\infty\}}\right]^{1/r} \right \|_{L^{\varphi}[0,1)}
\lesssim \|f\|_{H_{\varphi}[0,1)}.
\end{align*}

On another hand, by \eqref{e1} and \eqref{f2}, we have
\begin{align*}
{\rm B}_2 &\leq \left\|\sum_{k \in \zz}\lf(\mu^k\r)^{r}\lf|T(a^k)\r|^{r}\mathbf{1}_{\{\tau_k=\infty\}}\right\|_{L^{\varphi_{1/r}}[0,1)}^{1/r} \lesssim \left\|\sum_{k \in \zz} 2^{kr} \mathbf{1}_{\{\tau_k<\infty\}}\right\|_{L^{\varphi_{1/r}}[0,1)}^{1/r} \lesssim \|f\|_{H_{\varphi}[0,1)},
\end{align*}
which completes the proof Theorem \ref{t1}.
\end{proof}

\begin{theorem}\label{t6}
Let $\V\in A_{\fz} [0,1)$ be a Musielak--Orlicz function with uniformly
lower type $p^{-}_{\V}$ and uniformly upper type $p^{+}_{\V}$. Suppose that
the Doob maximal operator $M$ is bounded on $L^{\varphi^{*}_{1/r}}[0,1)$ for
some $r\in(\frac12,1 \wedge p^{-}_{\V}]$.
Then there exists a positive constant $C$ such that
\begin{equation}\label{e4}
\left\|\sum_{k \in \zz}\lf(\mu^k\r)^{r} \lf[\sigma_*(a^k)\r]^{r}\mathbf{1}_{\{\tau_k=\infty\}}\right\|_{L^{\varphi_{1/r}}[0,1)}
\le C\left\|\sum_{k \in \zz} 2^{kr} \mathbf{1}_{\{\tau_k<\infty\}}\right\|_{L^{\varphi_{1/r}}[0,1)},
\end{equation}
where, for any $k\in\zz$, $\tau_k$ is the stopping time associated
with the $(\varphi,\fz)_M$-atom $a^k$ and
$\mu^k:=3\cdot 2^k \|\mathbf{1}_{\{\tau_k<\infty\}}\|_{L^{\varphi}[0,1)}$.
\end{theorem}

\begin{proof}
For any $k\in\zz$ and $\ell \in \zz_{+}$, let $\lkl$ be as in \eqref{e3}. Moreover,
for any $m\in\lkl$, let $I_{k,\ell,m}$ be also as in \eqref{e3} and
$K_{k,\ell,m}\in\nn$ satisfying $\mathbb{P}(I_{k,\ell,m})=2^{-K_{k,\ell,m}}$. It was proved in \cite[Theorem 7.14]{wyong1} that, for any $k\in\zz$ and $x \in \{\tau_k=\infty\}$,
\begin{align}\label{e7}
\sigma_*(a^{k})(x)
&\lesssim \lf\|\mathbf{1}_{\{\tau_k<\infty\}}\r\|_{L^{\varphi}[0,1)}^{-1}\lf[\sum_{\ell \in \zz_{+}} \sum_{m\in\lkl} \sum_{j=0}^{K_{k,\ell,m}-1} 2^{j-K_{k,\ell,m}} \mathbf{1}_{I_{k,\ell,m}\dot + 2^{-j-1}}(x)  \r.\\
&
\quad +\lf.\sum_{\ell \in \zz_{+}} \sum_{m\in\lkl} \sum_{j=0}^{K_{k,\ell,m}-1} 2^{j-K_{k,\ell,m}} \sum_{i=j}^{K_{k,\ell,m}-1} 2^{i-K_{k,\ell,m}} \mathbf{1}_{I_{k,\ell,m}\dot + [2^{-j-1},2^{-j-1} \dot + 2^{-i})}(x)\r] \noz \\
&=: \lf\|\mathbf{1}_{\{\tau_k<\infty\}}\r\|_{L^{\varphi}[0,1)}^{-1} \left[{\rm A}_k(x)+{\rm B}_k(x)\right],\noz
\end{align}
which implies that
\begin{align}\label{zz}
\left\|\sum_{k \in \zz}\lf(\mu^k\r)^{r} \lf[\sigma_*(a^k)\r]^{r}\mathbf{1}_{\{\tau_k=\infty\}}\right\|_{L^{\varphi_{1/r}}[0,1)}
&\lesssim \left\|\sum_{k \in \zz} 2^{kr}\lf({\rm A}_k\r)^r\right\|_{L^{\varphi_{1/r}}[0,1)}+\left\|\sum_{k \in \zz} 2^{kr} \lf({\rm B}_k\r)^r\right\|_{L^{\varphi_{1/r}}[0,1)} \\
&=:{\rm Z}_1+{\rm Z}_2\noz.
\end{align}

By Lemma \ref{lem-du}, there exists a function $g\in L^{\varphi^{*}_{1/r}}[0,1)$ with norm
less than or equal to $1$ such that
\begin{align*}
Z_1 &\lesssim \int_0^{1} \sum_{k \in \zz} 2^{kr} \lf(A_k\r)^{r} g \, d\mathbb P
\sim \int_0^{1} \sum_{k \in \zz} 2^{kr} \sum_{\ell \in \zz_{+}} \sum_{m\in\lkl} \sum_{j=0}^{K_{k,\ell,m}-1} 2^{(j-K_{k,\ell,m})r} \mathbf{1}_{I_{k,\ell,m}\dot + 2^{-j-1}} g \, d\mathbb P \\
& \lesssim  \sum_{k \in \zz} 2^{kr} \sum_{\ell \in \zz_{+}} \sum_{m\in\lkl} \sum_{j=0}^{K_{{k,l,m}}-1} 2^{(j-K_{{k,l,m}})r} \left|\int_{I_{k,\ell,m}\dot + 2^{-j-1}} g \, d\mathbb P\right| \\
& \lesssim  \sum_{k \in \zz} 2^{kr} \sum_{\ell \in \zz_{+}} \sum_{m\in\lkl} \sum_{j=0}^{K_{k,\ell,m}-1} 2^{(j-K_{k,\ell,m})r} \int_0^{1} \mathbf{1}_{I_{k,\ell,m}} \frac{1}{\mathbb P(I_{k,\ell,m}\dot + 2^{-j-1})} \left|\int_{I_{k,\ell,m}\dot + 2^{-j-1}} g \, d\mathbb P\right|\, d\mathbb P,
\end{align*}
where the last inequality follows from $\mathbb P(I_{k,\ell,m})=\mathbb P(I_{k,\ell,m}\dot + 2^{-j-1})=2^{-K_{k,\ell,m}}$. From this and Lemma \ref{lem-du}, we deduce that
\begin{align*}
{\rm Z}_1&\lesssim \int_0^{1} \sum_{k \in \zz} 2^{kr} \sum_{\ell \in \zz_{+}} \sum_{m\in\lkl} \mathbf{1}_{I_{k,\ell,m}} \sum_{j=0}^{K_{k,\ell,m}-1} 2^{(j-K_{k,\ell,m})r} \frac{1}{\mathbb P(I_{k,\ell,m}\dot + 2^{-j-1})} \left|\int_{I_{k,\ell,m}\dot + 2^{-j-1}} g \, d\mathbb P\right|\, d\mathbb P \\
&\lesssim \int_0^{1} \sum_{k \in \zz} 2^{kr} \sum_{\ell \in \zz_{+}} \sum_{m\in\lkl} \mathbf{1}_{I_{k,\ell,m}} U_{r,K_{k,\ell,m}}(g) \, d\mathbb P\\
&\lesssim \left \|\sum_{k \in \zz } 2^{kr} \sum_{\ell \in \zz_{+}} \sum_{m\in\lkl} \mathbf{1}_{I_{k,\ell,m}} \right \|_{L^{\varphi_{1/r}}[0,1)} \left\| U_{r,K_{k,\ell,m}}(g)\right\|_{{L^{\varphi^{*}_{1/r}}[0,1)}}.
\end{align*}
By this, $\|g\|_{L^{\varphi^{*}_{1/r}}[0,1)}\le1$ and Theorem \ref{t4}, we conclude that
\begin{align}\label{zzz}
{\rm Z}_1 &\lesssim \left\|\sum_{k \in \zz} \sum_{\ell \in \zz_{+}} \sum_{m\in\lkl} 2^{kr}\mathbf{1}_{I_{k,\ell,m}} \right \|_{L^{\varphi_{1/r}}[0,1)} \|g\|_{L^{\varphi^{*}_{1/r}}[0,1)}
\lesssim \left\|\sum_{k \in \zz} 2^{kr} \mathbf{1}_{\{\tau_k<\infty\}}\right\|_{L^{\varphi_{1/r}}[0,1)}.
\end{align}

For ${\rm Z}_2$, we choose again a function $g\in L^{\varphi^{*}_{1/r}}[0,1)$ with $\|g\|_{L^{\varphi^{*}_{1/r}}[0,1)}\leq 1$ such that
\begin{align}\label{zz1}
{\rm Z}_2\lesssim \int_0^{1} \sum_{k \in \zz} 2^{kr}\lf( B_k\r)^{r} g \, d\mathbb P.
\end{align}
It is clear that, for any $k\in\zz$, $i$, $j$, $\ell\in\zz_{+}$, $m\in\lkl$
and $j\le i\le K_{k,\ell,m}-1$,
$$\mathbb P(I_{k,\ell,m}\dot + [2^{-j-1},2^{-j-1} \dot + 2^{-i}))=2^{-i}.$$
From this and \eqref{zz1}, it follows that
\begin{align*}
{\rm Z}_2 &\lesssim \int_0^{1} \sum_{k \in \zz} 2^{kr} \sum_{\ell \in \zz_{+}} \sum_{m\in\lkl} \sum_{j=0}^{K_{k,\ell,m}-1}\sum_{i=j}^{K_{k,\ell,m}-1}  2^{(j-K_{k,\ell,m})r} 2^{(i-K_{k,\ell,m})r} \mathbf{1}_{I_{k,\ell,m}\dot + [2^{-j-1},2^{-j-1} \dot + 2^{-i})} g \, d\mathbb P \\
&\lesssim \sum_{k \in \zz} 2^{kr} \sum_{\ell \in \zz_{+}} \sum_{m\in\lkl} \sum_{j=0}^{K_{k,\ell,m}-1}\sum_{i=j}^{K_{k,\ell,m}-1} 2^{(j-K_{k,\ell,m})r} 2^{(i-K_{k,\ell,m})r} \left| \int_{I_{k,\ell,m}\dot + [2^{-j-1},2^{-j-1} \dot + 2^{-i})} g\, d\mathbb P\right|\\
&\lesssim \int_0^{1} \sum_{k \in \zz} 2^{kr} \sum_{\ell \in \zz_{+}} \sum_{m\in\lkl} \sum_{j=0}^{K_{k,\ell,m}-1}\sum_{i=j}^{K_{k,\ell,m}-1} 2^{(j-K_{k,\ell,m})r} 2^{(i-K_{k,\ell,m})r} 2^{K_{k,\ell,m}-i} \mathbf{1}_{I_{k,\ell,m}} \\
& \qquad\times \frac{1}{\mathbb P(I_{k,\ell,m}\dot + [2^{-j-1},2^{-j-1} \dot + 2^{-i}))} \left| \int_{I_{k,\ell,m}\dot + [2^{-j-1},2^{-j-1} \dot + 2^{-i})} g\, d\mathbb P\right| d\mathbb P,
\end{align*}
which implies that
\begin{align*}
{\rm Z}_2 &\lesssim\int_0^{1} \sum_{k \in \zz} 2^{kr} \sum_{\ell \in \zz_{+}} \sum_{m\in\lkl} \mathbf{1}_{I_{{k,\ell,m}}} V_{r,K_{k,\ell,m}}(g) \, d\mathbb P.
\end{align*}
Combining this, Lemma \ref{lem-du} and Theorem \ref{t4}, we conclude that
\begin{align*}
{\rm Z}_2 &\lesssim \left \|\sum_{k \in \zz } 2^{kr} \sum_{\ell \in \zz_{+}} \sum_{m\in\lkl} \mathbf{1}_{I_{{k,\ell,m}}} \right \|_{{L^{\varphi_{1/r}}[0,1)}} \left\|V_{r,K_{k,\ell,m}}(g)\right\|_{{L^{\varphi^{*}_{1/r}}[0,1)}}\\
&\lesssim \left\|\sum_{k \in \zz} \sum_{\ell \in \zz_{+}} \sum_{m\in\lkl} 2^{kr}\mathbf{1}_{I_{{k,\ell,m}}} \right \|_{L^{\varphi_{1/r}}[0,1)} \|g\|_{L^{\varphi^{*}_{1/r}}[0,1)}
\lesssim \left\|\sum_{k \in \zz} 2^{kr} \mathbf{1}_{\{\tau_k<\infty\}}\right\|_{L^{\varphi_{1/r}}[0,1)}.
\end{align*}
From this, \eqref{zz} and \eqref{zzz}, we deduce the desired conclusion. This finishes the proof of
Theorem \ref{t6}.
\end{proof}

\begin{remark}
Notice that, under the assumptions of Theorem \ref{t6}, we have $p^{-}_{\V}\in(\frac12,\fz).$
\end{remark}

\begin{theorem}\label{t5}
Let $\V\in A_{\fz}[0,1)$ be a Musielak--Orlicz function with uniformly lower type $p^{-}_{\V}$ and uniformly upper type $p^{+}_{\V}$. Suppose that the Doob maximal operator $M$ is bounded on $L^{\varphi^{*}_{1/r}}[0,1)$ for some $r\in(\frac12,\min\{ 1, p^{-}_{\V}\}]$. Then
there exists a positive constant $C$ such that, for any $f\in H_{\varphi}[0,1)$,
$$
\left\|\sigma_*f\right\|_{L^{\varphi}[0,1)}
\le C \left\|f\right\|_{H_{\varphi}[0,1)} .
$$
\end{theorem}
\begin{proof}
The conclusions of this theorem follows from Theorems \ref{t1} and \ref{t6} immediately.
\end{proof}
\begin{remark}\label{rem-mf}
If the condition, that the Doob maximal operator is bounded
on $L^{\varphi^{*}_{1/r}}[0,1)$ for some $r\in(\frac12,\min\{ 1, p^{-}_{\V}\}]$ in Theorem \ref{t5},
is replaced by the condition that $ \varphi^{*}_{1/r}\in A_\infty [0,1)$ and
\begin{equation}\label{e8}
	q(\varphi^{*}_{1/r}) < (p^{+}_{\V}/r)' \leq (p^{-}_{\V}/r)'<\infty
\end{equation}
for some $r\in(\frac12,\min\{ 1, p^{-}_{\V}\}]$,
then, from Corollary \ref{thm-dooc}, we deduce that
the conclusion of Theorem \ref{t5} still holds true.
\end{remark}

Notice that Theorem \ref{t5} is new even for the weighted Hardy spaces [namely, for any given $p\in(\frac12,\fz)$, $\varphi(x,t):=w(x)t^{p}$ for any $x\in[0,1)$ and $t\in\D$] and for Orlicz Hardy spaces [namely, $\varphi(x,t)=\Phi(t)$ for any $x\in[0,1)$ and $t\in\D$, here $\Phi$ is an Orlicz function].

\begin{theorem}\label{t8}
Let $p \in (\frac12,\infty)$ and $w \in A_{2p}[0,1)$ be a special weight. Let $\varphi(x,t):=w(x)t^{p}$ for any $x\in[0,1)$ and $t\in\D$.
Then there exists a positive constant $C$ such that, for any $f\in H_{\varphi}[0,1)$,
$$
\left\|\sigma_*f\right\|_{L^{\varphi}[0,1)}
\le C \left\|f\right\|_{H_{\varphi}[0,1)} .
$$
\end{theorem}
\begin{proof}
Clearly, $\V$ is of uniformly lower type $p$ and of uniformly upper type $p$.
It is well known that $w \in A_{2p} [0,1)$ implies that there exists $r\in(\frac12,\min\{1,p\})$ such that $w \in A_{p/r}[0,1)$. By Example \ref{x1}(i), we know that, for any $x\in[0,1)$ and $t\in\D$,
$$\V^{*}_{1/r}(x,t)=[w(x)]^{\frac{-1}{p/r-1}}t^{(p/r)'}\lf({p}/{r}\r)^{\frac{-1}{p/r-1}}\frac1{(p/r)'}.$$ Observe that $\V^{*}_{1/r}$ is of uniformly lower type $(p/r)'$ and of uniformly upper type $(p/r)'$. We can easily check that $w \in A_{p/r}[0,1)$ implies
$w^{\frac{-1}{p/r-1}} \in A_{(p/r)'}[0,1)$. From this, it follows that
$\V^{*}_{1/r}\in A_{(p/r)'}[0,1)$.
Then there exists $\varepsilon\in\D$ such that
$\varphi^{*}_{1/r} \in A_{(p/r)'-\varepsilon}[0,1)$. Thus, we have
$q(\varphi^{*}_{1/r})<(p/r)'$ and hence $\varphi^{*}_{1/r}$
satisfies \eqref{e8}. The desired conclusion of the theorem follows from Remark \ref{rem-mf}.
This finishes the proof of Theorem \ref{t8}.
\end{proof}

Theorem \ref{t5} gives back the well known result of Weisz (\cite{wces3,wk2}) mentioned above when $\varphi(x,t):=t^{p}$ for any $x\in[0,1)$ and $t\in\D$, with $p\in(\frac12,\infty)$. For $p=1$, it is due to Fujii \cite{Fujii1979}
(see also Schipp et al. \cite{Schipp1981}).
If $\varphi(x,t):=t^{p}$ for any $x\in[0,1)$ and $t\in\D$, with $p\in(0,\frac12]$, then Theorem \ref{t5} is not true anymore (see Simon et al. \cite{wgyenge}, Simon \cite{Simon2000} and G{\'a}t et al. \cite{Gat2009}).

\begin{theorem}\label{t10}
Let $\Phi$ be an Orlicz function with lower type $p^{-}_{\Phi}$ and upper
type $p^{+}_{\Phi}$ and $\varphi(x,t):=\Phi(t)$ for any $x\in[0,1)$ and
$t\in\D$. If $p^{-}_{\Phi} \in (1/2,\infty)$, then there exists a
positive constant $C$ such that, for any $f\in H_{\varphi}[0,1)$,
$$
\left\|\sigma_*f\right\|_{L^{\varphi}[0,1)}\le C \left\|f\right\|_{H_{\varphi}[0,1)}.
$$
\end{theorem}

\begin{proof}
Choosing $r\in(\frac12,\min\{1, p^{-}_{\Phi}\})$, we can apply Remark \ref{rem-mf}.
Then the desired conclusion of this theorem follows from
Lemmas \ref{l2} and \ref{l20} and the fact that $q(\varphi^{*}_{1/r})=1$.
This finishes the proof of Theorem \ref{t10}.
\end{proof}

By standard arguments (see, for example, Jiao et al. \cite{wyong1}),
Theorem \ref{t5} implies the next convergence results of $(\sigma_nf)_{n\in\nn}$.
We state these convergence results only in the general case, more exactly,
under the assumptions of Theorem \ref{t5}. Let $p \in (\frac12,\infty)$,
$w \in A_{2p}[0,1)$ be a special weight and $\Phi$ an Orlicz function.
Obviously, we could formulate the convergence results if $\varphi(x,t):=w(x)t^{p}$
or if $\varphi(x,t):=\Phi(t)$ for any $x\in[0,1)$ and $t\in\D$, in other words,
under the assumptions of Theorems \ref{t8} or \ref{t10}, which are new results as well.

\begin{corollary} \label{c22}
Let $\V\in A_{\fz} [0,1)$ be a Musielak--Orlicz function with uniformly lower type $p^{-}_{\V}$ and uniformly upper type $p^{+}_{\V}$. Suppose that the Doob maximal operator $M$ is bounded on $L^{\varphi^{*}_{1/r}}[0,1)$ for some $r\in(\frac12,\min\{1, p^{-}_{\V}\}]$.
If $f\in H_{\varphi}[0,1)$, then $\sigma_{n} f$ converges almost everywhere
on $[0,1)$ as well as in the $L^{\varphi}[0,1)$-norm as $n\to\infty$.
\end{corollary}

For any integrable function $f$, the limit of $(\sigma_nf)_{n\in\nn}$ is exactly the function.
For any $k\in\zz_{+}$, let $I\in \F_k$ be an atom of $\F_k$. The \emph{restriction of a martingale $f$ to the atom $I$} is defined by setting, for any $n\in\zz_{+}$,
$$
f\mathbf{1}_I:=(\mathbb E_nf \mathbf{1}_I,n\geq k).
$$

\begin{corollary}\label{c50}
Let $\V\in A_{\fz} [0,1)$ be a Musielak--Orlicz function with uniformly
lower type $p^{-}_{\V}$ and uniformly upper type $p^{+}_{\V}$. Suppose that
the Doob maximal operator $M$ is bounded on $L^{\varphi^{*}_{1/r}}[0,1)$
for some $r\in(\frac12,\min\{1, p^{-}_{\V}\}]$. If $f\in H_{\varphi}[0,1)$
and there exists a dyadic interval $I$ such that the restriction
$f \mathbf{1}_I \in L^1(I)$, then
$$
\lim_{n\to\infty}\sigma_{n} f(x)=f(x) \qquad \mbox{for almost every $x\in I$ as well as in the $L^{\varphi}(I)$-norm.}
$$
\end{corollary}

\begin{corollary}\label{c51}
Let $\V\in A_{\fz} [0,1)$ be a Musielak--Orlicz function with uniformly
lower type $p^{-}_{\V}$ and uniformly upper type $p^{+}_{\V}$ satisfying
\eqref{doob}. Suppose that the Doob maximal operator $M$ is bounded on
$L^{\varphi^{*}_{1/r}}[0,1)$ for some $r\in(\frac12,\min\{1, p^{-}_{\V}\}]$.
If $f\in L^{\varphi}[0,1)$, then
$$
\lim_{n\to\infty}\sigma_{n} f(x)=f(x) \qquad \mbox{for almost every
$x\in [0,1)$ as well as in the $L^{\varphi} [0,1)$-norm.}
$$
\end{corollary}
\begin{proof}
By Theorem \ref{thm-doob} and $f\in L^{\varphi}[0,1)$, we know that
$f\in H_{\varphi}[0,1)$ and hence $f$ is integrable. The desired conclusion
follows from Corollary \ref{c50}.
This finishes the proof of Corollary \ref{c51}.
\end{proof}

Considering only $(\sigma_{2^n})_{n\in\nn}$, we do not need the restriction $r\in(\frac12, 1]$.

\begin{theorem} \label{t9}
Let $\V\in A_{\fz} [0,1)$ be a Musielak--Orlicz function with uniformly lower type $p^{-}_{\V}$ and uniformly upper type $p^{+}_{\V}$. Suppose that the Doob maximal operator $M$ is bounded on $L^{\varphi^{*}_{1/r}}[0,1)$ for some $r\in(0,\min\{1, p^{-}_{\V}\}]$. Then
there exists a positive constant $C$ such that
\begin{equation*}
\left\|\sum_{k \in \zz}\lf(\mu^k\r)^{r} \sup_{n\in\mathbb N}\lf|\sigma_{2^n}(a^k)\r|^{r}\mathbf{1}_{\{\tau_k=\infty\}}\right\|_{L^{\varphi_{1/r}}[0,1)}\le C\left\|\sum_{k \in \zz} 2^{kr} \mathbf{1}_{\{\tau_k<\infty\}}\right\|_{L^{\varphi_{1/r}}[0,1)},
\end{equation*}
where, for any $k\in\zz$, $\tau_k$ is the stopping time associated
with the $(\varphi,\fz)_M$-atom $a^k$ and
$\mu^k:=3\cdot 2^k \|\mathbf{1}_{\{\tau_k<\infty\}}\|_{L^{\varphi}[0,1)}$.
\end{theorem}

\begin{proof}
Similarly to \eqref{e7},
for any $k\in\zz$ and $\ell \in \zz_{+}$, let $\lkl$ be defined as in \eqref{e3}. Moreover,
for any $m\in\lkl$, let $I_{k,\ell,m}$ be defined also as in \eqref{e3} and
$K_{k,\ell,m}\in\nn$ satisfying $\mathbb{P}(I_{k,\ell,m})=2^{-K_{k,\ell,m}}$. It was proved in \cite[Theorem 7.14]{wyong1} that, for any $k\in\zz$ and $x \in \{\tau_k=\infty\}$,
\begin{align*}
\sup_{n\in \nn}\lf|\sigma_{2^n}(a^{k})(x)\r|
\lesssim \|\mathbf{1}_{\{\tau<\infty\}}\|_{L^{\varphi}[0,1)}^{-1} \sum_{\ell \in \zz_{+}} \sum_{m\in\lkl} \sum_{j=0}^{K_{k,\ell,m}-1} 2^{j-K_{k,\ell,m}} \mathbf{1}_{I_{k,\ell,m}\dot + 2^{-j-1}}(x).
\end{align*}
Then the proof of this theorem can be finished as in Theorem \ref{t6}.
\end{proof}

We deduce the next result from this and Theorem \ref{t1}.
\begin{theorem} \label{t5 Leb}
Let $\V\in A_{\fz} [0,1)$ be a Musielak--Orlicz function with uniformly lower type $p^{-}_{\V}$ and uniformly upper type $p^{+}_{\V}$. Suppose that the Doob maximal operator $M$ is bounded on $L^{\varphi^{*}_{1/r}}[0,1)$ for some $r\in(0,\min\{1, p^{-}_{\V}\}]$. Then
there exists a positive constant $C$ such that, for any $f\in H_{\varphi}[0,1)$,
$$
\left\|\sup_{n\in \mathbb N}|\sigma_{2^n}f|\right\|_{L^{\varphi}[0,1)}\le C \left\|f\right\|_{H_{\varphi}[0,1)}.
$$
\end{theorem}

The following corollaries can be proved as Corollaries \ref{c22}, \ref{c50} and \ref{c51}.

\begin{corollary} \label{c21}
Let $\V\in A_{\fz} [0,1)$ be a Musielak--Orlicz function with uniformly lower type $p^{-}_{\V}$ and uniformly upper type $p^{+}_{\V}$. Suppose that the Doob maximal operator $M$ is bounded on $L^{\varphi^{*}_{1/r}}[0,1)$ for some $r\in(0,\min\{1, p^{-}_{\V}\}]$.
If $f\in H_{\varphi}[0,1)$, then $\sigma_{2^n}f$ converges almost everywhere on $[0,1)$ as well as in the $L^{\varphi}[0,1)$-norm as $n\to\infty$.
\end{corollary}

\begin{corollary}\label{c52}
Let $\V\in A_{\fz} [0,1)$ be a Musielak--Orlicz function with uniformly lower type $p^{-}_{\V}$ and uniformly upper type $p^{+}_{\V}$. Suppose that the Doob maximal operator $M$ is bounded on $L^{\varphi^{*}_{1/r}}[0,1)$ for some $r\in(0,\min\{1, p^{-}_{\V}\}]$.
If $f\in H_{\varphi}[0,1)$ and there exists a dyadic interval $I$ such that the restriction $f \mathbf{1}_I \in L^1(I)$, then
$$
\lim_{n\to\infty}\sigma_{2^n}f(x)=f(x) \qquad \mbox{for almost every $x\in I$ as well as in the $L^{\varphi}(I)$-norm.}
$$
\end{corollary}

\begin{corollary}\label{c53}
Let $\V\in A_{\fz} [0,1)$ be a Musielak--Orlicz function with uniformly lower type $p^{-}_{\V}$ and uniformly upper type $p^{+}_{\V}$ satisfying \eqref{doob}. Suppose that the Doob maximal operator $M$ is bounded on $L^{\varphi^{*}_{1/r}}[0,1)$ for some $r\in(0,\min\{1, p^{-}_{\V}\}]$. If $f\in L^{\varphi}[0,1)$, then
$$
\lim_{n\to\infty}\sigma_{2^n}f(x)=f(x) \qquad \mbox{for almost every $x\in [0,1)$ as well as in the $L^{\varphi} [0,1)$-norm.}
$$
\end{corollary}

\bigskip

\noindent  Guangheng Xie and Dachun Yang (Corresponding author)

\smallskip

\noindent  Laboratory of Mathematics and Complex Systems
(Ministry of Education of China),
School of Mathematical Sciences, Beijing Normal University,
Beijing 100875, People's Republic of China

\smallskip

\noindent {\it E-mails}: \texttt{guanghengxie@mail.bnu.edu.cn} (G. Xie)

\noindent\phantom{{\it E-mails}:} \texttt{dcyang@bnu.edu.cn} (D. Yang)

\medskip

\noindent Ferenc Weisz

\smallskip

\noindent  Department of Numerical Analysis, E\"otv\"os L. University,
H-1117 Budapest, P\'azm\'any P. s\'et\'any 1/C., Hungary

\smallskip

\noindent {\it E-mail}: \texttt{weisz@inf.elte.hu}

\medskip

\noindent Yong Jiao

\smallskip

\noindent  School of Mathematics and Statistics, Central South University,
Changsha 410075, People's Republic of China

\smallskip

\noindent {\it E-mail}: \texttt{jiaoyong@csu.edu.cn}

\end{document}